\documentclass{elsarticle}

\usepackage{amssymb}

\usepackage{amsthm}

\usepackage{amsmath}

\usepackage{mathabx}

\makeatletter
\def\ps@pprintTitle{%
 \let\@oddhead\@empty
 \let\@evenhead\@empty
 \def\@oddfoot{\centerline{\thepage}}%
 \let\@evenfoot\@oddfoot}
\makeatother

\def\FF{\mathbb F}

\def\rank{\mathop{\mathrm {rank}}\nolimits}
\def\tr{\mathop{\top}\nolimits}
\def\Tr{\mathop{\mathrm{Tr}}\nolimits}
\def\sgl{{\mathop{SGL}\nolimits}_n(\mathbb{F}_{2})}
\def\ronetzero{\mathcal{R}_1^{\Tr_0}}
\def\diam{\mathop{\mathrm {diam}}\nolimits}

\theoremstyle{definition}
\newtheorem{defi}{Definition}[section]
\newtheorem{exa}[defi]{Example}
\theoremstyle{plain}
\newtheorem{thm}[defi]{Theorem}

\newtheorem{lemma}[defi]{Lemma}
\newtheorem{cor}[defi]{Corollary}
\newtheorem{prop}[defi]{Proposition}
\theoremstyle{remark}
\newtheorem{remark}[defi]{Remark}

\newcounter{myenumi}

\newsavebox{\cmm}
\savebox{\cmm}{\indent}
\newenvironment{myenumerate}[1]{
\begin{list}{
{\bf #1~\themyenumi}. } {\labelwidth=0pt
\labelsep=0pt\leftmargin=0pt\usecounter{myenumi}} }{\end{list}}

\newcounter{myenumiB}

\newenvironment{myenumerateB}[1]{
\begin{list}{
{\emph #1~\themyenumiB}. } {\labelwidth=0pt
\labelsep=0pt\leftmargin=0pt\usecounter{myenumiB}} }{\end{list}}

\begin{document}

\begin{frontmatter}

\title{The distance function on Coxeter-like graphs and self-dual codes}

\author[labelFAMNIT,labelIAM,labelIMFM]{Marko Orel\corref{cor1}}
\ead{marko.orel@upr.si}
\author[labelIAM]{Dra\v{z}enka Vi\v{s}nji\'{c}}
\ead{drazenka.visnjic@iam.upr.si}

\address[labelFAMNIT]{University of Primorska, FAMNIT, Glagolja\v{s}ka 8, 6000 Koper, Slovenia}
\address[labelIAM]{University of Primorska, IAM, Muzejski trg 2, 6000 Koper, Slovenia}
\address[labelIMFM]{IMFM, Jadranska 19, 1000 Ljubljana, Slovenia}

\cortext[cor1]{Corresponding author}

\begin{abstract}
Let $\sgl$ be the set of all invertible $n\times n$ symmetric matrices over the binary field $\FF_2$.
Let $\Gamma_n$ be the graph with the vertex set $\sgl$ where a pair of matrices $\{A,B\}$ form an edge if and only if $\rank(A-B)=1$. In particular, $\Gamma_3$ is the well-known Coxeter graph. The distance function $d(A,B)$ in $\Gamma_n$ is described for all matrices $A,B\in\sgl$. The diameter of $\Gamma_n$ is computed. For odd $n\geq 3$, it is shown that each matrix $A\in\sgl$ such that $d(A,I)=\frac{n+5}{2}$ and $\rank(A-I)=\frac{n+1}{2}$ where $I$ is the identity matrix induces a self-dual code in $\FF_2^{n+1}$. Conversely, each self-dual code $C$ induces a family ${\cal F}_C$ of such matrices $A$. The families given by distinct self-dual codes are disjoint. The identification $C\leftrightarrow {\cal F}_C$ provides a graph theoretical description of self-dual codes. A result of Janusz (2007) is reproved and strengthened by showing that the orthogonal group ${\cal O}_n(\FF_2)$ acts transitively on the set of all self-dual codes in $\FF_2^{n+1}$.
\end{abstract}


\end{frontmatter}

\section{Introduction}

Vector spaces formed by matrices/bilinear forms over a (finite) field, equipped with the metric $(A,B)\mapsto \rank(A-B)$, are studied in multiple research areas. In algebraic combinatorics, they are investigated within association schemes and distance-regular graphs~\cite{bannai,BCN,ivanov1,schmidt1}. In coding theory, they appear in the context of rank-metric codes~\cite{bence,delsarte,gabidulin,ravagnani,sheekey,silva}. In matrix theory, they are explored in preserver problems \cite{PazzisSemrl,HuangPravokotne,HuangSemrl2,semrl,wan}. This research field mixes also with the study of graph homomorphisms and cores~\cite{FFA,JACO,NOVA}. Often, the full set of matrices is replaced by a subset of hermitian, alternate, or symmetric matrices \cite{pavese,gabidulin2,Lavrauw,HuangAlter,HuangSemrl,FFA,JACO}. If we consider the graph $\Gamma$ with the vertex set formed by all rectangular $m\times n$ matrices with coefficients from a finite field~$\FF_q$ where a pair $\{A,B\}$ of matrices form an edge if and only if $\rank(A-B)=1$, then it is well known and easy to see that the distance in this graph equals $d_{\Gamma}(A,B)=\rank(A-B)$ for all matrices $A,B$ (cf.~\cite[Proposition~3.5]{wan}). Similar claims are true for analogous graphs formed by hermitian, alternate, or symmetric matrices (cf.~\cite{wan}). In all these cases, the simplicity of providing a distance formula for $d_{\Gamma}(A,B)$ lies in the fact that the vertex sets are vector spaces. The distance function as well as many other properties of the corresponding subgraphs, which are induced by invertible matrices, are much more difficult to investigate because the vertex sets are not closed under the addition. The only properties/results about these subgraphs we are aware of are contained in the papers \cite{LAA2016I,LAA2016II,E-JC} and in the survey~\cite[Examples~3.11-3.18]{NOVA}.

In this paper, we focus on graph $\Gamma_n$ where the vertex set is $\sgl$, i.e. the set of all $n\times n$ invertible binary symmetric matrices, and where $\{A,B\}$ is an edge if and only if $\rank(A-B)=1$. There are multiple reasons to focus on graph $\Gamma_n$. As observed by the first author in~\cite{E-JC}, it generalizes the well-known Coxeter graph~\cite{coxeter}, which is obtained if $n=3$ (see Figure~\ref{slika}).
\begin{figure}\label{slika}
\centering
\includegraphics[width=\textwidth]{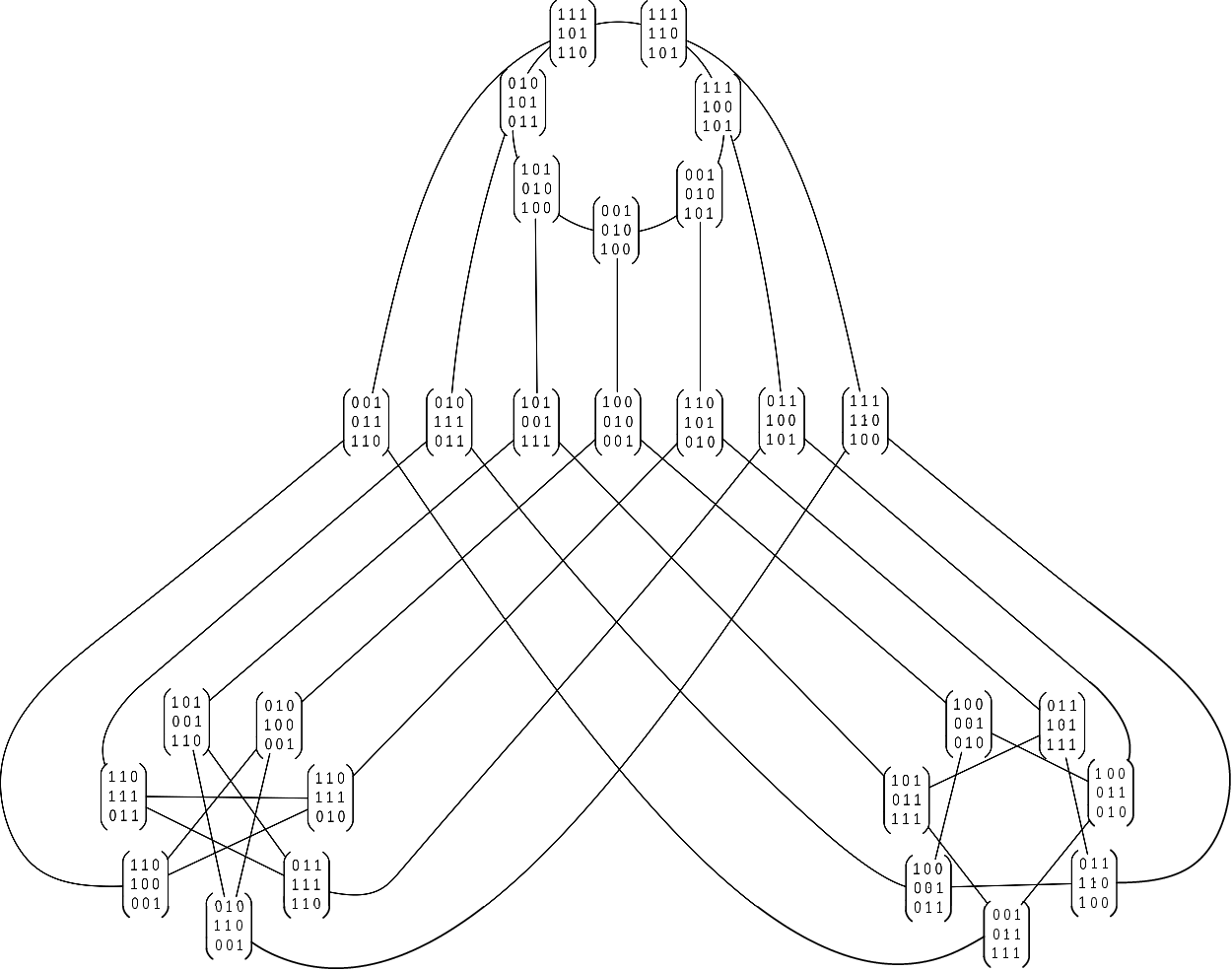}
\caption{Graph $\Gamma_3$ is the Coxeter graph.}
\end{figure}
Similarly, the graph obtained by the set $HGL_n(\FF_4)$ of all $n\times n$ invertible hermitian matrices over $\FF_4$ generalizes the well-known Petersen graph. These two graphs are very important as they are among the five (currently) known vertex-transitive graphs that do not have a Hamiltonian cycle. Hamiltonicity of vertex-transitive graphs is an active research area with a long tradition (see for example~\cite{babai-handbook,bermond,DMM,draganklavdija2009,lovasz-problem,marusic1983b} and the references therein). In~\cite{E-JC}, it was posed as an open problem if the graph~$\Gamma_n$ (and the corresponding graph formed by $HGL_n(\FF_4)$) has a Hamiltonian cycle. Here, we do not solve this problem. However, all of the above indicate that the infinite family of graphs $\Gamma_n$ $(n\in\mathbb{N})$ should be studied extensively. Some properties of graphs $\Gamma_n$ were obtained already in~\cite{E-JC}. In particular, it was shown that these graphs are cores for $n\geq 3$, that is, all their endomorphisms are automorphisms. In this paper, we describe the distance function $d(A,B):=d_{\Gamma_n}(A,B)$ in $\Gamma_n$ and hope to report about a similar result on $HGL_n(\FF_4)$ in the future. We also need to mention that it is expected that the computation of the diameter of $\Gamma_n$ (Corollary~\ref{diameter}) will help us to describe all automorphisms of $\Gamma_n$ as well as some related homomorphisms~\cite{inPreparation}.

To understand better the difficulty of the computation of the distance function in the context of invertible matrices, and to provide an additional reason why to focus on the binary field $\FF_2$, consider the graph $\widehat{\Gamma}_n$ with the vertex set formed by the set $S_n(\FF_2)$ of all $n\times n$ binary symmetric matrices where edges are defined in the same way as for the graph $\Gamma_n$. It is well known that the distance in $\widehat{\Gamma}_n$ is given by $d_{\widehat{\Gamma}_n}(A,B)=\rank(A-B)$ if $A-B$ is nonalternate or zero, and $d_{\widehat{\Gamma}_n}(A,B)=\rank(A-B)+1$ if $A-B$ is a nonzero alternate matrix (see~\cite[Proposition~5.5]{wan}). In fact, if $A-B$ is nonalternate of rank $r$, then $B=A+\sum_{i=1}^r {\bf x}_i{\bf x}_i^{\tr}$ for some linearly independent column vectors ${\bf x}_1,\ldots, {\bf x}_r$. Consequently, if $C_j=A+\sum_{i=1}^j {\bf x}_i{\bf x}_i^{\tr}$ for $j\in\{0,1,\ldots,r\}$, then $A=C_0,C_1,\ldots,C_r=B$ is a path of length $r$ between $A$ and $B$. Similarly we can find a path of length $r+1$ if $A-B$ is alternate of rank $r>0$ (see Lemma~\ref{alternate-canonical}). It is easy to see that there are no shorter paths of this kind. Now, if we consider the case $A,B\in \sgl$, then it is certainly possible that some of the matrices $C_j$ are not invertible. Hence, some distances in graph $\Gamma_n$, which is induced by the set $\sgl$, could be larger than $d_{\widehat{\Gamma}_n}(A,B)$. This seems quite likely (and turns out to be true) because the proportion of the invertible symmetric matrices $\frac{|\sgl|}{|S_n(\FF_2)|}$ is quite low (see Table~\ref{tabela}; value $|\sgl|$ can be found in~\cite[Lemma~9.5.9]{BCN} or~\cite{E-JC}).
\begin{table}\label{tabela}
\begin{center}
\renewcommand{\arraystretch}{1.4}
\begin{tabular}{|c|c|c|c|c|c|c|c|}
  \hline
  $n$ & 2& 3&  4 &5 & 6 & 7\\
  \hline
 $\frac{|\sgl|}{|S_n(\FF_2)|}\doteq $ & 0.5& 0.4375 &  0.4238& 0.4205 & 0.4197 & 0.4195\\
  \hline
 \end{tabular}
 \caption{$\lim_{n\to\infty}\frac{|\sgl|}{|S_n(\FF_2)|}=\prod_{i=1}^{\infty}\left(1-\frac{1}{2^{2i-1}}\right)=0.4194224417951075\ldots$}
\end{center}
\end{table}
This fact provides us another reason to focus on the binary field $\FF_2$ because the corresponding proportions over larger finite fields $\FF_q$ are much higher and converge to 1 as $q\to\infty$.

Lastly, we mention that the investigation of $\Gamma_n$ represent a novel graph theoretical technique to study linear codes, self-dual codes in particular. In fact, we show that for odd $n$ a certain subset of matrices in $\sgl$, which is defined by the distance and the rank function in $\Gamma_n$, correspond to self-dual codes in $\FF_2^{n+1}$ and vice versa. Moreover, the linear codes that are self-dual are completely determined by two graph parameters, namely $d(A,I)$ and $d_{\widehat{\Gamma}_n}(A,I)$ where $I$ is the identity matrix and $A\in \sgl$ is a matrix associated to the code. A well-known and intractable open problem in coding theory is to determine the number of permutation inequivalent self-dual codes in $\FF_2^{n+1}$. Its value for small $n$ and its asymptotical behavior are determined in \cite[Table~V]{conway1}\footnote{The value for $n+1=30$ is corrected in~\cite{conway2}.} and~\cite{hou}, respectively.  In~\cite[Theorem~10]{janusz}, Janusz proved that all self-dual codes are `orthogonally equivalent'. As a corollary, he was able to show that the enumeration of permutation inequivalent self-dual codes in $\FF_2^{n+1}$ is equivalent to the enumeration of certain double cosets in the orthogonal group~${\cal O}_{n+1}(\FF_2)$. Our technique enable us to improve~\cite[Theorem~10]{janusz} by showing that already the group ${\cal O}_n(\FF_2)$ acts transitively on the set of all self-dual codes in $\FF_2^{n+1}$.

The rest of the paper is organized as follows. In Section~\ref{s2}, we describe the necessary notation and state the main results of this paper, i.e. the computation of the distance function $d(A,B)$ (Theorems~\ref{thm-nonalter} and~\ref{thm-alter}). Sections~\ref{s3} and~\ref{s4} include auxiliary lemmas that are needed in the proofs. The results in Section~\ref{s3} could be interesting on their own from perspective of linear algebra. Sections~\ref{s5} and~\ref{s6} contain the proofs of Theorems~\ref{thm-nonalter} and~\ref{thm-alter}, respectively. In Section~\ref{s7}, we compute the diameter of graph $\Gamma_n$ (Corollary~\ref{diameter}). The connection of graph $\Gamma_n$ with self-dual codes is described in Section~\ref{s8} (Theorems~\ref{particija}, \ref{thm-distances}, \ref{thm-delovanjeinverza}, \ref{thm-code}).

\section{Notation and the statements of the main results}
\label{s2}

Throughout the paper, $\FF$ is a field and $\FF_2=\{0,1\}$ is the binary field. Given positive integers $n$ and $m$, $M_{n\times m}(\FF)$ denotes the set of all $n\times m$ matrices with coefficients from the field $\FF$. Similarly, $GL_n(\FF)$ denotes the subset in $M_{n\times n}(\FF)$, which is formed by all invertible matrices. We use $I=I_n$ to denote the identity matrix in $GL_n(\FF)$, while $J=J_{n\times m}$ and $O=O_{n\times m}$ are the matrices in $M_{n\times m}(\FF)$ with all coefficients equal to $1$ and $0$, respectively. The determinant, the rank, the transpose, and the trace of a matrix $A$ are denoted by $\det A$, $\rank A$, $A^{\tr}$, and $\Tr A$, respectively. The elements of $\FF^n:=M_{n\times 1}(\FF)$, i.e. the column vectors, are written in bold style, like ${\bf x}$. In particular, ${\bf j}_n=J_{n\times 1}$, while ${\bf e}_i=(0,\ldots,0,1,0,\ldots,0)^{\tr}\in \FF^n$ where $1$ appears as the $i$-th entry. The linear span of column vectors ${\bf x}_1,\ldots,{\bf x}_r\in \FF^n$ is denoted by $\langle {\bf x}_1,\ldots, {\bf x}_r\rangle$. The subset $\{A\in M_{n\times n}(\FF): A^{\tr}=A\}$ of all symmetric matrices is denoted by $S_n(\FF)$, while $SGL_n(\FF):=S_n(\FF)\cap GL_n(\FF)$ and $\ronetzero:=\{A\in S_n(\FF):\ \rank A=1,\ \Tr A=0\}$. Given a binary column vector ${\bf x}\in \FF_2^n$, let ${\bf x}^2:={\bf x}{\bf x}^{\tr}\in S_n(\FF_2)$ denote the corresponding rank-one matrix. Similarly, for ${\bf x}, {\bf y}\in \FF_2^n$, let ${\bf x}\circ {\bf y}:={\bf x}{\bf y}^{\tr}+{\bf y}{\bf x}^{\tr}\in S_n(\FF_2)$. Observe that ${\bf x}\circ {\bf y}={\bf y}\circ {\bf x}$, ${\bf x}\circ ({\bf y}+{\bf z})={\bf x}\circ {\bf y}+{\bf x}\circ {\bf z}$, and $({\bf x}+{\bf y})^2={\bf x}^2+{\bf y}^2+{\bf x}\circ {\bf y}$ for all ${\bf x},{\bf y},{\bf z}\in \FF_2^n$. Recall that a matrix $A=[a_{ij}]_{i,j\in \{1,\ldots,n\}}\in M_{n\times n}(\FF_2)$ is \emph{alternate} if both $A^{\tr}=A$ and $a_{ii}=0$ for all~$i$, or equivalently ${\bf x}^{\tr}A{\bf x}=0$ for all ${\bf x}\in \FF_2^n$. It is well known that the rank $r$ of an alternate matrix $A\in S_n(\FF_2)$ is even, and $A={\bf y}_1\circ {\bf y}_2+\cdots+{\bf y}_{r-1}\circ {\bf y}_{r}$ for some linearly independent vectors ${\bf y}_1,\ldots, {\bf y}_r\in\FF_2^n$ (cf.~\cite[Proposition~1.34]{wan}). On the other hand, each nonalternate matrix $A\in S_n(\FF_2)$ of rank $r$ can be written as $A=\sum_{i=1}^r {\bf x}_i^{2}$ for some linearly independent ${\bf x}_1,\ldots, {\bf x}_r\in \FF_2^n$ (cf.~\cite[Corollary~1.36]{wan}). In particular, each nonalternate $A\in \sgl$ can be written as $A=PP^{\tr}$ where matrix $P\in GL_n(\FF_2)$ has ${\bf x}_i$ as the $i$-th column.

Recall that a vector subspace $C\subseteq \FF_2^n$ is a (binary) \emph{linear code} of \emph{length} $n$ and \emph{dimension} $\dim C$. It is \emph{self-orthogonal} if $C\subseteq C^{\bot}:=\{{\bf x}\in \FF_2^n : {\bf x}^{\tr}{\bf c}=0\ \textrm{for all}\ {\bf c}\in C\}$ (cf.~\cite{plessknjiga}). In this case, $\dim C\leq \lfloor\frac{n}{2}\rfloor$. If $C=C^{\bot}$, then the code is \emph{self-dual}, $n$ is even, and $\dim C =\frac{n}{2}$.

We use $d(A,B)$ to denote the distance between vertices $A,B\in \sgl$ in graph $\Gamma_n$. We write $A\sim B$ if $d(A,B)=1$. Since $\Gamma_n$ is an induced subgraph in $\widehat{\Gamma}_n$, it follows that $d(A,B)\geq d_{\widehat{\Gamma}_n}(A,B)$ for all $A,B\in\sgl$, that is,
\begin{equation}\label{eq3}
d(A,B)\geq \left\{\begin{array}{ll}\rank(A-B),&\textrm{if}\ A-B\ \textrm{is nonalternate or zero};\\
\rank(A-B)+1,&\textrm{if}\ A-B\ \textrm{is alternate and nonzero}.
\end{array}\right.
\end{equation}
Theorem~\ref{thm-nonalter} describes $d(A,B)$ for all $A,B\in\sgl$ where $A-B$ is an arbitrary nonalternate matrix of rank $r$.
\begin{thm}\label{thm-nonalter}
Let $A,B\in \sgl$ be such that $B-A=\sum_{i=1}^r {\bf x}_i^2$ for some linearly independent column vectors ${\bf x}_1,\ldots,{\bf x}_r\in\FF_2^n$.
\begin{enumerate}
\item\label{nonalter_i} If ${\bf x}_i^{\tr}A^{-1}{\bf x}_i=1$ for all $i$, then
$$d(A,B)=\left\{\begin{array}{ll}r+2,& \textrm{if}\ {\bf x}_i^{\tr}A^{-1}{\bf x}_j=1\ \textrm{for all}\ i,j;\\
r+1, & \textrm{otherwise}.\end{array}\right.$$
\item\label{nonalter_ii} If $[{\bf x}_i^{\tr}A^{-1}{\bf x}_j]_{i,j=1}^r\in \ronetzero$, then $d(A,B)=r+2$.
\item\label{nonalter_iii} Otherwise, $d(A,B)=r$.
\end{enumerate}
\end{thm}
\begin{remark}\label{opomba2}
Observe that cases (i), (ii), (iii) are all distinct, except if $r$ is even and ${\bf x}_i^{\tr}A^{-1}{\bf x}_j=1$ for all $i,j$. In this case $A,B$ fit the assumptions in (i) and~(ii).
\end{remark}

Theorem~\ref{thm-alter} describes $d(A,B)$ for all $A,B\in\sgl$ where $A-B$ is an arbitrary alternate matrix of rank $r>0$ (see Lemma~\ref{alternate-canonical}).
\begin{thm}\label{thm-alter}
Let $A,B\in \sgl$ be such that $B-A=\sum_{i=1}^{r+1} {\bf x}_i^2$ for some linearly independent ${\bf x}_1,\ldots,{\bf x}_r\in\FF_2^n$ where $r>0$ is even and ${\bf x}_{r+1}={\sum}_{i=1}^r {\bf x}_i$.
\begin{enumerate}
\item\label{alter_i} If $[{\bf x}_i^{\tr} A^{-1}{\bf x}_j]_{i,j=1}^{r+1}$ is not of rank one, then $d(A,B)=r+1$.
\item\label{alter_ii} If $[{\bf x}_i^{\tr} A^{-1}{\bf x}_j]_{i,j=1}^{r+1}$ is of rank one, then $d(A,B)=r+2$.
\end{enumerate}
\end{thm}

\begin{remark}\label{opomba1}
Observe that ${\bf x}_{r+1}={\sum}_{i=1}^r {\bf x}_i$ implies that $\Tr([{\bf x}_i^{\tr} A^{-1}{\bf x}_j]_{i,j=1}^{r+1})=0$ and $\rank([{\bf x}_i^{\tr} A^{-1}{\bf x}_j]_{i,j=1}^{r+1})=\rank([{\bf x}_i^{\tr} A^{-1}{\bf x}_j]_{i,j=1}^{r})$.
\end{remark}

\begin{remark}
In Theorems~\ref{thm-nonalter} and \ref{thm-alter}, $B=A+XX^{\tr}$ where $X$ is the $n\times r$ matrix or the $n\times (r+1)$ matrix with ${\bf x}_i$ as the $i$-th column. By Lemma~\ref{l2},
$$X^{\tr}B^{-1}X=XA^{-1}X-X^{\tr}A^{-1}X(I+X^{\tr}A^{-1}X)^{-1}X^{\tr}A^{-1}X.$$
Therefore, if $[{\bf x}_i^{\tr}A^{-1}{\bf x}_j]_{i,j}=X^{\tr}A^{-1}X\in \ronetzero$,
then Proposition~\ref{prop0} implies that $X^{\tr}B^{-1}X=XA^{-1}X$. In particular, $[{\bf x}_i^{\tr}B^{-1}{\bf x}_j]_{i,j}\in \ronetzero$. Likewise, if in Theorem~\ref{thm-nonalter}, ${\bf x}_i^{\tr}A^{-1}{\bf x}_i=1$ for all $i$, then ${\bf x}_i^{\tr}B^{-1}{\bf x}_i=1$ for all $i$ by Corollary~\ref{c1}. Hence, matrices $A$ and $B$ really occur symmetrically in Theorems~\ref{thm-nonalter} and~\ref{thm-alter} (observe that, by Lemma~\ref{lemma3}, the case ${\bf x}_i^{\tr}A^{-1}{\bf x}_j=1$ for all $i,j$
in Theorem~\ref{thm-nonalter}~\eqref{nonalter_i} is equivalent to $[{\bf x}_i^{\tr}A^{-1}{\bf x}_j]_{i,j}\in \ronetzero$ and ${\bf x}_i^{\tr}A^{-1}{\bf x}_i=1$ for all $i$).
\end{remark}

\section{Auxiliary results from linear algebra}
\label{s3}

Lemma~\ref{schur} can be found in~\cite{handbook}. Lemma~\ref{l2} appears in~\cite[Corollary~2.9]{NOVA} or~\cite{zhang} where it is stated in a slightly different form (see also~\cite{handbook}).
\begin{lemma}\label{schur}
Let $\FF$ be any field. If $A\in M_{n\times n}(\FF)$ is partitioned in blocks as
$$A=\left(\begin{array}{cc}A_{11}&A_{12}\\
A_{21}&A_{22}\end{array}\right)$$
where $A_{22}$ is invertible, then $\det A=\det A_{22}\cdot \det(A_{11}-A_{12}(A_{22})^{-1}A_{21})$.
\end{lemma}

\begin{lemma}\label{l2}
Let $\FF$ be any field, $A\in GL_n(\FF)$, and $X,Y\in M_{n\times r}(\FF)$. Then,
$$\det(A+XY^{\tr})=\det(A)\cdot \det(I_r+Y^{\tr}A^{-1}X).$$
If $A+XY^{\tr}\in GL_n(\FF)$, then
$$(A+XY^{\tr})^{-1}=A^{-1}-A^{-1}X(I_r+Y^{\tr}A^{-1}X)^{-1} Y^{\tr}A^{-1}.$$
\end{lemma}

\begin{cor}\label{c1}
Let $\FF$ be any field of characteristic two. Suppose ${\bf x}_1,\ldots,{\bf x}_r\in \FF^n$,
and assume that $A, B:=A+\sum_{i=1}^r {\bf x}_i^2\in SGL_n(\FF)$. Then ${\bf x}_i^{\tr} A^{-1} {\bf x}_i=1$ for all $i$ if and only if  ${\bf x}_i^{\tr} B^{-1} {\bf x}_i=1$ for all $i$.
\end{cor}
\begin{proof}
Let $X$ be the $n\times r$ matrix with ${\bf x}_i$ as its $i$-th column. Since $B=A+XX^{\tr}$,
$$B^{-1}=A^{-1}-A^{-1}X(I_r+X^{\tr}A^{-1}X)^{-1} X^{\tr}A^{-1}$$
by Lemma~\ref{l2}. If ${\bf x}_i^{\tr} A^{-1}{\bf x}_i=1$ for all $i$, then $I_r+X^{\tr}A^{-1}X$ is an alternate matrix. Hence, $C:=A^{-1}X(I_r+X^{\tr}A^{-1}X)^{-1} X^{\tr}A^{-1}$ is also alternate and ${\bf x}_i^{\tr}B^{-1} {\bf x}_i={\bf x}_i^{\tr}A^{-1} {\bf x}_i-{\bf x}_i^{\tr}C {\bf x}_i=1+0=1$. The converse is symmetric.
\end{proof}

\begin{lemma}\label{lemma3}
Let $r$ be an odd integer such that $1\leq r\leq n$. If ${\bf x}_1,\ldots,{\bf x}_r\in\FF_2^n$ and $A,A+\sum_{i=1}^r {\bf x}_i^2\in \sgl$, then there exists $i$ such that ${\bf x}_i^{\tr} A^{-1} {\bf x}_i=0$.
\end{lemma}
\begin{proof}
Let ${\bf x}_i$ be the $i$-th column of $X\in M_{n\times r}(\FF_2)$. By Lemma~\ref{l2},
$$1=\det\left(A+\sum_{i=1}^r {\bf x}_i^2\right)=\det(A+XX^{\tr})=\det(I_r+X^{\tr}A^{-1}X).$$
Since $r$ is odd, and alternate matrices have even rank, the symmetric matrix $I_r+X^{\tr}A^{-1}X$ is not alternate, i.e. some of its diagonal entries equals one.
\end{proof}

\begin{lemma}\label{lemica}
Let $a\in \FF_2$ be 1 if and only if the positive integer $m$ is odd. Then,
\begin{equation}\label{eq2}
\left(\begin{array}{ccc}
a&1+a&{\bf j}_m^{\tr}\\
1+a&1+a&{\bf j}_m^{\tr}\\
{\bf j}_m&{\bf j}_m&I_m
\end{array}\right)\in SGL_{m+2}(\FF_2).
\end{equation}
\end{lemma}
\begin{proof}
Let $A$ be the $2\times 2$ top-left corner of the matrix in \eqref{eq2}. By Lemma~\ref{schur},
\begin{align*}
\det\left(\begin{array}{cc}A&J_{2\times m}\\
J_{2\times m}^{\tr}&I_m\end{array}\right)&=\det(A-J_{2\times m}J_{2\times m}^{\tr})\\
&=\det(A-aJ_{2\times 2})
=\det\left(\begin{array}{cc}
0&1\\
1&1\end{array}\right)=1.\tag*{\qedhere}
\end{align*}
\end{proof}

\begin{prop}\label{prop0}
Let $A\in S_n(\FF_2)$. Then $A\in \ronetzero$ if and only if there exist a permutation matrix $Q\in GL_n(\FF_2)$ and even $k\in \{2,\ldots,n\}$
such that
\begin{equation}\label{eq8}
A=Q\left(\begin{array}{cc}J_{k\times k}& O\\O&O\end{array}\right)Q^{\tr}
\end{equation}
In this case, $I+A\in \sgl$ and $(I+A)^{-1}=I+A$.
\end{prop}
\begin{proof}
Since $Q^{\tr}=Q^{-1}$ and $\Tr$ is invariant under the similarity operation, it is obvious that matrix \eqref{eq8} is in $\ronetzero$.
Moreover, since $k$ is even, we deduce that
$$(I+A)^2=Q\left(\begin{array}{cc}(J_{k\times k}+I_k)^2& O\\O&I_{n-k}^{2}\end{array}\right)Q^{\tr}=I.$$ Hence, $I+A$ is invertible with $(I+A)^{-1}=I+A$.

Suppose now that $A=[a_{ij}]_{i,j=1}^n\in \ronetzero$ is arbitrary. Let $S_1=\{i\in \{1,\ldots,n\} : a_{ii}=1\}$ and $k=|S_1|$. Since $A$ has zero trace, it follows that $k$ is even. Being of rank one, the matrix $A$ is nonalternate and therefore $k>0$. Pick a permutation matrix $Q$ such that matrix $[b_{ij}]_{i,j=1}^n=Q^{\tr}AQ$ satisfies $\{i\in \{1,\ldots,n\} : b_{ii}=1\}=\{1,\ldots,k\}$. Since matrix $[b_{ij}]_{i,j=1}^n$ is symmetric and of rank one, it follows that $b_{ij}=b_{ij}^2=b_{ij}b_{ji}=b_{ii}b_{jj}$, i.e.
$$[b_{ij}]_{i,j=1}^n=\left(\begin{array}{cc}J_{k\times k}& O\\O&O\end{array}\right).\qedhere$$
\end{proof}

\begin{lemma}\label{alternate-canonical}
If $A\in S_n(\FF_2)$ is an alternate matrix of rank $r>0$, then there exist linearly independent column vectors ${\bf x}_1,\ldots, {\bf x}_r\in \FF_2^n$ such that
$$A={\bf x}_1^2+\cdots +{\bf x}_r^2+({\bf x}_1+\cdots+{\bf x}_r)^2.$$
Moreover, if $A={\bf y}_1\circ {\bf y}_2+\cdots+{\bf y}_{r-1}\circ {\bf y}_{r}$ for some linearly independent vectors ${\bf y}_1,\ldots, {\bf y}_r$ where $r$ is even, then the vectors
\begin{align*}
{\bf x}_1&:={\bf y}_1,\\
{\bf x}_2&:={\bf y}_2,\\
{\bf x}_{2k-1}&:={\bf y}_1+\cdots +{\bf y}_{2k-2}+{\bf y}_{2k-1}, \qquad &(2\leq k\leq r/2),\\
{\bf x}_{2k}&:={\bf y}_1+\cdots +{\bf y}_{2k-2}+{\bf y}_{2k}, \qquad &(2\leq k\leq r/2).
\end{align*}
have the desired property.
\end{lemma}
\begin{proof}
Observe that ${\bf x}_{2k-1}+{\bf x}_{2k}={\bf y}_{2k-1}+{\bf y}_{2k}$ for each $k$, and consequently
\begin{align*}
{\bf y}_{2k-1}&={\bf x}_1+\cdots +{\bf x}_{2k-2}+{\bf x}_{2k-1}, \qquad &(2\leq k\leq r/2),\\
{\bf y}_{2k}&={\bf x}_1+\cdots +{\bf x}_{2k-2}+{\bf x}_{2k}, \qquad &(2\leq k\leq r/2).
\end{align*}
We prove the equality
$${\bf y}_1\circ {\bf y}_2+\cdots+{\bf y}_{r-1}\circ {\bf y}_{r}={\bf x}_1^2+\cdots +{\bf x}_r^2+({\bf x}_1+\cdots+{\bf x}_r)^2$$
by induction on $r$. If $r=2$, the claim is obvious. Let $r\geq 4$ and
assume that $${\bf y}_1\circ {\bf y}_2+\cdots+{\bf y}_{r-3}\circ {\bf y}_{r-2}={\bf x}_1^2+\cdots +{\bf x}_{r-2}^2+({\bf x}_1+\cdots+{\bf x}_{r-2})^2.$$
Then,
\begin{align*}
A&={\bf x}_1^2+\cdots +{\bf x}_{r-2}^2+({\bf x}_1+\cdots+{\bf x}_{r-2})^2+{\bf y}_{r-1}\circ {\bf y}_{r}\\
&={\bf x}_1^2+\cdots +{\bf x}_{r-2}^2+({\bf x}_1+\cdots+{\bf x}_{r-2})^2\\
&+\big(({\bf x}_1+\cdots+{\bf x}_{r-2})+{\bf x}_{r-1}\big)\circ \big(({\bf x}_1+\cdots+{\bf x}_{r-2})+{\bf x}_{r}\big)\\
&={\bf x}_1^2+\cdots +{\bf x}_{r-2}^2+({\bf x}_1+\cdots+{\bf x}_{r-2})^2\\
&+0+{\bf x}_{r-1}\circ {\bf x}_{r}+({\bf x}_1+\cdots+{\bf x}_{r-2})\circ({\bf x}_{r-1}+{\bf x}_{r})\\
&={\bf x}_1^2+\cdots +{\bf x}_{r-2}^2+({\bf x}_1+\cdots+{\bf x}_{r})^2+({\bf x}_{r-1}+{\bf x}_{r})^2+{\bf x}_{r-1}\circ {\bf x}_{r}\\
&={\bf x}_1^2+\cdots +{\bf x}_{r}^2+({\bf x}_1+\cdots+{\bf x}_{r})^2.\tag*{\qedhere}
\end{align*}
\end{proof}

\begin{lemma}\label{l1}
Let $1\leq s<r\leq n$ where $s$ is an odd integer, and let ${\bf x}_1,\ldots,{\bf x}_r\in\FF_2^n$. Suppose $\{i_1,\ldots,i_s\}\subseteq \{1,\ldots,r\}$ is a subset with $s$ elements and complement $\{i_{s+1},\ldots,i_{r}\}=\{1,\ldots,r\}\backslash\{i_1,\ldots,i_s\}$. If
\begin{align*}
{\bf w}_j&={\bf x}_{i_j}+ {\bf x}_{i_{s+1}}+ {\bf x}_{i_{s+2}}+\cdots + {\bf x}_{i_{r}}\qquad &(j=1,\ldots,s),\\
{\bf w}_k&={\bf x}_{i_k}  &(k=s+1,\ldots,r),
\end{align*}
then
$${\bf x}_1^2+\cdots +{\bf x}_r^2+({\bf x}_1+\cdots+{\bf x}_r)^2+({\bf x}_{i_1}+\cdots+{\bf x}_{i_s})^2={\bf w}_1^2+\cdots +{\bf w}_r^2.$$
Moreover, for all $C\in S_n(\FF_2)$ we have
\begin{equation}\label{eq35}
\rank\big( [{\bf w}_{i}^{\tr}C{\bf w}_{j}]_{i,j\in\{1,\ldots,r\}}\big)=\rank\big( [{\bf x}_{i}^{\tr}C{\bf x}_{j}]_{i,j\in\{1,\ldots,r\}}\big)
\end{equation}
and
$$\Tr [{\bf w}_{i}^{\tr}C{\bf w}_{j}]_{i,j\in\{1,\ldots,r\}}=\sum_{j=1}^s {\bf x}_{i_j}^{\tr}C{\bf x}_{i_j}.$$
\end{lemma}
\begin{proof}
By applying the equality $({\bf a} + {\bf b})^2={\bf a}^2+{\bf b}^2+{\bf a}\circ {\bf b}$ where ${\bf b}={\bf x}_{i_{s+1}}+\cdots+{\bf x}_{i_r}$ several times, we deduce that
\begin{align*}
{\bf w}_1^2+\cdots +{\bf w}_r^2&=\sum_{j=1}^r {\bf x}_{i_j}^2 + ({\bf x}_{i_{s+1}}+\cdots+{\bf x}_{i_r})^2\cdot s\\
&\qquad\qquad\ \! + ({\bf x}_{i_1}+\cdots+{\bf x}_{i_s})\circ ({\bf x}_{i_{s+1}}+\cdots+{\bf x}_{i_r})\\
&=\sum_{j=1}^r {\bf x}_{i_{j}}^2 + ({\bf x}_{i_1}+\cdots+{\bf x}_{i_r})^2+ ({\bf x}_{i_1}+\cdots+{\bf x}_{i_s})^2\\
&=\sum_{j=1}^r {\bf x}_{j}^2 + ({\bf x}_{1}+\cdots+{\bf x}_{r})^2+ ({\bf x}_{i_1}+\cdots+{\bf x}_{i_s})^2.
\end{align*}
Let $X$ and $W$ be the $n\times r$ matrices with ${\bf x}_i$ and ${\bf w}_i$ as the $i$-th column, respectively. Then, $W=XR$ for some $R \in GL_r(\FF _2)$. Therefore,
$$[{\bf w}_{i}^{\tr}C{\bf w}_{j}]_{i,j=1}^r=W^{\tr} CW=R^{\tr}X^{\tr} C XR= R^{\tr}[{\bf x}_{i}^{\tr}C{\bf x}_{j}]_{i,j=1}^r R$$
and $\rank\big( [{\bf w}_{i}^{\tr}C{\bf w}_{j}]_{i,j=1}^r \big)=\rank\big( [{\bf x}_{i}^{\tr}C{\bf x}_{j}]_{i,j=1}^r \big)$. Since $s$ is odd, we deduce
\begin{align*}
\Tr [{\bf w}_{i}^{\tr}C{\bf w}_{j}]_{i,j=1}^r&= \sum_{i=1}^r {\bf w}_{i}^{\tr}C{\bf w}_{i}\\
&=\sum_{j=1}^s({\bf x}_{i_j} + {\bf x}_{i_{s+1}} + \cdots + {\bf x}_{i_r})^{\tr}C ({\bf x}_{i_j} + {\bf x}_{i_{s+1}}+ \cdots + {\bf x}_{i_r})\\
&+ \sum_{j=s+1} ^r {\bf x}_{i_j}^{\tr}C {\bf x}_{i_j}\\
&= \sum_{j=1}^s {\bf x}_{i_j}^{\tr} C{\bf x}_{i_j} + s \cdot \sum_{j=s+1}^r {\bf x}_{i_j}^{\tr} C {\bf x}_{i_j}+ \sum_{j=s+1}^{r} {\bf x}_{i_j}^{\tr} C{\bf x}_{i_j}\\
&=\sum_{j=1}^s{\bf x}_{i_j}^{\tr} C {\bf x}_{i_j}.\tag*{\qedhere}
\end{align*}
\end{proof}

\begin{lemma}\label{l3}
Let $1\leq s<r\leq n$ where $s$ and $r$ are even integers, and let ${\bf x}_1,\ldots,{\bf x}_r\in\FF_2^n$. Suppose $\{i_1,\ldots,i_s\}\subseteq \{1,\ldots,r\}$ is a subset with $s$ elements and $\{i_{s+1},\ldots,i_{r}\}=\{1,\ldots,r\}\backslash\{i_1,\ldots,i_s\}$ is its complement. If
\begin{align*}
{\bf w}_j&={\bf x}_{i_j}+ {\bf x}_{i_{s+1}}+ {\bf x}_{i_{s+2}}+\cdots + {\bf x}_{i_{r}}\qquad &(j=1,\ldots,s-1),\\
{\bf w}_s&={\bf x}_{i_s},\\
{\bf w}_k&={\bf x}_{i_s}+{\bf x}_{i_k}  &(k=s+1,\ldots,r),
\end{align*}
then the same claim as in Lemma~\ref{l1} is true.
\end{lemma}
\begin{proof}
By applying the equality $({\bf a} + {\bf b})^2={\bf a}^2+{\bf b}^2+{\bf a}\circ {\bf b}$ several times, and recalling that $s,r$ are both even, we deduce that
\begin{align*}
\sum_{j=1}^r {\bf w}_j^2&= \sum_{j=1}^r {\bf x}_{i_j}^2 + (s-1)\cdot({\bf x}_{i_{s+1}}  + \cdots+{\bf x}_{i_r})^2\\
&+   \sum _{j=1}^{s-1}{\bf x}_{i_{j}}\circ ({\bf x}_{i_{s+1}}+\cdots+{\bf x}_{i_r}) + (r-s)\cdot {\bf x}_{i_{s}}^2+   \sum_{j=s+1}^r{\bf x}_{i_{j}} \circ {\bf x}_{i_{s}}\\
&=  \sum_{j=1}^r{\bf x}_{i_{j}}^2 + ({\bf x}_{i_{s+1}} + \cdots + {\bf x}_{i_r})^2 \\
&+ ({\bf x}_{i_{1}} + \cdots + {\bf x}_{x_{i_{s-1}}} ) \circ ({ \bf x}_{{i_{s+1}}} + \cdots + {\bf x}_{i_r} ) + {\bf x}_{i_s} \circ ({\bf x}_{i_{s+1}}+ \cdots + {\bf x}_{i_{r}})\\
&=\sum_{j=1}^r {\bf {x}}_{i_{j}}^2 + ({\bf x}_{i_{s+1}} + \cdots + {\bf x}_{i_r})^2+ ({\bf x}_{i_1}+\cdots+{\bf x}_{i_s})\circ ({\bf x}_{i_{s+1}}+\cdots+{\bf x}_{i_r})\\
&=\sum_{j=1}^r {\bf x}_{i_{j}}^2 + ({\bf x}_{i_1} + \cdots + {\bf x}_{i_r})^2+ ({\bf x}_{i_1}+\cdots+{\bf x}_{i_s})^2\\
&=\sum_{j=1}^r {\bf x}_{j}^2 + ({\bf x}_{1} + \cdots + {\bf x}_{r})^2+ ({\bf x}_{i_1}+\cdots+{\bf x}_{i_s})^2
\end{align*}
as claimed. We prove~\eqref{eq35} in the same way as in the proof of Lemma~\ref{l1}. Finally,
\begin{align*}
\Tr [{\bf w}_{i}^{\tr}C{\bf w}_{j}]_{i,j=1}^r&=\sum_{i=1}^r {\bf w}_{i}^{\tr}C{\bf w}_{i}\\
&=\sum_{j=1}^{s-1}({\bf x}_{i_j} + {\bf x}_{i_{s+1}} + \cdots + {\bf x}_{i_r})^{\tr}C ({\bf x}_{i_j} + {\bf x}_{i_{s+1}}+ \cdots + {\bf x}_{i_r})\\
&+ {\bf x}_{i_s}^{\tr}C{\bf x}_{i_s} +\sum_{j=s+1}^r({\bf x}_{i_j}+ {\bf x}_{i_s})^{\tr}C({\bf x}_{i_j}+ {\bf x}_{i_s})\\
&= \sum_{j=1}^{s-1}{\bf x}_{i_j}^{\tr} C{\bf x}_{i_j} + (s-1) \cdot \sum_{j=s+1}^r {\bf x}_{i_j}^{\tr} C {\bf x}_{i_j}\\
&+{\bf x}_{i_s}^{\tr}C {\bf x}_{i_s} + \sum_{j=s+1}^{r} {\bf x}_{i_j}^{\tr} C{\bf x}_{i_j} + (r-s)\cdot {\bf x}_{i_s}^{\tr}C{\bf x}_{i_s}\\
&=\sum_{j=1}^s{\bf x}_{i_j}^{\tr} C {\bf x}_{i_j}.\tag*{\qedhere}
\end{align*}
\end{proof}

\begin{lemma}\label{lemma4}
Let ${\bf x}_1,\ldots,{\bf x}_r, {\bf v}\in\FF_2^n$ where $r$ is even. Then,
\begin{equation}\label{eq10}
{\bf x}_1^2+\cdots +{\bf x}_r^2+({\bf x}_1+\cdots+{\bf x}_r)^2+{\bf v}^2=({\bf x}_1+\cdots+{\bf x}_r+{\bf v})^2+ ({\bf x}_1+{\bf v})^2+\cdots+({\bf x}_r+{\bf v})^2.
\end{equation}
\end{lemma}
\begin{proof}
Since $({\bf z}+{\bf w})^2={\bf z}^2+{\bf w}^2+{\bf z}\circ {\bf w}$, the right-hand side of~\eqref{eq10} equals
\begin{align*}
&({\bf x}_1+\cdots+{\bf x}_r)^2+{\bf v}^2+({\bf x}_1+\cdots+{\bf x}_r)\circ {\bf v}+\sum_{i=1}^r ({\bf x}_i^2+{\bf v}^2+{\bf x}_i\circ{\bf v})\\
=&({\bf x}_1+\cdots+{\bf x}_r)^2+{\bf v}^2+ r\cdot {\bf v}^2 + \sum_{i=1}^r {\bf x}_i^2,
\end{align*}
which equals the left-hand side of~\eqref{eq10} because $r$ is even.
\end{proof}

\begin{lemma}\label{lema-pomozna}
Let ${\bf x}_1,\ldots,{\bf x}_r\in \FF_2^n$ be linearly independent vectors. If $\sum_{i=1}^r {\bf x}_i^2=\sum_{i=1}^r {\bf y}_i^2$ or $r$ is even and $\sum_{i=1}^r {\bf x}_i^2 + ({\bf x}_1+\cdots +{\bf x}_r)^2=\sum_{i=1}^r {\bf y}_i^2+({\bf y}_1+\cdots +{\bf y}_r)^2$ for some ${\bf y}_1,\ldots,{\bf y}_r\in \FF_2^n$, then $\langle {\bf x}_1,\ldots,{\bf x}_r\rangle=\langle {\bf y}_1,\ldots,{\bf y}_r\rangle$.
\end{lemma}
\begin{proof}
Let $A_x\in \left\{\sum_{i=1}^r {\bf x}_i^2, \sum_{i=1}^r {\bf x}_i^2 + ({\bf x}_1+\cdots +{\bf x}_r)^2\right\}$. Then, $\rank A_x=r$. If there exists $j$ such that ${\bf y}_{j}\notin \langle {\bf x}_1,\ldots,{\bf x}_r\rangle$, then the the equation $A_x=A_y$ where $A_y$ is defined analogously as $A_x$ implies that $r+1=\rank(A_x-{\bf y}_j^2)=\rank(A_{y}-{\bf y}_j^2)\leq r$, a contradiction. Hence, ${\bf y}_{j}\in \langle {\bf x}_1,\ldots,{\bf x}_r\rangle$ for all $j$. Since $\rank A_y=r$, it follows that $\langle {\bf x}_1,\ldots,{\bf x}_r\rangle=\langle {\bf y}_1,\ldots,{\bf y}_r\rangle$.
\end{proof}

Given ${\bf x}\in\FF_2^n$, let $\overline{{\bf x}}\in \FF_2^{n+1}$ be defined by
$$\overline{{\bf x}}=\left(\begin{array}{c}{\bf x}\\{\bf x}^{\tr}{\bf x}\end{array}\right).$$
\begin{lemma}\label{orthocode1}
Let $\rank \big([{\bf x}_i^{\tr}{\bf x}_j]_{i,j=1}^r\big)\leq 1$ for some ${\bf x}_1,\ldots,{\bf x}_r\in \FF_2^n$. Then, $\langle \overline{{\bf x}}_1,\ldots,\overline{{\bf x}}_r \rangle$ is a self-orthogonal code in $\FF_2^{n+1}$. If ${\bf x}_1,\ldots,{\bf x}_r$ are linearly independent, then $r\leq \big\lfloor\frac{n+1}{2}\big\rfloor$.
\end{lemma}
\begin{proof}
Since $[{\bf x}_i^{\tr} {\bf x}_j]_{i,j=1}^{r}$ is of rank $\leq 1$, all its $2\times 2$ minors vanish. Hence, $$\overline{{\bf x}}_i^{\tr} \overline{{\bf x}}_j={\bf x}_i^{\tr} {\bf x}_j+{\bf x}_i^{\tr} {\bf x}_i\cdot {\bf x}_j^{\tr} {\bf x}_j={\bf x}_i^{\tr} {\bf x}_i\cdot {\bf x}_j^{\tr} {\bf x}_j-({\bf x}_i^{\tr} {\bf x}_j)^{2}=0$$ for all $i,j$ and the code $\langle \overline{{\bf x}}_1,\ldots,\overline{{\bf x}}_r \rangle$ is self-orthogonal. If ${\bf x}_1,\ldots,{\bf x}_r\in\FF_2^n$ are linearly independent, then $\dim \langle \overline{{\bf x}}_1,\ldots,\overline{{\bf x}}_r \rangle=r$ and therefore $r\leq \big\lfloor\frac{n+1}{2}\big\rfloor$.
\end{proof}

We end this section by recalling a version of Witt theorem for nonalternate symmetric bilinear forms over a field with characteristic two. Lemma~\ref{witt-osnovna} is a special case of~\cite[Theorem~3]{pless}.

\begin{lemma}\label{witt-osnovna}
Let $U$ be a vector subspace in $\FF_2^n$ and let $\sigma: U\to\FF_2^n$ be an injective linear map such that
\begin{equation}\label{eq24}
\sigma({\bf u}_1)^{\tr}\sigma({\bf u}_2)={\bf u}_1^{\tr} {\bf u}_2
\end{equation}
for all ${\bf u}_1, {\bf u}_2\in U$. Then, $\sigma$ can be extended to an injective linear map $\sigma: \FF_2^n\to\FF_2^n$ such that~\eqref{eq24} holds for all ${\bf u}_1, {\bf u}_2\in \FF_2^n$ if and only if the following two conditions are satisfied:
\begin{enumerate}
\item ${\bf j}_n\in U$ if and only if ${\bf j}_n\in \sigma(U)$,
\item if ${\bf j}_n\in U$, then $\sigma({\bf j}_n)={\bf j}_n$.
\end{enumerate}
\end{lemma}

\section{More auxiliary results}
\label{s4}

\begin{lemma}\label{witt1}
Let $A\in \sgl$ be a nonalternate matrix where $n\geq 5$ and assume that vectors ${\bf x}_1,{\bf x}_2,{\bf x}_3\in\FF_2^n$ are linearly independent. If
$$[{\bf x}_i^{\tr} A^{-1}{\bf x}_j]_{i,j=1}^{3}=\left(\begin{array}{ccc}1&1&0\\
1&1&0\\
0&0&0
\end{array}\right),$$
then there exists ${\bf w}\in \FF_2^n$ such that
\begin{equation}\label{eq26}
{\bf w}^{\tr}A^{-1}{\bf w}=1={\bf w}^{\tr}A^{-1}{\bf x}_1,\quad {\bf w}^{\tr}A^{-1}{\bf x}_2=0={\bf w}^{\tr}A^{-1}{\bf x}_3
\end{equation}
or
\begin{equation}\label{eq27}
{\bf w}^{\tr}A^{-1}{\bf w}=1={\bf w}^{\tr}A^{-1}{\bf x}_2,\quad {\bf w}^{\tr}A^{-1}{\bf x}_1=0={\bf w}^{\tr}A^{-1}{\bf x}_3.
\end{equation}
\end{lemma}
\begin{proof}
Since $A^{-1}$ is nonalternate, there exists $P\in GL_n(\FF_2)$ such that $A^{-1}=P^{\tr}P$. Denote $\dot{\bf x}_i=P{\bf x}_i$ for all $i$. Then,
$$[\dot{{\bf x}}_i^{\tr}\dot{{\bf x}}_j]_{i,j=1}^{3}=\left(\begin{array}{ccc}1&1&0\\
1&1&0\\
0&0&0
\end{array}\right).$$
We split the proof in two cases.
\begin{myenumerate}{Case}
\item Let ${\bf j}_n\in\langle \dot{\bf x}_1, \dot{\bf x}_2, \dot{\bf x}_3\rangle$. Suppose that $n$ is even. Then, ${\bf j}_n^{\tr}{\bf j}_n=0$ and ${\bf j}_n\in \{\dot{\bf x}_3,\dot{\bf x}_1+\dot{\bf x}_2,\dot{\bf x}_1+\dot{\bf x}_2+\dot{\bf x}_3\}$. Since ${\bf j}_n^{\tr}\dot{\bf x}_1=\dot{\bf x}_1^{\tr}\dot{\bf x}_1=1$, and values $\dot{\bf x}_3^{\tr}\dot{\bf x}_1$, $(\dot{\bf x}_1+\dot{\bf x}_2)^{\tr}\dot{\bf x}_1$, $(\dot{\bf x}_1+\dot{\bf x}_2+\dot{\bf x}_3)^{\tr}\dot{\bf x}_1$ are all zero, we have a contradiction. Hence, $n$ is odd, ${\bf j}_n^{\tr}{\bf j}_n=1$, and ${\bf j}_n\in \{\dot{\bf x}_1,\dot{\bf x}_2,\dot{\bf x}_1+\dot{\bf x}_3,\dot{\bf x}_2+\dot{\bf x}_3\}$. We consider two subcases.
\begin{myenumerateB}{Subcase}
\item\label{subcase1} Let ${\bf j}_n\in \{\dot{\bf x}_1,\dot{\bf x}_1+\dot{\bf x}_3\}$. Define vectors ${\bf y}_1={\bf j}_n$,  ${\bf y}_2={\bf e}_1$,  ${\bf y}_3={\bf e}_2+{\bf e}_3$ and the map $\sigma: \langle {\bf y}_1, {\bf y}_2, {\bf y}_3\rangle\to\FF_2^n$ by
    $$\sigma(\alpha_1 {\bf y}_1+ \alpha_2 {\bf y}_2+\alpha_3 {\bf y}_3)=\alpha_1 {\bf j}_n+ \alpha_2 \dot{{\bf x}}_2+\alpha_3 \dot{{\bf x}}_3\qquad (\alpha_1, \alpha_2, \alpha_3\in\FF_2).$$
    Then, $\sigma$ is linear and injective. Moreover,
    $$\sigma({\bf j}_n)={\bf j}_n\in \langle {\bf y}_1, {\bf y}_2, {\bf y}_3\rangle\cap \sigma(\langle {\bf y}_1, {\bf y}_2, {\bf y}_3\rangle),$$
    and
    \begin{equation}\label{eq28}
    [{\bf y}_i^{\tr}{\bf y}_j]_{i,j=1}^{3}=\left(\begin{array}{ccc}1&1&0\\
1&1&0\\
0&0&0
\end{array}\right)=[\sigma({\bf y}_i)^{\tr}\sigma({\bf y}_j)]_{i,j=1}^{3}.
\end{equation}
Consequently,
\begin{equation}\label{eq25}
\sigma({\bf y})^{\tr}\sigma({\bf z})={\bf y}^{\tr}{\bf z}
\end{equation}
for all ${\bf y},{\bf z}\in \langle {\bf y}_1, {\bf y}_2, {\bf y}_3\rangle$. By Lemma~\ref{witt-osnovna}, we can extend $\sigma$ to a linear map $\sigma$ on $\FF_2^n$ such that~\eqref{eq25} holds for all  ${\bf y},{\bf z}\in\FF_2^n$. Then ${\bf w}=P^{-1}\sigma({\bf e}_5)$ satisfies
\begin{align*}
{\bf w}^{\tr}A^{-1}{\bf w}&=\sigma({\bf e}_5)^{\tr}\sigma({\bf e}_5)={\bf e}_5^{\tr}{\bf e}_5=1,\\
{\bf w}^{\tr}A^{-1}{\bf x}_1&=\sigma({\bf e}_5)^{\tr}\dot{{\bf x}}_1\in \{\sigma({\bf e}_5)^{\tr}\sigma({\bf j}_n), \sigma({\bf e}_5)^{\tr}\sigma({\bf j}_n+{\bf y}_3)\}\\
&\phantom{a}\qquad\qquad\qquad=\{{\bf e}_5^{\tr}{\bf j}_n, {\bf e}_5^{\tr}({\bf j}_n+{\bf y}_3)\}=\{1\},\\
{\bf w}^{\tr}A^{-1}{\bf x}_2&=\sigma({\bf e}_5)^{\tr}\sigma({\bf y}_2)={\bf e}_5^{\tr}{\bf y}_2=0,\\
{\bf w}^{\tr}A^{-1}{\bf x}_3&=\sigma({\bf e}_5)^{\tr}\sigma({\bf y}_3)={\bf e}_5^{\tr}{\bf y}_3=0.
\end{align*}
Hence, \eqref{eq26} is true.

\item Let ${\bf j}_n\in \{\dot{\bf x}_2,\dot{\bf x}_2+\dot{\bf x}_3\}$. Here we define vectors ${\bf y}_1={\bf e}_1$,  ${\bf y}_2={\bf j}_n$,  ${\bf y}_3={\bf e}_2+{\bf e}_3$ and the map $\sigma: \langle {\bf y}_1, {\bf y}_2, {\bf y}_3\rangle\to\FF_2^n$ by
    $$\sigma(\alpha_1 {\bf y}_1+ \alpha_2 {\bf y}_2+\alpha_3 {\bf y}_3)=\alpha_1 \dot{{\bf x}}_1+ \alpha_2 {\bf j}_n+\alpha_3 \dot{{\bf x}}_3\qquad (\alpha_1, \alpha_2, \alpha_3\in\FF_2).$$ Then we continue as in Subcase~\ref{subcase1} to deduce that ${\bf w}=P^{-1}\sigma({\bf e}_5)$ satisfies~\eqref{eq27}.
\end{myenumerateB}

\item\label{case2lemma4-1} Let ${\bf j}_n\notin\langle \dot{\bf x}_1, \dot{\bf x}_2, \dot{\bf x}_3\rangle$. Suppose $n=5$ and let $\dot{\bf x}_3=(\beta_1,\beta_2,\beta_3,\beta_4,\beta_5)^{\tr}$. Since $\dot{{\bf x}}_1, \dot{{\bf x}}_2$ are linearly independent, $1=\dot{{\bf x}}_1^{\tr} \dot{{\bf x}}_1=\dot{{\bf x}}_2^{\tr} \dot{{\bf x}}_2=\dot{{\bf x}}_1^{\tr} \dot{{\bf x}}_2$, and $\dot{\bf x}_1, \dot{\bf x}_2$ differ from ${\bf j}_5$, it follows that either
     $$\{\dot{\bf x}_1, \dot{\bf x}_2\}=\{{\bf e}_i, {\bf e}_i+{\bf e}_j+{\bf e}_k\}\quad \textrm{or}\quad \{\dot{\bf x}_1, \dot{\bf x}_2\}=\{{\bf e}_i+{\bf e}_j+{\bf e}_k,{\bf e}_i+{\bf e}_l+{\bf e}_m\}$$ for some indices $\{i,j,k,l,m\}=\{1,2,3,4,5\}$. In both cases, the equalities
     $0=\dot{{\bf x}}_1^{\tr} \dot{{\bf x}}_3=\dot{{\bf x}}_2^{\tr} \dot{{\bf x}}_3=\dot{{\bf x}}_3^{\tr} \dot{{\bf x}}_3$ imply that  $\beta_i=0$, $\beta_j=\beta_k$, and $\beta_l=\beta_m$. Hence,  $\dot{\bf x}_1 + \dot{\bf x}_3={\bf j}_5$ or  $\dot{\bf x}_2 + \dot{\bf x}_3={\bf j}_5$ or $\dot{\bf x}_1 +\dot{\bf x}_2 +\dot{\bf x}_3=0$, which contradicts the assumption of Case~\ref{case2lemma4-1} or the linear independence of $\dot{\bf x}_1, \dot{\bf x}_2, \dot{\bf x}_3$.

     Therefore, $n\geq 6$. Define vectors ${\bf y}_1={\bf e}_1+\cdots + {\bf e}_5$,  ${\bf y}_2={\bf e}_1$,  ${\bf y}_3={\bf e}_2+{\bf e}_3$ and the map $\sigma: \langle {\bf y}_1, {\bf y}_2, {\bf y}_3\rangle\to\FF_2^n$ by
    $$\sigma(\alpha_1 {\bf y}_1+ \alpha_2 {\bf y}_2+\alpha_3 {\bf y}_3)=\alpha_1 \dot{{\bf x}}_1+ \alpha_2 \dot{{\bf x}}_2+\alpha_3 \dot{{\bf x}}_3\qquad (\alpha_1, \alpha_2, \alpha_3\in\FF_2).$$
    Then, ${\bf j}_n\notin\langle \dot{\bf x}_1, \dot{\bf x}_2, \dot{\bf x}_3\rangle=\sigma(\langle {\bf y}_1, {\bf y}_2, {\bf y}_3\rangle)$. Since $n\geq 6$, it follows that ${\bf j}_n\notin \langle {\bf y}_1, {\bf y}_2, {\bf y}_3\rangle$. Moreover, \eqref{eq28} is still true, which implies \eqref{eq25} for all ${\bf y},{\bf z}\in \langle {\bf y}_1, {\bf y}_2, {\bf y}_3\rangle$. We continue as in Subcase~\ref{subcase1} to see that ${\bf w}=P^{-1}\sigma({\bf e}_5)$ fits~\eqref{eq26}.\qedhere
\end{myenumerate}
\end{proof}

\begin{lemma}\label{witt2}
Let ${\bf x}_1,\ldots,{\bf x}_r\in\FF_2^n$ be linearly independent vectors where $r$ is an even number such that $2\leq r\leq n$. If
$A\in \sgl$ and ${\bf x}_i^{\tr} A^{-1}{\bf x}_j=1$ for all $i,j\leq r$, then there exists ${\bf x}_{r+1}\in \FF_2^n$ such that
\begin{equation}\label{eq45}
{\bf x}_s^{\tr} A^{-1}{\bf x}_{r+1}=1\ \textrm{for some}\ s\in\{1,\ldots,r\}
\end{equation}
and
\begin{equation}\label{eq46}
{\bf x}_t A^{-1}{\bf x}_{r+1}=0\ \textrm{for all}\ t\in\{1,\ldots,r+1\}\backslash\{s\}.
\end{equation}
\end{lemma}
\begin{remark}\label{remark}
Observe that vectors ${\bf x}_1,\ldots, {\bf x}_{r+1}$ must be linearly independent. In fact, if we pre-multiply the equation ${\bf x}_{r+1}=\sum_{k=1}^{r}\alpha_k {\bf x}_k$ by ${\bf x}_s^{\tr}A^{-1}$ and by ${\bf x}_t^{\tr}A^{-1}$ where $t\in\{1,\ldots,r\}\backslash\{s\}$, respectively, then we deduce that $1=\sum_{k=1}^r \alpha_k=0$, which is not possible.
\end{remark}

\begin{proof}
Clearly, $A^{-1}$ is not alternate. Hence, there exists $P\in GL_n(\FF_2)$ such that $A^{-1}=P^{\tr}P$. Define $\dot{{\bf x}}_i=P{\bf x}_i$ for all $i$. Then, $\dot{{\bf x}}_i^{\tr}\dot{{\bf x}}_j=1$ for all $i,j$. In particular, $\rank\left([\dot{{\bf x}}_i^{\tr}\dot{{\bf x}}_j]_{i,j}^r\right)=1$.
By Lemma~\ref{orthocode1}, $r\leq \big\lfloor\frac{n+1}{2}\big\rfloor$, i.e. $n\geq 2r-1$. Firstly we prove that
\begin{equation}\label{eq29}
\textrm{if}\ n=2r-1,\ \textrm{then}\ {\bf j}_n\in \langle \dot{{\bf x}}_1,\ldots, \dot{{\bf x}}_r\rangle.
\end{equation}
In fact, if $n=2r-1$ and ${\bf j}_n\notin \langle \dot{{\bf x}}_1,\ldots, \dot{{\bf x}}_r\rangle$, then $\dot{{\bf x}}_1,\ldots, \dot{{\bf x}}_r, {\bf j}_n$ are linearly independent. Since $n$ is odd it follows that ${\bf j}_n^{\tr}{\bf j}_n=1=\dot{{\bf x}}_i^{\tr}\dot{{\bf x}}_i=\dot{{\bf x}}_i^{\tr}{\bf j}_n$ for all $i$. As above, Lemma~\ref{orthocode1} implies that $r+1\leq \big\lfloor\frac{n+1}{2}\big\rfloor=r$, a contradiction.

We split the rest of the proof in two cases.
\begin{myenumerate}{Case}
\item Let ${\bf j}_n\notin \langle \dot{{\bf x}}_1,\ldots, \dot{{\bf x}}_r\rangle$. Define vectors ${\bf y}_i=\sum_{j=1}^{2i-1} {\bf e}_j$ for  $i\in\{1,\ldots,r\}$ and the map $\sigma: \langle {\bf y}_1,\ldots, {\bf y}_r\rangle\to\FF_2^n$ by
    $$\sigma\left(\sum_{i=1}^r \alpha_i {\bf y}_i\right)=\sum_{i=1}^r \alpha_i \dot{{\bf x}}_i\qquad (\alpha_i\in\FF_2).$$
    Since $n\geq 2r$ by~\eqref{eq29}, it follows that ${\bf j}_n\notin \langle {\bf y}_1,\ldots,{\bf y}_r\rangle$. Since $\sigma(\langle {\bf y}_1,\ldots,{\bf y}_r\rangle)=\langle\dot{{\bf x}}_1,\ldots,\dot{{\bf x}}_r\rangle$, it follows that ${\bf j}_n\notin \sigma(\langle {\bf y}_1,\ldots,{\bf y}_r\rangle)$. Since ${\bf y}_i^{\tr}{\bf y}_j=1=\sigma({\bf y}_i)^{\tr}\sigma({\bf y}_j)$ for all $i,j$, we deduce~\eqref{eq25} for all ${\bf y}, {\bf z}\in \langle {\bf y}_1,\ldots,{\bf y}_r\rangle$. By Lemma~\ref{witt-osnovna}, we can linearly extend $\sigma$ on whole $\FF_2^n$ such that~\eqref{eq25} is true for all ${\bf y}, {\bf z}\in \FF_2^n$. If ${\bf x}_{r+1}:=P^{-1}\sigma({\bf e}_1+{\bf e}_2)$, then
    \begin{align*}
    {\bf x}_1^{\tr}A^{-1}{\bf x}_{r+1}&=\dot{{\bf x}}_1^{\tr}\sigma({\bf e}_1+{\bf e}_2)=\sigma({\bf y}_1)^{\tr}\sigma({\bf e}_1+{\bf e}_2)={\bf y}_1^{\tr}({\bf e}_1+{\bf e}_2)=1,\\
    {\bf x}_t^{\tr}A^{-1}{\bf x}_{r+1}&={\bf y}_t^{\tr}({\bf e}_1+{\bf e}_2)=0\qquad (t\in\{2,\ldots,r\}),\\
    {\bf x}_{r+1}^{\tr}A^{-1}{\bf x}_{r+1}&=({\bf e}_1+{\bf e}_2)^{\tr}({\bf e}_1+{\bf e}_2)=0,
    \end{align*}
    as claimed.

\item Let ${\bf j}_n\in \langle \dot{{\bf x}}_1,\ldots, \dot{{\bf x}}_r\rangle$. Then ${\bf j}_n=\sum_{i=1}^r \beta_{i} \dot{{\bf x}}_i$ for some $\beta_i\in\FF_2$. Hence,
    \begin{equation}\label{eq30}
    1=\dot{{\bf x}}_j^{\tr}\dot{{\bf x}}_j=\dot{{\bf x}}_j^{\tr}{\bf j}_n=\sum_{i=1}^r \beta_{i} \dot{{\bf x}}_j^{\tr}\dot{{\bf x}}_i=\sum_{i=1}^r \beta_{i}
    \end{equation}
    for all $j$, and therefore
    $${\bf j}_n^{\tr}{\bf j}_n=\left(\sum_{j=1}^r \beta_{j} \dot{{\bf x}}_j\right)^{\tr}\left(\sum_{i=1}^r \beta_{i} \dot{{\bf x}}_i\right)=\sum_{i,j=1}^r \beta_i\beta_j \dot{{\bf x}}_j^{\tr}\dot{{\bf x}}_i=\left(\sum_{i=1}^r \beta_{i}\right)^2=1,$$
    i.e. $n$ is odd. Since $r$ is even, \eqref{eq30} implies the existence of $s,k\in\{1,\ldots,r\}$ such that $\beta_s=0$ and $\beta_k=1$. Hence, vectors in the set $\{{\bf j}_n,\dot{{\bf x}}_1,\ldots, \dot{{\bf x}}_r\}\backslash\{\dot{{\bf x}}_k\}$ are linearly independent. Denote them by $\ddot{{\bf x}}_1,\ldots, \ddot{{\bf x}}_r$ in some order where $\ddot{\bf x}_1=\dot{\bf x}_s$ and $\ddot{\bf x}_r={\bf j}_n$. Further let ${\bf y}_m=\sum_{j=1}^{2m-1} {\bf e}_j$ for  $m\in\{1,\ldots,r-1\}$ and ${\bf y}_r={\bf j}_n$. Then ${\bf y}_1,\ldots, {\bf y}_r$ are linearly independent and the map $\sigma: \langle {\bf y}_1,\ldots, {\bf y}_r\rangle\to\FF_2^n$,
    $$\sigma\left(\sum_{i=1}^r \alpha_i {\bf y}_i\right)=\sum_{i=1}^r \alpha_i \ddot{{\bf x}}_i\qquad (\alpha_i\in\FF_2),$$
    is well defined. Moreover, $$\sigma({\bf j}_n)={\bf j}_n\in \langle {\bf y}_1,\ldots, {\bf y}_r\rangle\cap \sigma(\langle {\bf y}_1,\ldots, {\bf y}_r\rangle).$$
    Since ${\bf y}_i^{\tr}{\bf y}_j=1=\sigma({\bf y}_i)^{\tr}\sigma({\bf y}_j)$ for all $i,j$, we can apply Lemma~\ref{witt-osnovna} as above to extend $\sigma$ linearly on whole $\FF_2^n$ such that~\eqref{eq25} is true for all ${\bf y}, {\bf z}\in \FF_2^n$. If ${\bf x}_{r+1}:=P^{-1}\sigma({\bf e}_1+{\bf e}_2)$, then
    \begin{align*}
    {\bf x}_s^{\tr}A^{-1}{\bf x}_{r+1}&=\dot{{\bf x}}_s^{\tr}\sigma({\bf e}_1+{\bf e}_2)=\sigma({\bf y}_1)^{\tr}\sigma({\bf e}_1+{\bf e}_2)={\bf y}_1^{\tr}({\bf e}_1+{\bf e}_2)=1,\\
    {\bf x}_t^{\tr}A^{-1}{\bf x}_{r+1}&=\sigma({\bf e}_1+{\bf e}_2+\cdots )^{\tr}\sigma({\bf e}_1+{\bf e}_2)=0 \qquad (t\in\{1,\ldots,r\}\backslash\{s,k\}),\\
    {\bf x}_{r+1}^{\tr}A^{-1}{\bf x}_{r+1}&=\sigma({\bf e}_1+{\bf e}_2)^{\tr}\sigma({\bf e}_1+{\bf e}_2)=0,\\
    {\bf x}_{k}^{\tr}A^{-1}{\bf x}_{r+1}&=\dot{{\bf x}}_k^{\tr}\sigma({\bf e}_1+{\bf e}_2)=\left({\bf j}_n+ \sum_{j\neq k,s} \beta_{j} \dot{{\bf x}}_j\right)^{\tr}\sigma({\bf e}_1+{\bf e}_2)\\
    &=\sigma\left({\bf j}_n+ \sum_{j=2}^{r-1} \beta_{j} {\bf y}_j\right)^{\tr}\sigma({\bf e}_1+{\bf e}_2)=\left({\bf j}_n+ \sum_{j=2}^{r-1} \beta_{j} {\bf y}_j\right)^{\tr}({\bf e}_1+{\bf e}_2)\\
    &=0,
    \end{align*}
    as claimed.\qedhere
\end{myenumerate}
\end{proof}

\begin{lemma}\label{lemma7}
Let ${\bf x}_1,\ldots,{\bf x}_r\in \FF_2^n$ be linearly independent where $4\leq r\leq n$. Further, let $A,A+\sum_{i=1}^r {\bf x}_i^2\in \sgl$ and assume that ${\bf x}_i^{\tr}A^{-1}{\bf x}_i=0$ for some $i\in \{1,\ldots,r\}$. Then there exist ${\bf y}_1,\ldots,{\bf y}_r\in \FF_2^n$ such that
$\sum_{i=1}^r {\bf x}_i^2=\sum_{i=1}^r {\bf y}_i^2$ and both matrices $A+\sum_{i=1}^{r-2} {\bf y}_i^2, A+\sum_{i=1}^{r-1} {\bf y}_i^2$ are in  $\sgl$. \end{lemma}
\begin{proof}
Let $B=A+\sum_{i=1}^r {\bf x}_i^2$. Define $S_0=\{i\in \{1,\ldots,r\} : {\bf x}_i^{\tr}B^{-1}{\bf x}_i=0\}$ and $S_1=\{i\in \{1,\ldots,r\} : {\bf x}_i^{\tr}B^{-1}{\bf x}_i=1\}$. By Corollary~\ref{c1}, $|S_0|\geq 1$. If $|S_0|=1$ and $S_0=\{j\}$, then there exist $k,l,m\in S_1$ and we can define vectors $\dot{{\bf x}}_1,\ldots, \dot{{\bf x}}_r$ by
$$\dot{{\bf x}}_t:=\left\{\begin{array}{lll}{\bf x}_t&\textrm{if}& t\notin\{j,k,l,m\},\\
{\bf x}_t+{\bf x}_j+{\bf x}_k+{\bf x}_l+{\bf x}_m&\textrm{if}& t\in\{j,k,l,m\}.\\\end{array}\right.$$
Then, \begin{equation*}
1=\dot{{\bf x}}_j^{\tr}B^{-1}\dot{{\bf x}}_j,\
0=\dot{{\bf x}}_k^{\tr}B^{-1}\dot{{\bf x}}_k=\dot{{\bf x}}_l^{\tr}B^{-1}\dot{{\bf x}}_l=\dot{{\bf x}}_m^{\tr}B^{-1}\dot{{\bf x}}_m
\end{equation*}
so the set $\dot{S}_0=\{i\in \{1,\ldots,r\} : \dot{{\bf x}}_i^{\tr}B^{-1}\dot{{\bf x}}_i=0\}$ has three elements and $\sum_{i=1}^r {\bf x}_i^2=\sum_{i=1}^r \dot{{\bf x}}_i^2$. Since we can replace vectors ${\bf x}_1,\ldots,{\bf x}_r$ by $\dot{{\bf x}}_1,\ldots, \dot{{\bf x}}_r$, we assume in the rest of the proof that $|S_0|\geq 2$.

Next, we claim that there exist $i_1,i_2\in \{1,\ldots,r\}$ such that
\begin{align}
\nonumber 0={\bf x}_{i_1}^{\tr}B^{-1}{\bf x}_{i_1}&={\bf x}_{i_2}^{\tr}B^{-1}{\bf x}_{i_2}={\bf x}_{i_1}^{\tr}B^{-1}{\bf x}_{i_2}\\
\label{eq13}&\textrm{or}\\
\nonumber0={\bf x}_{i_1}^{\tr}B^{-1}{\bf x}_{i_1}&,\  1={\bf x}_{i_2}^{\tr}B^{-1}{\bf x}_{i_2}={\bf x}_{i_1}^{\tr}B^{-1}{\bf x}_{i_2}.
\end{align}
Suppose \eqref{eq13} is not true. Then, for each $s\in S_0$ the row $s$ of the matrix $I_r+ [{\bf x}_{j_1}^{\tr}B^{-1}{\bf x}_{j_2}]_{j_1,j_2=1}^r$ equals $(a_1,\ldots,a_r)$ where $a_{j_2}=1$ if and only if $j_2\in S_0$. Since $|S_0|\geq 2$, it follows that
$$\det\big(I_r+ [{\bf x}_{j_1}^{\tr}B^{-1}{\bf x}_{j_2}]_{j_1,j_2=1}^r\big)=0.$$ On the other hand,
$$1=\det A=\det\left(B+\sum_{i=1}^r {\bf x}_i^2\right)=\det\big(I_r+ [{\bf x}_{j_1}^{\tr}B^{-1}{\bf x}_{j_2}]_{j_1,j_2=1}^r\big),$$
a contradiction.

Let $\{{\bf y}_1,\ldots,{\bf y}_r\}=\{{\bf x}_1,\ldots,{\bf x}_r\}$ where ${\bf y}_{r-1}={\bf x}_{i_2}$ and ${\bf y}_{r}={\bf x}_{i_1}$. Then,
\begin{align*}
\det\left(A+\sum_{i=1}^{r-1} {\bf y}_i^2\right)&=\det(B+{\bf y}_r^2)=1+ {\bf y}_r^{\tr}B^{-1}{\bf y}_r=1\\
\det\left(A+\sum_{i=1}^{r-2} {\bf y}_i^2\right)&=\det(B+{\bf y}_{r-1}^2+{\bf y}_r^2)\\
&=\det\left(I_2+ \left(\begin{array}{cc}{\bf y}_{r-1}^{\tr}B^{-1}{\bf y}_{r-1}&{\bf y}_{r-1}^{\tr}B^{-1}{\bf y}_r\\
{\bf y}_r^{\tr}B^{-1}{\bf y}_{r-1}&{\bf y}_r^{\tr}B^{-1}{\bf y}_r\end{array}\right)\right)=1,
\end{align*}
as claimed.
\end{proof}

\begin{lemma}\label{lemma5}
Let $A\in\sgl$, let vectors ${\bf x}_1,\ldots,{\bf x}_r\in \FF_2^n$ be linearly independent, and assume that ${\bf x}_i^{\tr}A^{-1}{\bf x}_i=1$ for all $i$. If $\rank({\bf x}_1^2+\cdots + {\bf x}_r^2+ {\bf y}^2)<r$ for some ${\bf y}\in \FF_2^n$, then ${\bf y}^{\tr}A^{-1}{\bf y}=1$.
\end{lemma}
\begin{proof}
Pick any $P\in GL_n(\FF_2)$ that has ${\bf x}_i$ as the $i$-th column for $i=1,\ldots,r$. Let ${\bf z}=P^{-1} {\bf y}$. Then ${\bf x}_1^2+\cdots + {\bf x}_r^2+ {\bf y}^2=P({\bf e}_1^2+\cdots + {\bf e}_r^2+ {\bf z}^2)P^{\tr}$, so
\begin{equation}\label{eq14}
\rank({\bf e}_1^2+\cdots + {\bf e}_r^2+ {\bf z}^2)<r.
\end{equation}
By \cite[Lemma~3.1]{LAA2006}, ${\bf z}=\sum_{i=1}^r z_i {\bf e}_i$ for some $z_i\in\FF_2$. By \eqref{eq14} and Lemma~\ref{l2},
$$0=\det(I_r+(z_1,\ldots,z_r)^{\tr} (z_1,\ldots,z_r))=1+(z_1,\ldots,z_r)(z_1,\ldots,z_r)^{\tr}=1+\sum_{i=1}^{r} z_i^2.$$
Consequently, in characteristic two we deduce that
$${\bf y}^{\tr}A^{-1}{\bf y}=\left(\sum_{i=1}^r z_i {\bf e}_i\right)^{\tr}P^{\tr}A^{-1}P\left(\sum_{i=1}^r z_i {\bf e}_i\right)=\sum_{i=1}^r z_i^2\cdot {\bf x}_i^{\tr}A^{-1}{\bf x}_i=1.\qedhere$$
\end{proof}

\begin{lemma}\label{lemma66}
Let $A, A+\sum_{i=1}^{r} {\bf x}_i^2\in\sgl$ where ${\bf x}_1,\ldots,{\bf x}_r\in \FF_2^n$ are linearly independent and ${\bf x}_i^{\tr}A^{-1}{\bf x}_i=1$ for all $i$. Then, $d(A,A+\sum_{i=1}^{r} {\bf x}_i^2)\geq r+1$.
\end{lemma}
\begin{proof}
By \eqref{eq3}, $d(A,A+\sum_{i=1}^{r} {\bf x}_i^2)\geq r$. Suppose that  $d(A,A+\sum_{i=1}^{r} {\bf x}_i^2)=r$. Then there exist nonzero vectors ${\bf y}_1,\ldots, {\bf y}_r$ such that $A+\sum_{i=1}^{s} {\bf y}_i^2\in\sgl$ for all $s\in\{1,\ldots,r\}$ and $\sum_{i=1}^{r} {\bf y}_i^2=\sum_{i=1}^{r} {\bf x}_i^2$. In particular, $\rank(\sum_{i=1}^{r} {\bf x}_i^2+ {\bf y}_1^2)=\rank(\sum_{i=2}^{r} {\bf y}_i^2)\leq r-1$. By Lemma~\ref{lemma5}, ${\bf y}_1^{\tr}A^{-1}{\bf y}_1=1$. On the other hand, $1=\det(A+{\bf y}_1^2)=1+{\bf y}_1^{\tr}A^{-1}{\bf y}_1=0$, a contradiction.
\end{proof}

\begin{lemma}\label{lemma6}
Let $A, A+\sum_{i=1}^{r} {\bf x}_i^2\in\sgl$ where ${\bf x}_1,\ldots,{\bf x}_r\in \FF_2^n$ are linearly independent, $2\leq r\leq n$, and $[{\bf x}_i^{\tr}A^{-1}{\bf x}_j]_{i,j=1}^r\in \ronetzero$. Then, $d(A,A+\sum_{i=1}^{r} {\bf x}_i^2)\geq r+2$.
\end{lemma}
\begin{proof}
By Proposition~\ref{prop0}, we can permute vectors ${\bf x}_1,\ldots,{\bf x}_r$  to achieve that
\begin{equation}\label{eq16}
[{\bf x}_i^{\tr}A^{-1}{\bf x}_j]_{i,j=1}^r=\left(\begin{array}{cc}J_{k\times k}& O\\O&O\end{array}\right)
\end{equation}
for some even $k\in \{2,\ldots,r\}$.

Suppose that $d=d(A,A+\sum_{i=1}^{r} {\bf x}_i^2)$ where $d\leq r+1$.  Then there exist nonzero ${\bf y}_1,\ldots,{\bf y}_d\in\FF_2^n$ such that $A+\sum_{i=1}^{s} {\bf y}_i^2\in\sgl$ for all $s\in\{1,\ldots,d\}$ and $\sum_{i=1}^{d} {\bf y}_i^2=\sum_{i=1}^{r} {\bf x}_i^2$. If there exists $i_0$ such that ${\bf y}_{i_0}\notin \langle {\bf x}_1,\ldots,{\bf x}_r\rangle$, then
$$r+1=\rank({\bf x}_1^2+\cdots +{\bf x}_r^2+{\bf y}_{i_0}^2)=\rank\left(\sum_{i\neq i_0} {\bf y}_i^2\right)\leq d-1\leq r,$$
a contradiction. Hence, $\langle {\bf y}_1,\ldots,{\bf y}_d\rangle\subseteq \langle {\bf x}_1,\ldots,{\bf x}_r\rangle$. On the other hand,
$$\dim \langle {\bf y}_1,\ldots,{\bf y}_d\rangle\geq \rank({\bf y}_1^2+\cdots +{\bf y}_d^2)=\rank({\bf x}_1^2+\cdots +{\bf x}_r^2)=r.$$
Consequently,  $d\in \{r,r+1\}$ and $\langle {\bf y}_1,\ldots,{\bf y}_d\rangle=\langle {\bf x}_1,\ldots,{\bf x}_r\rangle$. In particular,
$${\bf y}_i=\sum_{j=1}^r \alpha_{j}^{(i)}{\bf x}_j\qquad (i=1,\ldots,d)$$
for some $\alpha_{j}^{(i)}\in\FF_2$.

We next claim that whenever ${\bf y}_s^{\tr} A^{-1} {\bf y}_s=0$ for some $s$, we have also ${\bf y}_s^{\tr} A^{-1} {\bf y}_t=0$ for all $t$. In fact,  \begin{align*}
0={\bf y}_s^{\tr} A^{-1} {\bf y}_s&=\left(\sum_{j=1}^r \alpha_{j}^{(s)}{\bf x}_j\right)^{\tr}A^{-1}\left(\sum_{j=1}^r \alpha_{j}^{(s)}{\bf x}_j\right)\\
&=\sum_{j=1}^r (\alpha_{j}^{(s)})^2 {\bf x}_j^{\tr} A^{-1} {\bf x}_j=\sum_{j=1}^k \alpha_{j}^{(s)}
\end{align*}
and consequently
\begin{align*}
{\bf y}_s^{\tr} A^{-1} {\bf y}_t=&\left(\sum_{j=1}^r \alpha_{j}^{(s)}{\bf x}_j\right)^{\tr}A^{-1}\left(\sum_{i=1}^r \alpha_{i}^{(t)}{\bf x}_i\right)\\
=&\sum_{j=1}^r \sum_{i=1}^r \alpha_{j}^{(s)}\alpha_{i}^{(s)} {\bf x}_j^{\tr} A^{-1} {\bf x}_i\\
=&\left(\sum_{j=1}^k  \alpha_{j}^{(s)}\right)\cdot \left(\sum_{i=1}^k  \alpha_{i}^{(t)}\right)=0.
\end{align*}

Next, we use the induction to prove that ${\bf y}_s^{\tr} A^{-1} {\bf y}_s=0$ for all $s$, and therefore
\begin{equation}\label{eq15}
{\bf y}_s^{\tr} A^{-1} {\bf y}_t=0\ \textrm{for all}\ s,t.
\end{equation}
Since $1=\det(A+{\bf y}_1^2)=1+{\bf y}_1^{\tr} A^{-1} {\bf y}_1$, we deduce that ${\bf y}_1^{\tr} A^{-1} {\bf y}_1=0$. To prove the inductive step assume that $0={\bf y}_1^{\tr} A^{-1} {\bf y}_1=\cdots={\bf y}_{s-1}^{\tr} A^{-1} {\bf y}_{s-1}$. Then,
\begin{align*}
1&=\det\left(A+\sum_{i=1}^{s} {\bf y}_i^2\right)\\
&=\det\left(I_s+[{\bf y}_i^{\tr} A^{-1} {\bf y}_j]_{i,j=1}^s\right)\\
&=\det\left(\begin{array}{cc}I_{s-1}&O_{(s-1)\times 1}\\
O_{1\times (s-1)}& 1+{\bf y}_s^{\tr} A^{-1} {\bf y}_s\end{array}\right)=1+{\bf y}_s^{\tr} A^{-1} {\bf y}_s
\end{align*}
by Lemma~\ref{l2}. Hence, ${\bf y}_s^{\tr} A^{-1} {\bf y}_s=0$. This completes the proof of~\eqref{eq15}. Since $\langle {\bf y}_1,\ldots,{\bf y}_d\rangle=\langle {\bf x}_1,\ldots,{\bf x}_r\rangle$, both vectors ${\bf x}_1, {\bf x}_2$ are linear combinations of ${\bf y}_1,\ldots,{\bf y}_d$. Therefore, \eqref{eq15} implies that ${\bf x}_1^{\tr} A^{-1} {\bf x}_2=0$, which contradicts~\eqref{eq16}. Consequently, $d(A,A+\sum_{i=1}^{r} {\bf x}_i^2)\geq r+2$.
\end{proof}

\begin{lemma}\label{lemma1}
Suppose $1\leq r\leq n$, $A\in \sgl$, and ${\bf x}_1,\ldots,{\bf x}_r\in \FF_2^{n}$ are linearly independent. If ${\bf x}_i^{\tr} A^{-1} {\bf x}_j=1$ for all $i,j$, and $\sum_{i=1}^r {\bf x}_i^2=\sum_{i=1}^r {\bf y}_i^2$ for some vectors ${\bf y}_1,\ldots,{\bf y}_r\in \FF_2^{n}$, then ${\bf y}_i^{\tr} A^{-1} {\bf y}_j=1$ for all $i,j$.
\end{lemma}
\begin{proof}
By Lemma~\ref{lema-pomozna}, $\langle {\bf y}_1, \ldots , {\bf y}_r \rangle=\langle {\bf x}_1, \ldots , {\bf x}_r \rangle$. Hence, for each $i$ there exist constants $\alpha_k^{(i)}\in \FF_2$ s.t. ${\bf y}_i= \sum_{j=1}^r \alpha_k^{(i)}{\bf x}_k$. Let $X, Y\in M_{n\times r}(\FF_2)$ and $P\in M_{r\times r}(\FF_2)$ be the matrices with the $i$-th column equal to ${\bf x}_i$, ${\bf y}_i$ and ${\bf p}_i=(\alpha_1^i, \dots, \alpha_r^i)^{\tr}$, respectively. Obviously, $Y=XP$. Therefore,
$$XX^{\tr}= {\bf x}_1^2+\cdots+ {\bf x}_r^2={\bf y}_1^2+\cdots+ {\bf y}_r^2=YY^{\tr}=XPP^{\tr}X^{\tr},$$
i.e. $X(I_r-PP^{\tr})X^{\tr}=0$. Since $\rank X=r$, we deduce that $P^{\tr}=P^{-1}$. Hence,
$$[{\bf y}_i^{\tr} A^{-1} {\bf y}_j]_{i,j=1}^{r}=Y^{\tr}A^{-1}Y=P^{\tr} X^{\tr} A^{-1} XP=P^{\tr}JP=: {\bf z}^2$$ where ${\bf z}=(z_1,\ldots,z_r)^{\tr}=P^{\tr}{\bf j}_r$. Since $I_r=P^{\tr}P=[{\bf p}_i^{\tr}{\bf p}_j]_{i,j=1}^r$, it follows that ${\bf p}_i^{\tr}{\bf p}_i=1$ for all $i$. Since the underlying field is $\FF_2$, we deduce that $z_i={\bf p}_i^{\tr} {\bf j}_r={\bf p}_i^{\tr}{\bf p}_i=1$ for all $i$. Hence, $[{\bf y}_i^{\tr} A^{-1} {\bf y}_j]_{i,j=1}^{r}=J$ as claimed.
\end{proof}

\begin{lemma}\label{lemma2}
Let $A\in \sgl$ and let ${\bf x}_1,\ldots,{\bf x}_r\in \FF_2^{n}$ be linearly independent where $2\leq r\leq n$, $[{\bf x}_i^{\tr}A^{-1}{\bf x}_j]_{i,j=1}^r\in \ronetzero$, and ${\bf x}_i^{\tr} A^{-1} {\bf x}_j=0$ for some $i,j$.
\begin{enumerate}
\item\label{lemma2i} Let $k\in \{2,\ldots,r-1\}$ be even. Then, there exist linearly independent vectors ${\bf y}_1,\ldots,{\bf y}_r\in \FF_2^{n}$ and a permutation matrix $Q\in GL_r(\FF_2)$ such that
    \begin{equation}\label{eq17}
    \sum_{i=1}^r {\bf x}_i^2=\sum_{i=1}^r {\bf y}_i^2,\quad  \sum_{i=1}^r {\bf x}_i=\sum_{i=1}^r {\bf y}_i,
    \end{equation}
    and
    $$[{\bf y}_i^{\tr}A^{-1}{\bf y}_j]_{i,j=1}^r=Q\left(\begin{array}{cc}J_{k\times k} & O\\
    O&O\end{array}\right)Q^{\tr}.$$
\item\label{lemma2ii} If $\sum_{i=1}^r {\bf x}_i^2=\sum_{i=1}^r {\bf y}_i^2$ for some ${\bf y}_1,\ldots,{\bf y}_r\in \FF_2^{n}$, then $[{\bf y}_i^{\tr}A^{-1}{\bf y}_j]_{i,j=1}^r\in \ronetzero$ and ${\bf y}_{i'}^{\tr} A^{-1} {\bf y}_{j'}=0$ for some~$i',j'$.
\end{enumerate}
\end{lemma}
\begin{proof}
\eqref{lemma2i} The linear independence of ${\bf y}_1,\ldots,{\bf y}_r$ will follow from Lemma~\ref{lema-pomozna} and \eqref{eq17}. By Proposition~\ref{prop0}, there is a permutation matrix $Q\in GL_r(\FF_2)$ s.t.
\begin{equation}\label{eq19}
[{\bf x}_i^{\tr}A^{-1}{\bf x}_j]_{i,j=1}^r=Q\left(\begin{array}{cc}J_{l\times l} & O\\
    O&O\end{array}\right)Q^{\tr}
\end{equation}
for some even $l\in \{2,\ldots,r\}$. By the assumption, $l\leq r-1$. To simplify writings we assume that $Q=I_r$. The claim is obvious if $k=l$. Moreover, since we can apply an induction process, it suffices to prove the claim for the case $k=l+2$ whenever $l+2\leq r-1$ and for the case $k=l-2$ whenever $l-2\geq 2$.
\begin{myenumerate}{Case}
\item Let $k=l+2$ and $l+3\leq r$. Define vectors ${\bf y}_1,\ldots, {\bf y}_r$ by
$${\bf y}_i:=\left\{\begin{array}{lll}{\bf x}_i&\textrm{if}& i\notin\{l,l+1,l+2,l+3\},\\
{\bf x}_i+{\bf x}_l+{\bf x}_{l+1}+{\bf x}_{l+2}+{\bf x}_{l+3}&\textrm{if}& i\in\{l,l+1,l+2,l+3\}.\end{array}\right.$$
Then, \eqref{eq17} is true and
\begin{align*}
[{\bf y}_i^{\tr}A^{-1}{\bf y}_j]_{i,j=1}^r&=\left(\begin{array}{cccc}
J_{(l-1)\times(l-1)}&O_{(l-1)\times 1}&J_{(l-1)\times 3}&O\\
O_{1\times (l-1)}&0&O_{1\times 3}&O\\
J_{3\times(l-1)}&O_{3\times 1}&J_{3\times 3}&O\\
O&O&O&O
\end{array}\right)\\
&=\dot{Q}\left(\begin{array}{cc}J_{(l+2)\times (l+2)} & O\\
    O&O\end{array}\right)\dot{Q}^{\tr}
\end{align*}
for appropriate permutation matrix $\dot{Q}$.

\item Let $k=l-2$ and $l\geq 4$. Define vectors ${\bf y}_1,\ldots, {\bf y}_r$ by
$${\bf y}_i:=\left\{\begin{array}{lll}{\bf x}_i&\textrm{if}& i\notin\{l-2,l-1,l,l+1\},\\
{\bf x}_i+{\bf x}_{l-2}+{\bf x}_{l-1}+{\bf x}_{l}+{\bf x}_{l+1}&\textrm{if}& i\in\{l-2,l-1,l,l+1\}.\end{array}\right.$$
Then, \eqref{eq17} is true and
\begin{align*}
[{\bf y}_i^{\tr}A^{-1}{\bf y}_j]_{i,j=1}^r&=\left(\begin{array}{cccc}
J_{(l-3)\times(l-3)}&O_{(l-3)\times 3}&{\bf j}_{l-3}&O\\
O_{3\times (l-3)}&O_{3\times 3}&O_{3\times 1}&O\\
{\bf j}_{l-3}^{\tr }&O_{1\times 3}&1&O\\
O&O&O&O
\end{array}\right)\\
&=\dot{Q}\left(\begin{array}{cc}J_{(l-2)\times (l-2)} & O\\
    O&O\end{array}\right)\dot{Q}^{\tr}
\end{align*}
for appropriate permutation matrix $\dot{Q}$.
\end{myenumerate}

\eqref{lemma2ii} By Lemma~\ref{lema-pomozna}, $\langle {\bf y}_1,\ldots,{\bf y}_r\rangle=\langle{\bf x}_1,\ldots,{\bf x}_r\rangle$ so ${\bf y}_1,\ldots,{\bf y}_r$ are linearly independent. By \eqref{lemma2i}, we may assume that in \eqref{eq19}, $l=2$ and $Q=I_r$. If we multiply the equation \begin{equation}\label{eq18}
{\bf x}_1^2+\cdots +{\bf x}_r^2={\bf y}_1^2+\cdots +{\bf y}_r^2
\end{equation}
from the right-hand side by $A^{-1}{\bf x}_1$ and $A^{-1}{\bf x}_2$, respectively, we deduce that
$$\sum_{i=1}^r{\bf y}_i\cdot {\bf y}_i^{\tr} A^{-1}{\bf x}_1={\bf x}_1+{\bf x}_2=\sum_{i=1}^r{\bf y}_i\cdot {\bf y}_i^{\tr} A^{-1}{\bf x}_2.$$
The linear independence of ${\bf y}_1,\ldots,{\bf y}_r$ imply that  ${\bf y}_i^{\tr} A^{-1}{\bf x}_1={\bf y}_i^{\tr} A^{-1}{\bf x}_2=:b_i$ for all $i\in \{1,\ldots,r\}$. Moreover, there exists $i_0\in \{1,\ldots,r\}$ such that $b_{i_0}=1$. If we multiply \eqref{eq18} by $A^{-1}{\bf x}_j$, we deduce that
\begin{equation}\label{eq20}
{\bf y}_i^{\tr} A^{-1}{\bf x}_j=0 \qquad (i\in \{1,\ldots,r\}, j\in\{3,\ldots,r\}).
\end{equation}
Let $c_{ij}:={\bf y}_i^{\tr} A^{-1}{\bf y}_j$. Then $c_{ji}=c_{ij}$. If we multiply \eqref{eq18} by $A^{-1}{\bf y}_i$, \eqref{eq20} implies
\begin{equation}\label{eq21}
b_i({\bf x}_1+{\bf x}_2)=\sum_{j=1}^{r} c_{ji} {\bf y}_j\qquad (i\in\{1,\ldots,r\}).
\end{equation}
Let $S_1=\{i\in\{1,\ldots,r\} : b_i=1\}$ and $S_0=\{i\in\{1,\ldots,r\} : b_i=0\}$. If $i\in S_0$, then~\eqref{eq21} and the linear independence of ${\bf y}_1,\ldots,{\bf y}_r$ show that
\begin{equation}\label{eq23}
c_{ij}=c_{ji}=0\qquad (j\in\{1,\ldots,r\}).
\end{equation}
Moreover, if $i_1,i_2\in S_1$, then the same arguments imply that
\begin{equation}\label{eq22}
c_{j i_1}=c_{j i_2}=c_{i_2 j}=c_{i_1 j}\qquad (j\in\{1,\ldots,r\}).
\end{equation}
Since $i_0\in S_1$, \eqref{eq21} implies the existence of $j\in\{1,\ldots,r\}$ such that $c_{j i_0}=1$. By~\eqref{eq23}, $j\in S_1$.
 Consequently, \eqref{eq22} implies that  $c_{i_1 i_0}=1$ for all $i_1\in S_1$ and therefore $c_{i_1 i_2}=1$ for all $i_1,i_2\in S_1$. Along~\eqref{eq23}, it implies that the matrix $[c_{ij}]_{i,j=1}^r=[{\bf y}_i^{\tr} A^{-1}{\bf y}_j]_{i,j=1}^r$ is of rank one. Moreover, \eqref{eq21} implies that ${\bf x}_1+{\bf x}_2=\sum_{j\in S_1} {\bf y}_j$ and consequently
\begin{align*}
0&={\bf x}_1^{\tr} A^{-1}{\bf x}_1+{\bf x}_2^{\tr} A^{-1}{\bf x}_2\\
&=({\bf x}_1+{\bf x}_2)^{\tr}A^{-1}({\bf x}_1+{\bf x}_2)=\sum_{j\in S_1} {\bf y}_j^{\tr} A^{-1}{\bf y}_j=|S_1|\ (\textrm{mod}~2),
\end{align*}
i.e. $[{\bf y}_i^{\tr} A^{-1}{\bf y}_j]_{i,j=1}^r\in \ronetzero$. By Lemma~\ref{lemma1}, ${\bf y}_{i'}^{\tr} A^{-1} {\bf y}_{j'}=0$ for some $i',j'$.
\end{proof}

\section{Proof of Theorem~\ref{thm-nonalter}}
\label{s5}

\begin{proof}[Proof of Theorem~\ref{thm-nonalter} for $r\leq 2$] If $r=1$, then obviously $d(A,B)=r$. This is also claimed by Theorem~\ref{thm-nonalter} because $1=\det B=\det(A+{\bf x}_1^{2})=1+{\bf x}_1^{\tr}A^{-1}{\bf x}_1$ implies that ${\bf x}_1^{\tr}A^{-1}{\bf x}_1=0$, and therefore the assumption \eqref{nonalter_iii} holds.

Let $r=2$. By Lemma~\ref{l2}, $1=\det B=\det(I_2+[{\bf x}_i^{\tr}A^{-1}{\bf x}_j]_{i,j=1}^2)$, i.e.
\begin{equation}\label{eq31}
{\bf x}_1^{\tr}A^{-1}{\bf x}_1+{\bf x}_2^{\tr}A^{-1}{\bf x}_2+{\bf x}_1^{\tr}A^{-1}{\bf x}_1\cdot {\bf x}_2^{\tr}A^{-1}{\bf x}_2-({\bf x}_1^{\tr}A^{-1}{\bf x}_2)^2=0.
\end{equation}
We separate two cases.
\begin{myenumerate}{Case}
\item Let ${\bf x}_i^{\tr}A^{-1}{\bf x}_i=0$ for some $i\in\{1,2\}$. If $j\in \{1,2\}\backslash\{i\}$, then~\eqref{eq31} implies that ${\bf x}_j^{\tr}A^{-1}{\bf x}_j={\bf x}_i^{\tr}A^{-1}{\bf x}_j$, and therefore the assumption~\eqref{nonalter_iii} holds. Moreover, $\det(A+{\bf x}_i^2)=1+{\bf x}_i^{\tr}A^{-1}{\bf x}_i=1$. Hence, $A+{\bf x}_i^2$ is adjacent to both $A$ and $B$, i.e. $d(A,B)\leq 2$. From~\eqref{eq3} we deduce that $d(A,B)=2$ as claimed.

\item Let ${\bf x}_1^{\tr}A^{-1}{\bf x}_1=1={\bf x}_2^{\tr}A^{-1}{\bf x}_2$. Then, \eqref{eq31} implies that ${\bf x}_1^{\tr}A^{-1}{\bf x}_2=1$. As indicated in Remark~\ref{opomba2}, we need to prove that $d(A,B)=4$. Since ${\bf x}_1^{\tr}A^{-1}{\bf x}_1\neq 0$, it follows that $A^{-1}$ is nonalternate. Hence, $A^{-1}=P^{\tr}P$ for some $P\in GL_n(\FF_2)$. Since vectors $P{\bf x}_1, P{\bf x}_2$ are linearly independent, they are nonzero, and at least one of them is different from ${\bf j}_n$. We may assume that $P{\bf x}_1=:(\alpha_1,\ldots, \alpha_n)^{\tr}$ is such. Then there exist $j,k$ such that $\alpha_j=0$ and $\alpha_k=1$. Let ${\bf y}=P^{-1}({\bf e}_j+{\bf e}_k)$. Then ${\bf y}^{\tr}A^{-1}{\bf y}=0$ and ${\bf y}^{\tr}A^{-1}{\bf x}_1=1$. Consequently, regardless of the value ${\bf y}^{\tr}A^{-1}{\bf x}_2$, Lemma~\ref{l2} implies that
    \begin{align*}
    &\det(A+{\bf y}^2)=1+{\bf y}^{\tr}A^{-1}{\bf y}=1,\\
    &\det(A+{\bf y}^2+{\bf x}_1^2)=\det\left(I_2+\left(\begin{array}{cc}{\bf y}^{\tr}A^{-1}{\bf y}&{\bf y}^{\tr}A^{-1}{\bf x}_1\\
    {\bf y}^{\tr}A^{-1}{\bf x}_1&{\bf x}_1^{\tr}A^{-1}{\bf x}_1\end{array}\right)\right)=1,\\
    &\det(A+{\bf y}^2+{\bf x}_1^2+{\bf x}_2^2)\\
    &\qquad=\det\left(I_3+\left(\begin{array}{ccc}{\bf y}^{\tr}A^{-1}{\bf y}&{\bf y}^{\tr}A^{-1}{\bf x}_1&{\bf y}^{\tr}A^{-1}{\bf x}_2\\
    {\bf y}^{\tr}A^{-1}{\bf x}_1&{\bf x}_1^{\tr}A^{-1}{\bf x}_1&{\bf x}_1^{\tr}A^{-1}{\bf x}_2\\
     {\bf y}^{\tr}A^{-1}{\bf x}_2&{\bf x}_1^{\tr}A^{-1}{\bf x}_2&{\bf x}_2^{\tr}A^{-1}{\bf x}_2\end{array}\right)\right)=1,
    \end{align*}
    which means that $A\sim A+{\bf y}^2 \sim A+{\bf y}^2+{\bf x}_1^2\sim  A+{\bf y}^2+{\bf x}_1^2+{\bf x}_2^2\sim B$. Hence, $d(A,B)\leq 4$. By Lemma~\ref{lemma6}, $d(A,B)=4$.\qedhere
\end{myenumerate}
\end{proof}

\begin{proof}[Proof of Theorem~\ref{thm-nonalter} for $r=3$]
By Lemma~\ref{l2},
\begin{equation}\label{eq37}
1=\det B=\det\left(I_3+[{\bf x}_i^{\tr}A^{-1}{\bf x}_j]_{i,j=1}^3\right).
\end{equation}
Since alternate matrices have even rank, there exists $k\in \{1,2,3\}$ such that ${\bf x}_k^{\tr}A^{-1}{\bf x}_k=0$. We separate two cases.
\begin{myenumerate}{Case}
\item Let $[{\bf x}_i^{\tr}A^{-1}{\bf x}_j]_{i,j=1}^3\in \ronetzero$. Then we may assume that
\begin{equation}\label{eq34}
[{\bf x}_i^{\tr}A^{-1}{\bf x}_j]_{i,j=1}^3=\left(\begin{array}{ccc}1&1&0\\1&1&0\\0&0&0\end{array}\right)
\end{equation}
(otherwise we permute vectors ${\bf x}_1, {\bf x}_2, {\bf x}_3$). Let $Q\in GL_n(\FF_2)$ be any invertible matrix with ${\bf x}_1, {\bf x}_2, {\bf x}_3$ as the first three columns. Then, \eqref{eq34} is the top-left $3\times 3$ block of the invertible matrix $Q^{\tr}A^{-1}Q\in\sgl$. Hence, $n>3$. Moreover, a straightforward computation of the determinant shows that no member of $SGL_4(\FF_2)$ has~\eqref{eq34} in the top-left corner. Hence, $n\geq 5$. By Lemma~\ref{witt1}, there exists ${\bf w}\in \FF_2^n$ such that~\eqref{eq26} or~\eqref{eq27} is true. We may assume~\eqref{eq26} (otherwise we permute vectors ${\bf x}_1, {\bf x}_2$). Now, we can apply Lemma~\ref{l2} as in the proof for $r=2$ to deduce that matrices in the path
\begin{align*}
A&\sim A + ({\bf x}_1+{\bf x}_2+{\bf x}_3)^2\\
&\sim A + ({\bf x}_1+{\bf x}_2+{\bf x}_3)^2 + {\bf w}^2\\
&\sim  A + ({\bf x}_1+{\bf x}_2+{\bf x}_3)^2 + {\bf w}^2 + ({\bf x}_1+{\bf x}_3+{\bf w})^2\\
&\sim A + ({\bf x}_1+{\bf x}_2+{\bf x}_3)^2 + {\bf w}^2 + ({\bf x}_1+{\bf x}_3+{\bf w})^2 + ({\bf x}_1+{\bf x}_2+{\bf w})^2\\
&=A+{\bf x}_1^2+{\bf x}_2^2+{\bf x}_3^2+({\bf x}_2+{\bf x}_3+{\bf w})^2\sim B
\end{align*}
have determinant one. Hence, $d(A,B)\leq 5$. By Lemma~\ref{lemma6}, $d(A,B)=5$.

\item Let $[{\bf x}_i^{\tr}A^{-1}{\bf x}_j]_{i,j=1}^3\notin \ronetzero$. Since ${\bf x}_k^{\tr}A^{-1}{\bf x}_k=0$, the assumption of~\eqref{nonalter_iii} is satisfied. We may assume that $k=3$ (otherwise we suitably permute vectors ${\bf x}_1$, ${\bf x}_2$, ${\bf x}_3$).
    We claim that there exist distinct $i,j\in\{1,2,3\}$ such that
    \begin{align}
    \nonumber 0={\bf x}_i^{\tr}A^{-1}{\bf x}_i&={\bf x}_j^{\tr}A^{-1}{\bf x}_j={\bf x}_i^{\tr}A^{-1}{\bf x}_j\\
    \label{eq36} &\textrm{or}\\
    \nonumber 0={\bf x}_i^{\tr}A^{-1}{\bf x}_i&,\ 1={\bf x}_j^{\tr}A^{-1}{\bf x}_j={\bf x}_i^{\tr}A^{-1}{\bf x}_j.
    \end{align}
    To prove~\eqref{eq36} we separate three subcases.
    \begin{myenumerateB}{Subcase}
    \item Let ${\bf x}_1^{\tr}A^{-1}{\bf x}_1=0={\bf x}_2^{\tr}A^{-1}{\bf x}_2$. If~\eqref{eq36} is not true, then the invertible matrix in the right-hand side of~\eqref{eq37} equals $J_{3\times 3}$, a contradiction.
    \item Let ${\bf x}_1^{\tr}A^{-1}{\bf x}_1=1={\bf x}_2^{\tr}A^{-1}{\bf x}_2$.  If~\eqref{eq36} is not true, then
    ${\bf x}_1^{\tr}A^{-1}{\bf x}_3=0={\bf x}_2^{\tr}A^{-1}{\bf x}_3$. Since $[{\bf x}_i^{\tr}A^{-1}{\bf x}_j]_{i,j=1}^3\notin \ronetzero$, we have  ${\bf x}_1^{\tr}A^{-1}{\bf x}_2=0$. This is a contradiction because the invertible matrix in the right-hand side of~\eqref{eq37} is
    $$\left(\begin{array}{ccc}0&0&0\\
    0&0&0\\
    0&0&1\end{array}\right).$$
    \item Let $\{{\bf x}_1^{\tr}A^{-1}{\bf x}_1, {\bf x}_2^{\tr}A^{-1}{\bf x}_2\}=\{0,1\}$. If~\eqref{eq36} is not true, then we get in contradiction as above, i.e. the matrix in the right-hand side of~\eqref{eq37} equals
        $$\left(\begin{array}{ccc}1&0&1\\
          0&0&0\\
    1&0&1\end{array}\right)\qquad \textrm{or}\qquad \left(\begin{array}{ccc}0&0&0\\
          0&1&1\\
    0&1&1\end{array}\right).$$
    \end{myenumerateB}

    \quad Now, if $i,j$ are as in~\eqref{eq36}, then Lemma~\ref{l2} implies that matrices in the path
$$
A\sim A + {\bf x}_i^2\sim A + {\bf x}_i^2 +{\bf x}_j^2\sim  A + {\bf x}_1^2+{\bf x}_2^2+{\bf x}_3^2=B
$$
have determinant one. Hence, $d(A,B)\leq 3$. By~\eqref{eq3}, $d(A,B)=3$ as claimed.\qedhere
\end{myenumerate}
\end{proof}

The proof of Theorem~\ref{thm-nonalter} for $r\geq 4$ applies the induction process. In it, we say that a pair of matrices $(A,B)$ in $\sgl$ satisfy the condition \eqref{nonalter_i},  \eqref{nonalter_ii}, or \eqref{nonalter_iii} with respect to vectors ${\bf x}_1,\ldots,{\bf x}_r$ if the conditions \eqref{nonalter_i},  \eqref{nonalter_ii}, or \eqref{nonalter_iii} in the statement of Theorem~\ref{thm-nonalter} are satisfied, respectively.

\begin{proof}[Proof of Theorem~\ref{thm-nonalter} for $r\geq 4$]
We already know that the claim is true for $r\in \{1,2,3\}$. Let $r\geq 4$ and assume the claim is true for values $1,2,\ldots,r-1$. We separate three cases.
\begin{myenumerate}{Case}
\item Let \eqref{nonalter_ii} be satisfied and assume \eqref{nonalter_i} is not, i.e. $[{\bf x}_i^{\tr}A^{-1}{\bf x}_j]_{i,j=1}^r\neq J_{r\times r}$. Then, there exists $i_0$ such that ${\bf x}_{i_0}^{\tr}A^{-1}{\bf x}_j=0={\bf x}_j^{\tr}A^{-1}{\bf x}_{i_0}$ for all $j$. Hence, $[{\bf x}_i^{\tr}A^{-1}{\bf x}_j]_{i,j\in \{1,\ldots,r\}\backslash\{i_0\}}\in \ronetzero$. By Proposition~\ref{prop0}, there exists a permutation matrix $Q\in GL_{r-1}(\FF_2)$ such that
    $$[{\bf x}_i^{\tr}A^{-1}{\bf x}_j]_{i,j\in \{1,\ldots,r\}\backslash\{i_0\}}=Q\left(\begin{array}{cc}J_{k\times k}& O\\O&O\end{array}\right)Q^{\tr}$$
    for some even $k\geq 2$. Let $B'=A+\sum_{i\neq i_{0}} {\bf x}_i^2$. By Lemma~\ref{l2},
    \begin{align*}
\det B'&=\det\left(I_{r-1}+[{\bf x}_i^{\tr}A^{-1}{\bf x}_j]_{i,j\in \{1,\ldots,r\}\backslash\{i_0\}}\right)\\
&=\det \left(Q\left(\begin{array}{cc}I_k+J_{k\times k}& O\\O&I_{r-1-k}\end{array}\right)Q^{\tr}\right)\\
&=\det (I_k+{\bf j}_{k}^{2})\\
&=1+{\bf j}_{k}^{\tr}{\bf j}_{k}=1
\end{align*}
because $k$ is even. Hence, $B'\in \sgl$. By the induction hypothesis, $d(A,B')\leq (r-1)+2=r+1$. Consequently,
$d(A,B)\leq d(A,B')+d(B',B)\leq r+2$. By Lemma~\ref{lemma6}, $d(A,B)=r+2$ as claimed.

\item\label{case2nonalter} Let \eqref{nonalter_iii} be satisfied. Then, there exists $i$ such that ${\bf x}_i^{\tr}A^{-1}{\bf x}_i=0$. By Lemma~\ref{lemma7}, there exist ${\bf y}_1,\ldots,{\bf y}_r\in \FF_2^n$ such that $\sum_{i=1}^r {\bf x}_i^2=\sum_{i=1}^r {\bf y}_i^2$ and both matrices $B''=A+\sum_{i=1}^{r-2} {\bf y}_i^2$ and $B'=A+\sum_{i=1}^{r-1} {\bf y}_i^2$ are in  $\sgl$. Moreover, ${\bf y}_1,\ldots,{\bf y}_r$ are linearly independent by Lemma~\ref{lema-pomozna}. By Lemma~\ref{l2},
    \begin{equation}\label{eq38}
    1=\det B'=\det\left(I_{r-1}+[{\bf y}_i^{\tr}A^{-1}{\bf y}_j]_{i,j=1}^{r-1}\right)
    \end{equation}
    and
    \begin{equation} \label{eq39}
    1=\det B''=\det\left(I_{r-2}+[{\bf y}_i^{\tr}A^{-1}{\bf y}_j]_{i,j=1}^{r-2}\right).
    \end{equation}
    Therefore, there exist $i_0,j_0\leq r-1$ such that
    \begin{equation} \label{eq40}
    {\bf y}_{i_0}^{\tr}A^{-1}{\bf y}_{j_0}=0.
    \end{equation}
    In fact, the opposite would force the matrices in the right-hand side of~\eqref{eq38}, \eqref{eq39} to be alternate, a contradiction because such matrices have even rank, while one of the numbers $r-2$ and $r-1$ is odd. We split Case~\ref{case2nonalter} into three subcases, depending on the type of the pair $(A,B')$ with respect to vectors ${\bf y}_1,\ldots, {\bf y}_{r-1}$.
    \begin{myenumerateB}{Subcase}
    \item\label{subcase_nov_1} Suppose that $(A,B')$ satisfy the condition \eqref{nonalter_ii}. By~Lemma~\ref{lemma2}~\eqref{lemma2i} and \eqref{eq40}, we may assume that
    $$[{\bf y}_i^{\tr}A^{-1}{\bf y}_j]_{i,j=1}^{r-1}=Q\left(\begin{array}{cc}J_{2\times 2} & O\\
    O&O\end{array}\right)Q^{\tr}$$
    for some permutation matrix $Q\in GL_{r-1}(\FF_2)$. If $\dot{Q}\in GL_{r}(\FF_2)$ is the permutation matrix with $Q$ in the top-left corner and $1$ in the $(r,r)$-entry, then
    \begin{equation}\label{eq43}
    [{\bf y}_i^{\tr}A^{-1}{\bf y}_j]_{i,j=1}^{r}=\dot{Q}\left(\begin{array}{ccc}J_{2\times 2} & O& \dot{{\bf b}}\\
    O&O&\ddot{{\bf b}}\\
    \dot{{\bf b}}^{\tr}&\ddot{{\bf b}}^{\tr}&b_r\end{array}\right)\dot{Q}^{\tr}
    \end{equation}
    for some $\dot{{\bf b}}=(b_1,b_2)^{\tr}\in \FF_2^2$, $\ddot{{\bf b}}=(b_3,\ldots,b_{r-1})^{\tr}\in\FF_2^{r-3}$, $b_r\in\FF_2$. Since
    \begin{align}
    \nonumber 1=\det B&=\det\left(I_{r}+[{\bf y}_i^{\tr}A^{-1}{\bf y}_j]_{i,j=1}^{r}\right)\\
    \nonumber&=\det \left(\dot{Q}\left(\begin{array}{ccc}J_{2\times 2}+I_2 & O& \dot{{\bf b}}\\
    O&I_{r-3}&\ddot{{\bf b}}\\
    \dot{{\bf b}}^{\tr}&\ddot{{\bf b}}^{\tr}&1+b_r\end{array}\right)\dot{Q}^{\tr}\right)\\
    \label{eq41}&=\det \left(\begin{array}{ccc}J_{2\times 2}+I_2 & O& \dot{{\bf b}}\\
    O&I_{r-3}&\ddot{{\bf b}}\\
    \dot{{\bf b}}^{\tr}&\ddot{{\bf b}}^{\tr}&1+b_r\end{array}\right),
    \end{align}
    the last column of the matrix in~\eqref{eq41} is not a linear combination of the other columns. Hence,
    $1+b_r\neq b_3+\cdots+ b_{r-1}$, i.e.
    \begin{equation}\label{eq42}
    b_r=b_3+\cdots+ b_{r-1}.
    \end{equation}
    Since $[{\bf x}_i^{\tr}A^{-1}{\bf x}_j]_{i,j=1}^{r}\notin\ronetzero$ by
    the assumption of Case~\ref{case2nonalter}, Lemma~\ref{lemma2}~\eqref{lemma2ii} implies that $(b_1,\ldots,b_r)$ is not the zero vector. Moreover, \eqref{eq43} implies that
    $$[{\bf y}_{\sigma(i)}^{\tr}A^{-1}{\bf y}_{\sigma(j)}]_{i,j=1}^{r}=\left(\begin{array}{ccc}J_{2\times 2} & O& \dot{{\bf b}}\\
    O&O&\ddot{{\bf b}}\\
    \dot{{\bf b}}^{\tr}&\ddot{{\bf b}}^{\tr}&b_r\end{array}\right)$$
    for some permutation $\sigma$ of the set $\{1,\ldots,r\}$.

    \qquad If there exists $s\in\{1,2\}$ such that $b_s=1$ and $\{t\}=\{1,2\}\backslash\{s\}$, then we can consider the matrix $\dot{B}'=A+\sum_{i\neq \sigma(t)} {\bf y}_i^2$. Since
    $$[{\bf y}_{\sigma(i)}^{\tr}A^{-1}{\bf y}_{\sigma(j)}]_{i,j\in\{1,\ldots,r\}\backslash\{t\}}=\left(\begin{array}{ccc}1 & O& 1\\
    O&O&\ddot{{\bf b}}\\
    1&\ddot{{\bf b}}^{\tr}&b_r\end{array}\right),$$
    we have $\dot{B}'\in\sgl$ by Lemma~\ref{l2}. If $(A,\dot{B}')$ is of type \eqref{nonalter_ii} for vectors ${\bf y}_{\sigma(1)},\ldots, {\bf y}_{\sigma(t-1)}, {\bf y}_{\sigma(t+1)},\ldots, {\bf y}_{\sigma(r)}$, then $b_r=1$ and $\ddot{{\bf b}}=0$, which contradicts~\eqref{eq42}. Clearly, $(A,\dot{B}')$ is not of type \eqref{nonalter_i}. Hence, $(A,\dot{B}')$ is of type \eqref{nonalter_iii}. By the induction hypothesis,
    $d(A,\dot{B}')=r-1$. Consequently, $d(A,B)\leq d(A,\dot{B}')+d(\dot{B}',B)=r$. Therefore, \eqref{eq3} implies that $d(A,B)=r$.

    \qquad If $b_1=0=b_2$ and there exists $k\in \{3,\ldots, r-1\}$ such that $b_k=0$, then consider the matrix $\ddot{B}'=A+\sum_{i\neq \sigma(k)} {\bf y}_i^2$. Since
    $${\tiny [{\bf y}_{\sigma(i)}^{\tr}A^{-1}{\bf y}_{\sigma(j)}]_{i,j\in\{1,\ldots,r\}\backslash\{k\}}=\left(\begin{array}{ccccccccc}
    1 & 1& 0&\cdots&0&0&\cdots&0 &0\\
    1 & 1& 0&\cdots&0&0&\cdots&0 &0\\
    0 & 0& 0&\cdots&0&0&\cdots&0 &b_3\\
    \vdots&\vdots&\vdots&\ddots&\vdots&\vdots&\cdots&\vdots&\vdots\\
    0 & 0& 0&\cdots&0&0&\cdots&0 &b_{k-1}\\
    0 & 0& 0&\cdots&0&0&\cdots&0 &b_{k+1}\\
    \vdots&\vdots&\cdots&\vdots&\vdots&\vdots&\ddots&\vdots&\vdots\\
    0 & 0& 0&\cdots&0&0&\cdots&0 &b_{r-1}\\
    0 & 0& b_3&\cdots&b_{k-1}&b_{k+1}&\cdots&b_{r-1}&b_r\end{array}\right)}$$
    and $b_r=\sum_{i=3, i\neq k}^{r-1} b_i$, Lemma~\ref{l2} implies that
    $$\det \ddot{B}'=\det\left(I_{r-1}+[{\bf y}_{\sigma(i)}^{\tr}A^{-1}{\bf y}_{\sigma(j)}]_{i,j\in\{1,\ldots,r\}\backslash\{k\}}\right)=1,$$
    i.e. $\ddot{B}'\in \sgl$. Clearly, the pair $(A, \ddot{B}')$ is not of type \eqref{nonalter_i} with respect to vectors ${\bf y}_{\sigma(1)},\ldots, {\bf y}_{\sigma(k-1)}, {\bf y}_{\sigma(k+1)},\ldots, {\bf y}_{\sigma(r)}$. It is not of type \eqref{nonalter_ii} either because $(b_1,\ldots,b_r)$ is not the zero vector. Therefore, it is of type \eqref{nonalter_iii}. As above, we now deduce $d(A,\ddot{B}')=r-1$ and $d(A,B)=r$ by applying the induction step.

    \qquad Finally, to end Subcase~\ref{subcase_nov_1}, let $b_1=0=b_2$ and $1=b_3=\cdots=b_{r-1}$. By~\eqref{eq42},
    $$b_r=\left\{\begin{array}{ll}1 & \textrm{if}\ r\ \textrm{is even},\\
    0 & \textrm{if}\ r\ \textrm{is odd}.\end{array}\right.$$ Define vectors ${\bf z}_1,\ldots, {\bf z}_r$ by
$${\bf z}_i:=\left\{\begin{array}{lll}{\bf y}_{\sigma(i)}&\textrm{if}& i\notin\{1,2,3,r\},\\
{\bf y}_{\sigma(i)}+{\bf y}_{\sigma(1)}+{\bf y}_{\sigma(2)}+{\bf y}_{\sigma(3)}+{\bf y}_{\sigma(r)}&\textrm{if}& i\in\{1,2,3,r\}.\\\end{array}\right.$$
Then $\sum_{i=1}^r {\bf z}_i^2=\sum_{i=1}^r {\bf y}_i^2$, and a straightforward computation shows that
\begin{equation}\label{eq44}
[{\bf z}_{i}^{\tr}A^{-1}{\bf z}_{j}]_{i,j\in\{2,\ldots,r\}}=\left(\begin{array}{cc}(1+b_r)J_{2\times 2}& J_{2\times (r-3)}\\
J_{(r-3)\times 2}&O_{(r-3)\times (r-3)}\end{array}\right)+{\bf e}_2^2.
\end{equation}
Let $\dddot{B}'=A+\sum_{i=2}^{r} {\bf z}_i^2$. By Lemma~\ref{l2}, $\dddot{B}'$ has the same determinant as matrix $I_{r-1} + [{\bf z}_{i}^{\tr}A^{-1}{\bf z}_{j}]_{i,j\in\{2,\ldots,r\}}$, which equals the invertible matrix~\eqref{eq2} from Lemma~\ref{lemica} where $a=b_r$ and $m=r-3$. Therefore, $\dddot{B}'\in\sgl$. Since $r\geq 4$, it is clear from~\eqref{eq44} that the pair $(A,\dddot{B}')$ is of type \eqref{nonalter_iii} with respect to vectors ${\bf z}_2,\ldots, {\bf z}_r$ . Hence, we deduce that $d(A,B)=r$ as above.

    \item Suppose that $(A,B')$ satisfy the condition \eqref{nonalter_iii}. Then by induction hypothesis, $d(A,B')=r-1$ and consequently,
    $d(A,B)\leq d(A,B')+d(B',B)=r$. As above, \eqref{eq3} implies that $d(A,B)=r$.

    \item Suppose that $(A,B')$ satisfy the condition \eqref{nonalter_i}.  Then ${\bf y}_i^{\tr}A^{-1}{\bf y}_i=1$ for all $i\in\{1,\ldots,r-1\}$. Since $B'', B'\in \sgl$ and one of the numbers $r-2, r-1$ is odd, we get a contradiction by Lemma~\ref{lemma3}.
    \end{myenumerateB}

\item Let \eqref{nonalter_i} be satisfied. By Lemma~\ref{lemma3}, $r$ is even. Suppose firstly that ${\bf x}_i^{\tr}A^{-1}{\bf x}_j=1$ for all $i,j$. By Lemma~\ref{witt2}, there exists ${\bf x}_{r+1}\in \FF_2^n$ such that~\eqref{eq45} and~\eqref{eq46} is true. We may assume that $s=1$ in \eqref{eq45}-\eqref{eq46} (otherwise we permute vectors ${\bf x}_1,\ldots,{\bf x}_r$). Let $B_1=A+\sum_{j=1}^{r+1}{\bf x}_j^2$. By Lemmas~\ref{l2} and~\ref{schur},
    \begin{align}
    \nonumber \det B_1&=\det\left(I_{r+1}+[{\bf x}_i^{\tr}A^{-1}{\bf x}_j]_{i,j\in\{1,\ldots,r+1\}}\right)\\
    \nonumber &=\det \left(\begin{array}{cc}J_{r\times r}+I_r&{\bf e}_1\\
    {\bf e}_1^{\tr}&1\end{array}\right)\\
    \label{eq47}&=\det(J_{r\times r}+I_r+{\bf e}_1^2).
    \end{align}
    Clearly, the columns space of matrix $J_{r\times r}+I_r+{\bf e}_1^2$ equals the whole space $\FF_2^r$. Hence, $B_1\in \sgl$. Let $B_2=A+\sum_{j=1,j\neq 2}^{r+1}{\bf x}_j^2$. As in~\eqref{eq47}, we deduce that $\det B_2= \det(J_{(r-1)\times (r-1)}+I_{r-1}+{\bf e}_1^2)$ and $B_2\in \sgl$. Moreover, the pair $(A,B_2)$ is of the type~\eqref{nonalter_iii} with respect to the vectors ${\bf x}_1, {\bf x}_3, \ldots, {\bf x}_{r+1}$, which are linearly independent by Remark~\ref{remark}. By Case~\ref{case2nonalter}, $d(A,B_2)=r$. Consequently,
    $$d(A,B)\leq d(A,B_2)+d(B_2,B_1)+d(B_1,B)=r+2.$$
    By Lemma~\ref{lemma6}, $d(A,B)=r+2$.

    \qquad Suppose now that there exist $i,j$ such that ${\bf x}_i^{\tr}A^{-1}{\bf x}_j=0$. By Lemma~\ref{lemma66}, $d(A,B)\geq r+1$. Define vectors ${\bf y}_1,\ldots,{\bf y}_{r+1}$ by
    \begin{align*}
    {\bf y}_i&={\bf x}_1+{\bf x}_2+{\bf x}_3+{\bf x}_i\qquad (i=1,2,3),\\
    {\bf y}_4&={\bf x}_1+{\bf x}_2+{\bf x}_3,\\
    {\bf y}_i&={\bf x}_{i-1}\qquad (i=5,\ldots,r+1).
    \end{align*}
    Then, $\sum_{i=1}^{r+1} {\bf y}_i^2=\sum_{i=1}^r {\bf x}_i^2$. Let the matrix $D=[d_{ij}]_{i,j=1}^{r+1}$ be defined by $$D=I_{r+1}+ [{\bf y}_i^{\tr}A^{-1}{\bf y}_j]_{i,j=1}^{r+1},$$ and let $D_{ij}$ denote the determinant of the $r\times r$ submatrix of $D$, which is obtained by deleting the $i$-th row and the $j$-th column. By Lemma~\ref{l2}, $$1=\det B=\det D.$$
    The Laplace expansion along the $i$-th row yields $1=\sum_{j=1}^{r+1} d_{ij} D_{ij}$ for all $i$. Since
    $$d_{ii}=\left\{\begin{array}{ll}1& \textrm{if}\ i\in\{1,2,3\},\\
    0& \textrm{if}\ i\in\{4,\ldots,r+1\},\end{array}\right.$$
    $r$ is even, $d_{ij}=d_{ji}$, and $D_{ij}=D_{ji}$, in characteristic two we deduce that $1=\sum_{i=1}^{r+1}\sum_{j=1}^{r+1} d_{ij} D_{ij}=D_{11}+D_{22}+D_{33}$. Consequently, there exists $i\in\{1,2,3\}$ such that $D_{ii}=1$. We may assume that $D_{11}=1$. Let $B'=A+\sum_{i=2}^{r+1} {\bf y}_i^2$. Since $$\det B'=\det(I_{r}+ [{\bf y}_i^{\tr}A^{-1}{\bf y}_j]_{i,j=2}^{r+1})=D_{11}=1,$$
    we have $B'\in\sgl$. The diagonal of $[{\bf y}_i^{\tr}A^{-1}{\bf y}_j]_{i,j=2}^{r+1}$ is $(0,0,1,\ldots,1)$.
    If
    \begin{equation}\label{eq48}
    [{\bf y}_i^{\tr}A^{-1}{\bf y}_j]_{i,j=2}^{r+1}\neq \left(\begin{array}{cc}O_{2\times 2}&O_{2\times (r-2)}\\
    O_{(r-2)\times 2}&J_{(r-2)\times (r-2)}\end{array}\right),
    \end{equation}
    then $(A,B')$ is of type \eqref{nonalter_iii} with respect to the vectors ${\bf y}_2,\ldots,{\bf y}_{r+1}$, which are linearly independent. Consequently, $d(A,B')=r$ by Case~\ref{case2nonalter} and therefore
    $d(A,B)\leq d(A,B')+d(B',B)=r+1$. Hence, $d(A,B)=r+1$ as claimed.

    \qquad Finally, assume that the two matrices in~\eqref{eq48} are the same. Then, for each $i\geq 5$ we have $0={\bf y}_2^{\tr}A^{-1}{\bf y}_i={\bf x}_1^{\tr}A^{-1}{\bf x}_{i-1}+{\bf x}_3^{\tr}A^{-1}{\bf x}_{i-1}$, i.e. ${\bf x}_1^{\tr}A^{-1}{\bf x}_{i-1}={\bf x}_3^{\tr}A^{-1}{\bf x}_{i-1}$. If we replace ${\bf y}_2$ by ${\bf y}_3$ in this computation we further deduce that
        $${\bf x}_1^{\tr}A^{-1}{\bf x}_{i-1}={\bf x}_2^{\tr}A^{-1}{\bf x}_{i-1}={\bf x}_3^{\tr}A^{-1}{\bf x}_{i-1}=:a_{i-1}\qquad (5\leq i\leq r+1).$$
        Consequently, for $i,j\geq 5$ we have
        \begin{align*}
        0=1+1&={\bf y}_4^{\tr}A^{-1}{\bf y}_i+ {\bf y}_4^{\tr}A^{-1}{\bf y}_j\\
        &=({\bf x}_1+{\bf x}_2+{\bf x}_3)^{\tr}A^{-1}{\bf x}_{i-1}+ ({\bf x}_1+{\bf x}_2+{\bf x}_3)^{\tr}A^{-1}{\bf x}_{j-1}\\
        &=a_{i-1}+a_{j-1}.
        \end{align*}
        Hence, $a_4=\cdots=a_r$. Moreover, for $i,j\geq 5$ we have also
        $$1={\bf y}_i^{\tr}A^{-1}{\bf y}_j={\bf x}_{i-1}^{\tr}A^{-1}{\bf x}_{j-1}.$$
        Since ${\bf x}_i^{\tr}A^{-1}{\bf x}_i=1$ for all $i$, the expansion of the equalities ${\bf y}_2^{\tr}A^{-1}{\bf y}_4=0={\bf y}_3^{\tr}A^{-1}{\bf y}_4$ implies that
        \begin{equation}\label{eq49}
        {\bf x}_1^{\tr}A^{-1}{\bf x}_2={\bf x}_1^{\tr}A^{-1}{\bf x}_3={\bf x}_2^{\tr}A^{-1}{\bf x}_3.
        \end{equation}
        Consequently,
        $$0={\bf y}_2^{\tr}A^{-1}{\bf y}_3={\bf x}_1^{\tr}A^{-1}{\bf x}_1+{\bf x}_1^{\tr}A^{-1}{\bf x}_2+{\bf x}_3^{\tr}A^{-1}{\bf x}_1+{\bf x}_3^{\tr}A^{-1}{\bf x}_2=1+{\bf x}_1^{\tr}A^{-1}{\bf x}_2,$$
        which means that all values in~\eqref{eq49} equal $1$. Hence,
        $$[{\bf x}_i^{\tr}A^{-1}{\bf x}_j]_{i,j=1}^{r}=\left(\begin{array}{cc}
    J_{3\times 3}&a_4 J_{3\times (r-3)}\\
    a_4 J_{(r-3)\times 3}&J_{(r-3)\times (r-3)}\end{array}\right).$$
    Since ${\bf x}_i^{\tr}A^{-1}{\bf x}_j=0$ for some $i,j$, it follows that $a_4=0$. Consequently, Lemma~\ref{l2} implies that
    $$1=\det B=\det \left(\begin{array}{cc}
    J_{3\times 3}+I_3&O\\
    O&J_{(r-3)\times (r-3)}+I_{r-3}\end{array}\right),$$
    a contradiction because matrix $J_{3\times 3}+I_3$ is not invertible.\qedhere
\end{myenumerate}
\end{proof}

\section{Proof of Theorem~\ref{thm-alter}}
\label{s6}

\begin{proof}[Proof of Theorem~\ref{thm-alter}~\eqref{alter_i}] Since $A-B$ is an alternate matrix of rank $r$ and \eqref{eq3} holds, it suffices to prove that $d(A,B)\leq r+1$. Define
\begin{align}
\label{eq5}S_0&=\left\{i\in \{1,\ldots,r+1\}: {\bf x}_i^{\tr}A^{-1}{\bf x}_i=0\right\},\\
\label{eq6}S_1&=\left\{i\in \{1,\ldots,r+1\}: {\bf x}_i^{\tr}A^{-1}{\bf x}_i=1\right\}.
\end{align}
Since ${\bf x}_{r+1}^{\tr}A^{-1}{\bf x}_{r+1}=\sum_{i=1}^r {\bf x}_i^{\tr}A^{-1}{\bf x}_i$ and $r$ is even, it follows that $|S_0|\geq 1$.
We separate three cases.
\begin{myenumerate}{Case}
\item Let $|S_0|=1$ and $r=2$.  Since ${\bf x}_1+{\bf x}_2+{\bf x}_3=0$ and each pair of vectors among ${\bf x}_1,{\bf x}_2,{\bf x}_3$ are linearly independent, we may assume that $S_0=\{1\}$ and $S_1=\{2,3\}$. Let $a={\bf x}_1^{\tr}A^{-1}{\bf x}_2$. Then,
    $${\bf x}_1^{\tr}A^{-1}{\bf x}_3=a,\quad {\bf x}_2^{\tr}A^{-1}{\bf x}_3=1+a.$$
    If $X$ is the $n\times 3$ matrix with ${\bf x}_i$ as its $i$-th column, then Lemma~\ref{l2} implies that
    \begin{equation}\label{eq4}
    1=\det B=\det (A+XX^{\tr})=\det(I_3+X^{\tr}A^{-1}X)=\det\left(\begin{array}{ccc}1&a&a\\
    a&0&1+a\\
    a&1+a&0\end{array}\right).
    \end{equation}
    Consequently, $a=0$. Hence, $[{\bf x}_i^{\tr} A^{-1}{\bf x}_j]_{i,j=1}^{r+1}$ is of rank-one, a contradiction.

\item Let $|S_0|>1$ and $r=2$. Since ${\bf x}_1+{\bf x}_2+{\bf x}_3=0$, we deduce that $|S_0|=3$, i.e. ${\bf x}_i^{\tr} A^{-1} {\bf x}_i=0$ for $i=1,2,3$. If $a={\bf x}_1^{\tr}A^{-1}{\bf x}_2$. Then,
    $${\bf x}_1^{\tr}A^{-1}{\bf x}_3=a={\bf x}_2^{\tr}A^{-1}{\bf x}_3.$$
    As in \eqref{eq4} we deduce that
    $$1=\det B=\det\left(\begin{array}{ccc}1&a&a\\
    a&1&a\\
    a&a&1\end{array}\right),$$
    i.e. $a=0$. Define $B'=A+{\bf x}_1^2+{\bf x}_2^2$. Then
    $$\det B'=\det\left(I_2+\left(\begin{array}{cc}{\bf x}_1^{\tr}A^{-1}{\bf x}_1&{\bf x}_1^{\tr}A^{-1}{\bf x}_2\\
    {\bf x}_1^{\tr}A^{-1}{\bf x}_2&{\bf x}_3^{\tr}A^{-1}{\bf x}_3\end{array}\right)\right)=1.$$
    By Theorem~\ref{thm-nonalter}, $d(A,B')=2$. Hence, $d(A,B)\leq d(A,B')+d(B',B)=3=r+1$.

\item Let $r\geq 4$. If $|S_0|=1$, i.e. $S_0=\{i_0\}$ for some $i_0\in\{1,\ldots,r+1\}$, then there exist $j_1,k_1,l_1\in S_1$ and we can define ${\bf y}_1,\ldots,{\bf y}_{r+1}$ by
    \begin{align*}
    {\bf y}_{i_0}&={\bf x}_{j_1}+{\bf x}_{k_1}+{\bf x}_{l_1},\\
    {\bf y}_{j_1}&={\bf x}_{i_0}+{\bf x}_{k_1}+{\bf x}_{l_1},\\
    {\bf y}_{k_1}&={\bf x}_{i_0}+{\bf x}_{j_1}+{\bf x}_{l_1},\\
    {\bf y}_{l_1}&={\bf x}_{i_0}+{\bf x}_{j_1}+{\bf x}_{k_1},
    \end{align*}
    and ${\bf y}_{i}={\bf x}_{i}$ for all $i\in \{1,\ldots,r+1\}\backslash\{i_0,j_1,k_1,l_1\}$. Then $\sum_{i=1}^{r+1} {\bf x}_i^2=\sum_{i=1}^{r+1} {\bf y}_i^2$,  ${\bf y}_{r+1}={\sum}_{i=1}^r {\bf y}_i$, vectors ${\bf y}_1,\ldots,{\bf y}_r$ are linearly independent, and
    $$|\left\{i\in \{1,\ldots,r+1\}: {\bf y}_i^{\tr}A^{-1}{\bf y}_i=0\right\}|=|\{j_1,k_1,l_1\}|=3.$$
    Moreover, if $X$ and $Y$ are the $n\times (r+1)$ matrices with ${\bf x}_i$ and ${\bf y}_i$ as its $i$-th column, respectively, then $Y=XR$ for appropriate $R\in GL_{r+1}(\FF_2)$, and matrices $[{\bf y}_i A^{-1}{\bf y}_j]_{i,j=1}^{r+1}=Y^{\tr}A^{-1}Y$, $[{\bf x}_i A^{-1}{\bf x}_j]_{i,j=1}^{r+1}=X^{\tr}A^{-1}X$ have equal rank. By replacing ${\bf x}_1,\ldots,{\bf x}_{r+1}$ with ${\bf y}_1,\ldots,{\bf y}_{r+1}$, we may assume that $|S_0|\geq 2$.

    \quad To proceed, observe that as in \eqref{eq4} we can deduce that $\det D=1$ where $D=[d_{ij}]_{i,j=1}^{r+1}:=I_{r+1}+X^{\tr} A^{-1} X$. Let $D_{ij}$ be the determinant of the submatrix obtained by removing the $i$-th row and the $j$-th column of $D$. Then, for each $i\in \{1,\ldots,r+1\}$, the Laplace expansion along the $i$-th row yields $1=\sum_{j=1}^{r+1}d_{ij} D_{ij}$. Since $D$ is symmetric, we have $d_{ij}=d_{ji}$ and $D_{ij}=D_{ji}$. Since $r$ is even,
    $$1=\sum_{i=1}^{r+1}\sum_{j=1}^{r+1} d_{ij} D_{ij}=\sum_{i=1}^{r+1} d_{ii} D_{ii}.$$
    Therefore, there exists $i\in \{1,\ldots,r+1\}$ such that $d_{ii}=1=D_{ii}$. We may assume that $i=r+1$ (otherwise we permute vectors ${\bf x}_1,\ldots, {\bf x}_{r+1}$ in appropriate order). Define $B':=A+{\bf x}_1^2+\cdots+ {\bf x}_r^2$. Then,
    $$\det B'=\det\left(I_r+[{\bf x}_i^{\tr} A^{-1}{\bf x}_j]_{i,j=1}^{r}\right)=\det D_{r+1,r+1}=1.$$
    Since $|S_0|\geq 2$, the matrix $[{\bf x}_i^{\tr} A^{-1}{\bf x}_j]_{i,j=1}^{r}$ must have at least one zero on the diagonal. By Remark~\ref{opomba1},
     $\rank([{\bf x}_i^{\tr} A^{-1}{\bf x}_j]_{i,j=1}^{r})=\rank([{\bf x}_i^{\tr} A^{-1}{\bf x}_j]_{i,j=1}^{r+1})\neq 1$. By Theorem~\ref{thm-nonalter}, $d(A,B')=r$. Hence, $d(A,B)\leq d(A,B')+d(B',B)=r+1$.\qedhere
\end{myenumerate}
\end{proof}

\begin{proof}[Proof of Theorem~\ref{thm-alter}~\eqref{alter_ii}]
Let $S_0,S_1$ be as in \eqref{eq5}, \eqref{eq6}. By Remark~\ref{opomba1}, $0=\Tr([{\bf x}_i^{\tr} A^{-1}{\bf x}_j]_{i,j=1}^{r+1})=\sum_{i=1}^{r+1} {\bf x}_i^{\tr} A^{-1}{\bf x}_i$, so $|S_1|=2t$ for some $t\leq \frac{r}{2}$. Since the matrix $[{\bf x}_i^{\tr} A^{-1}{\bf x}_j]_{i,j=1}^{r+1}$ is of rank one, it is not alternate. Hence, its diagonal is not zero everywhere, and therefore $t\geq 1$. We may assume that
\begin{equation}\label{eq7}
S_1=\{1,2,\ldots,2t\}\quad \textrm{and}\quad S_0=\{2t+1,2t+2,\ldots,r+1\};
\end{equation}
(otherwise we permute the vectors ${\bf x}_1,\ldots, {\bf x}_{r+1}$).  Since $B-A$ is an alternate matrix of rank $r$, \eqref{eq3} implies that
$$d(A,B)\geq r+1.$$
Moreover, Lemma~\ref{l2}, Proposition~\ref{prop0}, and \eqref{eq7} imply that
\begin{align}
\nonumber B^{-1}&=A^{-1}-A^{-1}X(I_{r+1}+X^{\tr}A^{-1}X)X^{\tr}A^{-1}\\
\label{eq9}&=A^{-1}-A^{-1}X\left(\begin{array}{cc} I_{2t}+J_{(2t)\times (2t)}&O\\O&I_{r+1-2t}\end{array}\right)X^{\tr}A^{-1},
\end{align}
where $X$ is the $n\times (r+1)$ matrix with ${\bf x}_i$ as its $i$-th column. We divide the rest of the proof into three steps.
\begin{myenumerate}{Step}
\item\label{step1} We claim that $d(A,B)\geq r+2$.\smallskip

Suppose that $d(A,B)=r+1$. Then, $\rank(B'-B)=1$ and $d(A,B')=r$ for some $B'\in\sgl$. In particular, $B'=B+{\bf y}^2$ for some nonzero ${\bf y}\in \FF_2^n$.

\qquad If ${\bf y}\notin \langle{\bf x}_1,\ldots,{\bf x}_r\rangle$, then
$$B'- A=B-A+{\bf y}^2=P\left(\begin{array}{ccc}J_{r\times r}+I_r&O_{r\times 1}&O\\
O_{1\times r}&1&O\\
O&O&O\end{array}\right)P^{\tr},$$
for any $P\in GL_n(\FF_2)$ with ${\bf x}_1,\ldots,{\bf x}_r,{\bf y}$ as the first $r+1$ columns. In particular, $\rank(B'-A)=r+1$. Hence, \eqref{eq3} implies that $d(A,B')\geq r+1$, a contradiction.

\qquad Therefore, ${\bf y}\in \langle{\bf x}_1,\ldots,{\bf x}_r\rangle$, i.e. ${\bf y}=\sum_{j=1}^s {\bf x}_{i_j}$ for some subset $\{i_1,\ldots,i_s\}\subseteq \{1,\ldots, r\}$ of distinct indices. If $s=r$, then $B'=A+{\bf x}_1^2+\cdots +{\bf x}_r^2$. From Remark~\ref{opomba1} and \eqref{eq7} we deduce that $[{\bf x}_i^{\tr}A^{-1}{\bf x}_j]_{i,j=1}^r\in \ronetzero$. Consequently,  $d(A,B')=r+2$ by Theorem~\ref{thm-nonalter}, a contradiction. Therefore $1\leq s <r$. If $S_1':=\{j\in \{1,\ldots,s\}: i_{j}\leq 2t\}$ and $j\leq s$, then~\eqref{eq9} implies
\begin{align*}
{\bf x}_{i_j}^{\tr} &B^{-1}{\bf x}_{i_j}={\bf x}_{i_j}^{\tr} A^{-1}{\bf x}_{i_j}-{\bf x}_{i_j}^{\tr} A^{-1}X\left(\begin{array}{cc} I_{2t}+J_{(2t)\times (2t)}&O\\O&I_{r+1-2t}\end{array}\right)X^{\tr}A^{-1}{\bf x}_{i_j}\\
&=\left\{\begin{array}{lll} 1-(\underbrace{1\cdots 1}_{2t} 0 \cdots 0)\left(\begin{array}{cc} I_{2t}+J_{(2t)\times (2t)}&O\\O&I_{r+1-2t}\end{array}\right)\left(\begin{array}{c}1\\
\vdots\\
1\\
0\\
\vdots\\
0\end{array}\right)
&\textrm{if}& j\in S_1',\\0-(0 \cdots 0)\left(\begin{array}{cc} I_{2t}+J_{(2t)\times (2t)}&O\\O&I_{r+1-2t}\end{array}\right)\left(\begin{array}{c}
0\\
\vdots\\
0\end{array}\right)&\textrm{if}& j\notin S_1',\end{array}\right.\\
&=\left\{\begin{array}{lll}
1&\textrm{if}& j\in S_1',\\
0&\textrm{if}& j\notin S_1'.\end{array}\right.
\end{align*}
By Lemma~\ref{l2},
$$1=\det B'=\det (B+{\bf y}^2)=1+{\bf y}^{\tr}B^{-1}{\bf y}=1+\sum_{j=1}^s {\bf x}_{i_j}^{\tr} B^{-1}{\bf x}_{i_j}=1 + |S_1'|\ (\textrm{mod}~2),$$
which means that $|S_1'|$ is even. We denote $\{i_{s+1},\ldots,i_{r}\}=\{1,\ldots,r\}\backslash \{i_1,\ldots,i_s\}$ and separate two cases to end the proof of Step~\ref{step1}.
\begin{myenumerateB}{Case}
\item\label{case1} Let $s$ be odd. Then we define ${\bf w}_1,\ldots, {\bf w}_r$ as in Lemma~\ref{l1}, which together with Remark~\ref{opomba1} imply that
\begin{align*}
&B'=A+{\bf w}_1^2+\cdots +{\bf w}_r^2,\qquad \rank\big([{\bf w}_i^{\tr} A^{-1}{\bf w}_j]_{i,j=1}^{r}\big)=1,\\
&\Tr [{\bf w}_{i}^{\tr}A^{-1}{\bf w}_{j}]_{i,j\in\{1,\ldots,r\}}=\sum_{j=1}^s {\bf x}_{i_j}^{\tr}A^{-1}{\bf x}_{i_j}= |S_1'|\ (\textrm{mod}~2)=0.
\end{align*}
By Theorem~\ref{thm-nonalter}, $d(A,B')=r+2$, a contradiction. Hence, $d(A,B)\geq r+2$.

\item Let $s$ be even. Then we repeat the proof in Case~\ref{case1} where we replace Lemma~\ref{l1} by Lemma~\ref{l3}.
\end{myenumerateB}
\item\label{step2} We claim the matrix $B^{-1}-A^{-1}$ is not alternate.\smallskip

If $B^{-1}-A^{-1}$ is alternate, then it follows from \eqref{eq9} that $$X\left(\begin{array}{cc} I_{2t}+J_{(2t)\times (2t)}&O\\O&I_{r+1-2t}\end{array}\right)X^{\tr}=\sum_{1\leq i<j\leq 2t} {\bf x}_i\circ {\bf x}_j +{\bf x}_{2t+1}^2+\cdots+{\bf x}_{r+1}^2$$
is alternate. Therefore, ${\bf x}_{2t+1}^2+\cdots +{\bf x}_{r+1}^2$ is alternate as well. If we select a matrix $P\in GL_n(\FF_2)$ such that $P {\bf x}_i={\bf e}_i$ for $i=1,\ldots,r$, then the equality ${\bf x}_{r+1}=\sum_{i=1}^r {\bf x}_{i}$ implies that matrix
$$P({\bf x}_{2t+1}^2+\cdots +{\bf x}_{r+1}^2)P^{\tr}=\left(\begin{array}{ccc}O_{(2t)\times (2t)}&O_{(2t)\times (r-2t)}&O\\
O_{(r-2t)\times (2t)}&I_{r-2t}&O\\
O&O&O\end{array}\right)+\left(\begin{array}{cc}J_{r\times r}& O\\
O& O\end{array}\right)$$
is alternate, which is not true.

\item We claim that $d(A,B)=r+2$.\smallskip

Since $[{\bf x}_i^{\tr} A^{-1}{\bf x}_j]_{i,j=1}^{r+1}=X^{\tr} A^{-1} X$ is of rank one, it follows from~\eqref{eq7} that $$[{\bf x}_i^{\tr} A^{-1}{\bf x}_j]_{i,j=1}^{r+1}=\left(\begin{array}{cc} J_{(2t)\times (2t)}&O\\O&O\end{array}\right).$$ In particular,
$${\bf x}_i^{\tr} A^{-1}X=\left\{\begin{array}{lll}(0,\ldots,0,0,\ldots,0)&\textrm{if}&i>2t,\\
(\underbrace{1,\ldots,1}_{2t},0,\ldots,0)&\textrm{if}&i\leq 2t.\end{array}\right.$$
Consequently, \eqref{eq9} implies that
\begin{equation}\label{eq11}
{\bf x}_i^{\tr}(B^{-1}-A^{-1}){\bf x}_j=0\qquad (i,j\in \{1,\ldots,r+1\}).
\end{equation}
By Step~\ref{step2}, there exists nonzero ${\bf z}\in\FF_2^n$ such that
\begin{equation}\label{eq12}
{\bf z}^{\tr}(B^{-1}-A^{-1}){\bf z}=1.
\end{equation}
From~\eqref{eq11} we deduce that
${\bf z}\notin\langle {\bf x}_1,\ldots, {\bf x}_r\rangle$. Define column vector
$${\bf v}=\left\{\begin{array}{lll}{\bf z}&\textrm{if}&{\bf z}^{\tr}B^{-1}{\bf z}=0,\\
{\bf z}+{\bf x}_1&\textrm{if}&{\bf z}^{\tr}B^{-1}{\bf z}=1.\end{array}\right.$$
Then, \eqref{eq7}, \eqref{eq11}, \eqref{eq12} imply that ${\bf v}^{\tr} B^{-1}{\bf v}=0$ and ${\bf v}^{\tr} A^{-1}{\bf v}=1$ . In particular, $B'=B+{\bf v}^{2}$ is an invertible matrix. From Lemma~\ref{lemma4} we deduce that $B'=A+\sum_{i=1}^{r+1} {\bf v}_i^2$ where
${\bf v}_i={\bf x}_i+{\bf v}$ for all $i\in\{1,\ldots,r+1\}$. Since vectors ${\bf v}_1,\ldots,{\bf v}_{r+1}$ span the vector space $\langle {\bf x}_1,\ldots,{\bf x}_{r}, {\bf z}\rangle$, they are linearly independent. Moreover,
\begin{align*}
\Tr [{\bf v}_{i}^{\tr}A^{-1}{\bf v}_{j}]_{i,j\in\{1,\ldots,r+1\}}&=\sum_{i=1}^{r+1} {\bf v}_{i}^{\tr}A^{-1}{\bf v}_{i}\\
&=\sum_{i=1}^{r+1} {\bf x}_{i}^{\tr}A^{-1}{\bf x}_{i} + (r+1) {\bf v}^{\tr}A^{-1}{\bf v} \\
&=\Tr [{\bf x}_{i}^{\tr}A^{-1}{\bf x}_{j}]_{i,j\in\{1,\ldots,r+1\}}+(r+1) {\bf v}^{\tr}A^{-1}{\bf v}\\
&=0+1=1.
\end{align*}
Since $r+1$ is odd, it follows from Lemma~\ref{lemma3} that ${\bf v}_{i}^{\tr}A^{-1}{\bf v}_{i}=0$ for some~$i$. Consequently, Theorem~\ref{thm-nonalter} implies that $d(A,B')=r+1$ and therefore $d(A,B)\leq d(A,B')+d(B',B)=r+2$. From Step~\ref{step1} we infer that $d(A,B)=r+2$.\qedhere
\end{myenumerate}
\end{proof}

\section{Diameter of $\Gamma_n$}
\label{s7}

To determine the diameter of graph $\Gamma_n$ we need two more lemmas.

\begin{lemma}\label{diameter_lemma1}
Let $n\geq 2$.
\begin{enumerate}
\item\label{i-diameter} A pair $(A,B)$ of matrices in $\sgl$ that satisfy the condition \eqref{nonalter_i} from Theorem~\ref{thm-nonalter} where ${\bf x}_i^{\tr} A^{-1} {\bf x}_j=1$ for all $i,j$ exists if and only if $r\in \{1,\ldots, \lfloor\frac{n+1}{2}\rfloor\}$ is even.
\item\label{ii-diameter} A pair $(A,B)$ of matrices in $\sgl$ that satisfy the condition \eqref{nonalter_i} from Theorem~\ref{thm-nonalter} where ${\bf x}_i^{\tr} A^{-1} {\bf x}_j=0$ for some $i,j$ exists if and only if $r\in \{1,\ldots, n\}\backslash\{2\}$ is even.
\item \label{iii-diameter} A pair $(A,B)$ of matrices in $\sgl$ that satisfy the condition \eqref{nonalter_ii} from Theorem~\ref{thm-nonalter} exists if and only if $r\in \{2,3,\ldots, \lfloor\frac{n+1}{2}\rfloor\}$.
\item \label{iv-diameter} A pair $(A,B)$ of matrices in $\sgl$ that satisfy the condition \eqref{nonalter_iii} from Theorem~\ref{thm-nonalter} exists for each $r\in \{1,\ldots,n\}$.
\end{enumerate}
\end{lemma}
\begin{proof}
\eqref{i-diameter} If ${\bf x}_i^{\tr} A^{-1} {\bf x}_j=1$ for all $i,j$, then $r$ is even by Lemma~\ref{lemma3}. Moreover, $A^{-1}$ is nonalternate. Hence, $A^{-1}=PP^{\tr}$ for some $P\in GL_n(\FF_2)$ and vectors ${\bf y}_i:=P^{\tr}{\bf x}_i$ satisfy ${\bf y}_i^{\tr}{\bf y}_j=1$ for all $i,j$.
In particular, $\rank\left([{\bf y}_i^{\tr}{\bf y}_j]_{i,j=1}^{r}\right)=1$. By Lemma~\ref{orthocode1}, $r\leq \lfloor\frac{n+1}{2}\rfloor$.

Conversely, if $r\leq \big\lfloor\frac{n+1}{2}\big\rfloor$, then vectors ${\bf x}_i=\sum_{k=1}^{2i-1} {\bf e}_k$ for $i=1,\ldots,r$ satisfy ${\bf x}_i^{\tr}{\bf x}_j=1$ for all $i,j$. By Lemma~\ref{l2}, the matrix $B:=I_n+\sum_{i=1}^r {\bf x}_i^2$ satisfies $\det B=\det(I_r+J_r)=\det(I_r+{\bf j}_r^{2})=1+{\bf j}_r^{\tr}{\bf j}_r=1$ whenever $r$ is even. Hence, $(A,B)$ where $A=I_n$ is a required pair.

\eqref{ii-diameter} If $(A,B)$ is any such pair, then, as above, $r$ is even by Lemma~\ref{lemma3}. In the case $r=2$ we would necessarily have ${\bf x}_1^{\tr} A^{-1} {\bf x}_2=0$. Consequently, Lemma~\ref{l2} would imply that $1=\det B=\det(I_2+I_2)=0$, a contradiction.

Suppose now that $r\in \{1,\ldots, n\}\backslash\{2\}$ is even. If $r=4k$ for some $k\geq 1$, then we set $A=I_n$ and $B=I_n+\sum_{i=1}^r {\bf x}_i^2$ where
\begin{equation}\label{eq50}
{\bf x}_i=\left\{\begin{array}{ll}{\bf e}_i&\textrm{if}\ i\in \{1,\ldots,2k\},\\
{\bf e}_{i-2k}+{\bf e}_{i-2k+1}+\cdots +{\bf e}_{i}&\textrm{if}\ i\in \{2k+1,\ldots,4k\}.\end{array}\right.
\end{equation}
Clearly, ${\bf x}_1,\ldots, {\bf x}_{r}$ are linearly independent and $A\in \sgl$. By Lemma~\ref{l2}, $B$ and $I_r+[{\bf x}_i^{\tr}{\bf x}_j]_{i,j=1}^r$ have the same determinant. Observe that $$I_r+[{\bf x}_i^{\tr}{\bf x}_j]_{i,j=1}^r=\left(\begin{array}{cc}O&L\\
L^{\tr}&A_{22}\end{array}\right)$$
where $A_{22}\in S_{2k}(\FF_2)$ and $L=[l_{ij}]_{i,j=1}^{2k}$ is the lower triangular matrix with $l_{ij}=1$ if and only if  $i\geq j$. Consequently, $L$ is invertible and therefore $B\in \sgl$. If $r=4k-2$ for some $k\geq 2$, then we replace vectors~\eqref{eq50} by
\begin{equation*}
{\bf x}_i=\left\{\begin{array}{ll}{\bf e}_i&\textrm{if}\ i\in \{1,\ldots,2k-1\},\\
{\bf e}_1+{\bf e}_2+\cdots +{\bf e}_{2k+1}&\textrm{if}\ i=2k,\\
{\bf e}_{i+1-2k}+{\bf e}_{i+2-2k}+\cdots +{\bf e}_{i+1}&\textrm{if}\ i\in \{2k+1,\ldots,4k-3\},\\
{\bf e}_{1}+{\bf e}_{2k-1}+{\bf e}_{2k}&\textrm{if}\ i=4k-2.
\end{array}\right.
\end{equation*}
Now, the only difference is that $L=[l_{ij}]_{i,j=1}^{2k-1}$ satisfies $l_{ij}=1$ if and only if  $i\geq j$ or $(i,j)=(1,2k-1)$. Again, $L\in GL_{2k-1}(\FF_2)$ and therefore $B\in \sgl$.

\eqref{iii-diameter}. Let the pair $(A,B)$ satisfy $[{\bf x}_i^{\tr}A^{-1}{\bf x}_j]_{i,j=1}^r\in \ronetzero$. By Proposition~\ref{prop0}, $r\geq 2$.
Moreover, ${\bf x}_i^{\tr}A^{-1}{\bf x}_i=1$ for some $i$. Hence, $A^{-1}$ is not alternate and there exists $P\in GL_n(\FF_2)$ such that $A^{-1}=PP^{\tr}$. Linearly independent vectors ${\bf y}_j:=P^{\tr} {\bf x}_j$ satisfy $[{\bf y}_i^{\tr}{\bf y}_j]_{i,j=1}^r\in \ronetzero$. By Lemma~\ref{orthocode1}, $r\leq \lfloor\frac{n+1}{2}\big\rfloor$.

Conversely, let $r\in \{2,3,\ldots, \lfloor\frac{n+1}{2}\rfloor\}$, in particular $n\geq 3$. Define $A=I_n$ and $B=I_n+\sum_{i=1}^r {\bf x}_i^2$ where
$${\bf x}_1={\bf e}_1,\quad {\bf x}_2={\bf e}_1+{\bf e}_2+{\bf e}_3,\quad {\bf x}_i={\bf e}_{2i-2}+{\bf e}_{2i-1}\ \textrm{if}\ i\in \{3,\ldots,r\}.$$
Then, $A\in\sgl$, $$[{\bf x}_i^{\tr}A^{-1}{\bf x}_j]_{i,j=1}^r=\left(\begin{array}{cc}J_{2\times 2}&O\\
O&O\end{array}\right)\in\ronetzero,$$
and $\det B=\det(I_{r}+[{\bf x}_i^{\tr}A^{-1}{\bf x}_j]_{i,j=1}^r)=1$ by Lemma~\ref{l2}, i.e. $B\in \sgl$.

\eqref{iv-diameter} Let $r\in\{1,\ldots,n\}$, $A=I_n$, and $B=I_n+\sum_{i=1}^{r}{\bf x}_i^2$ where
$${\bf x}_i=\left\{\begin{array}{ll}{\bf e}_i& \textrm{if}\ i\ \textrm{is even},\\
{\bf e}_{n-2}+{\bf e}_n& \textrm{if}\ i=r=n\ \textrm{is odd},\\
{\bf e}_i+{\bf e}_{i+1}& \textrm{otherwise}.\end{array}\right.$$ It is straightforward to check that ${\bf x}_1,\ldots, {\bf x}_r$ are linearly independent and $A,B\in\sgl$. If $r=1$, then ${\bf x}_1^{\tr}A^{-1}{\bf x}_1=0$ implies that the pair $(A,B)$ does not fit the assumptions \eqref{nonalter_i}, \eqref{nonalter_ii} from Theorem~\ref{thm-nonalter}. If $r\geq 2$, then the same conclusion is obtained by observing in addition that
${\bf x}_2^{\tr}A^{-1}{\bf x}_2=1={\bf x}_1^{\tr}A^{-1}{\bf x}_2$.
\end{proof}

\begin{lemma}\label{diameter_lemma2}
Let $n\geq 2$ and suppose $0<r\leq n$ is even.
\begin{enumerate}
\item\label{v-diameter} A pair $(A,B)$ of matrices in $\sgl$ that satisfy the condition \eqref{alter_i} from Theorem~\ref{thm-alter} exists if and only if $(r,n)\notin\{(2,2), (2,3)\}$.

\item\label{vi-diameter} A pair $(A,B)$ of matrices in $\sgl$ that satisfy the condition \eqref{alter_ii} from Theorem~\ref{thm-alter} exists if and only if $r\leq \lfloor\frac{n+1}{2}\rfloor$.
\end{enumerate}
\end{lemma}
\begin{proof}
\eqref{v-diameter}
Let $A=I_n$. If $r=2$ and $n\geq 4$, then it is easy to check that, for ${\bf x}_1={\bf e}_1+{\bf e}_2$, ${\bf x}_2={\bf e}_3+{\bf e}_4$, we have  $B=A+{\bf x}_1^2 + {\bf x}_2^2 +({\bf x}_1+{\bf x}_2)^2\in\sgl$ and $[{\bf x}_i^{\tr} A^{-1}{\bf x}_j]_{i,j=1}^{r+1}=O$. If $r=4k$ or $r=4k+2$ for some $k\geq 1$, then define
\begin{equation}\label{eq53}
B=A+\left(\begin{array}{ccccc}
C& & & & \\
&C & & & \\
& &\ddots & & \\
& & &C& \\
& & & & 0\end{array}\right)\quad \textrm{and}\quad B=A+\left(\begin{array}{cccccc}
C& & & & &\\
&C & & & & \\
& &\ddots & & &\\
& & &C& &\\
& & & & D &\\
& & & &  & 0\end{array}\right),
\end{equation}
respectively where
$$C=\left(\begin{array}{cccc}
0&1&0&0\\
1&0&1&0\\
0&1&0&1\\
0&0&1&0\end{array}\right)\quad \textrm{and}\quad D=\left(\begin{array}{cccccc}
0&1&0&0&0&0\\
1&0&1&0&0&0\\
0&1&0&1&0&0\\
0&0&1&0&1&0\\
0&0&0&1&0&1\\
0&0&0&0&1&0
\end{array}\right).$$ In~\eqref{eq53}, $C$ appears $k$ and $k-1$ times respectively. Clearly, $B\in \sgl$. Since $C={\bf e}_1\circ {\bf e}_2+{\bf e}_3\circ ({\bf e}_2+{\bf e}_4)$ and $D={\bf e}_1\circ {\bf e}_2+{\bf e}_3\circ ({\bf e}_2+{\bf e}_4)+{\bf e}_5\circ ({\bf e}_4+{\bf e}_6)$, we deduce that in both cases $B-A={\bf y}_1\circ {\bf y}_2+\cdots+{\bf y}_{r-1}\circ {\bf y}_{r}$ for some linearly independent vectors ${\bf y}_1,\ldots, {\bf y}_r\in\FF_2^n$ where ${\bf y}_1={\bf e}_1$, ${\bf y}_2={\bf e}_2$, ${\bf y}_3={\bf e}_3$, ${\bf y}_4={\bf e}_2+{\bf e}_4$. By Lemma~\ref{alternate-canonical}, $B-A={\bf x}_1^2+\cdots +{\bf x}_r^2+({\bf x}_1+\cdots+{\bf x}_r)^2$ where ${\bf x}_1={\bf e}_1$, ${\bf x}_2={\bf e}_2$, ${\bf x}_3={\bf e}_1+{\bf e}_2+{\bf e}_3$, ${\bf x}_4={\bf e}_1+{\bf e}_4$. Hence, ${\bf x}_1^{\tr}A^{-1}{\bf x}_1=1={\bf x}_1^{\tr}A^{-1}{\bf x}_4$ and ${\bf x}_4^{\tr}A^{-1}{\bf x}_4=0$, which means that $[{\bf x}_i^{\tr} A^{-1}{\bf x}_j]_{i,j=1}^{r+1}$ has rank at least two.

Conversely, suppose that $r=2$, $n\in \{2,3\}$,  $A\in \sgl$, and $B=A+{\bf x}_1^2+{\bf x}_2^2+({\bf x}_1+{\bf x}_2)^2\in\sgl$ where ${\bf x}_1,{\bf x}_2$ are linearly independent, and $[{\bf x}_i^{\tr} A^{-1}{\bf x}_j]_{i,j=1}^{r+1}$ is not of rank one. Let $\alpha={\bf x}_1^{\tr} A^{-1}{\bf x}_1$, $\beta={\bf x}_2^{\tr} A^{-1}{\bf x}_2$, $\gamma={\bf x}_1^{\tr} A^{-1}{\bf x}_2$. In characteristic two, Lemma~\ref{l2} implies that
\begin{align*}
1=\det B
&=\det(I_3+[{\bf x}_i^{\tr} A^{-1}{\bf x}_j]_{i,j=1}^{r+1})\\
&=\det \left(\begin{array}{ccc}
1+\alpha&\gamma&\alpha+\gamma\\
\gamma&1+\beta&\beta+\gamma\\
\alpha+\gamma&\beta+\gamma&1+\alpha+\beta
\end{array}\right)=1+\alpha\beta-\gamma^2.
\end{align*}
Hence, $\alpha\beta=\gamma^2$. By Remark~\ref{opomba1}, $\rank [{\bf x}_i^{\tr} A^{-1}{\bf x}_j]_{i,j=1}^{2}\neq 1$, i.e. $\alpha, \beta, \gamma$ are all zero. If $A^{-1}$ is not alternate, then $A^{-1}=PP^{\tr}$ for some $P\in GL_n(\FF_2)$, and $\langle P^{\tr}{\bf x}_1, P^{\tr}{\bf x}_2\rangle$ is a self-orthogonal code of length $n$ and dimension $2$. Hence, $n\geq 4$, a contradiction. If $A^{-1}$ is alternate, then $n=2$, $A^{-1}={\bf e}_1\circ {\bf e}_2$, and $\{{\bf x}_1,{\bf x}_2,{\bf x}_1+{\bf x}_2\}=\{{\bf e}_1,{\bf e}_2,{\bf e}_1+{\bf e}_2\}$. Consequently $B=O$, a contradiction.

\eqref{vi-diameter} If $0<r\leq \lfloor\frac{n+1}{2}\rfloor$ is even, then define $A=I_n$, $B=A+\sum_{i=1}^{r+1} {\bf x}_i^2$ where
$$
{\bf x}_1={\bf e}_1,\quad
{\bf x}_i={\bf e}_{2i-2}+{\bf e}_{2i-1}\ (i=2,\ldots,r),\quad
{\bf x}_{r+1}={\bf x}_1+\cdots+{\bf x}_r.$$
Then, $[{\bf x}_i^{\tr} A^{-1}{\bf x}_j]_{i,j=1}^{r+1}=({\bf e}_1+{\bf e}_{r+1})^2$ is of rank one and $A, B\in \sgl$.

Conversely, suppose that a pair $(A,B)$ of matrices in $\sgl$ satisfies the condition \eqref{alter_ii} in Theorem~\ref{thm-alter}. By Remark~\ref{opomba1}, $\rank \left([{\bf x}_i^{\tr} A^{-1}{\bf x}_j]_{i,j=1}^{r}\right)=1$. In particular, ${\bf x}_i^{\tr}A^{-1}{\bf x}_i=1$ for some $i$, which means that $A^{-1}$ is not alternate. Hence, there exists $P\in GL_n(\FF_2)$ such that $A^{-1}=PP^{\tr}$ and $\rank\left([{\bf y}_i^{\tr}{\bf y}_j]_{i,j=1}^r\right)=1$ for linearly independent vectors ${\bf y}_j:=P^{\tr} {\bf x}_j$. By Lemma~\ref{orthocode1}, $r\leq \lfloor\frac{n+1}{2}\big\rfloor$.
\end{proof}

Theorems~\ref{thm-nonalter}, \ref{thm-alter} and Lemmas~\ref{diameter_lemma1}, \ref{diameter_lemma2} imply Corollary~\ref{diameter}.
\begin{cor}\label{diameter}
The diameter of graph $\Gamma_n$ equals
$$\diam(\Gamma_n)=\left\{\begin{array}{ll}
2 & \textrm{if}\ n=2,\\
4 & \textrm{if}\ n=3,\\
n+1 & \textrm{if}\ n\geq 4\ \textrm{is even},\\
n & \textrm{if}\ n\geq 5\ \textrm{is odd}.
\end{array}\right.$$
\end{cor}

\section{Binary self-dual codes and $\Gamma_n$}
\label{s8}
In this section, $n\geq 3$ is odd. Recall that, given ${\bf x}\in\FF_2^n$, we defined
$$\overline{{\bf x}}=\left(\begin{array}{c}{\bf x}\\{\bf x}^{\tr}{\bf x}\end{array}\right)\in \FF_2^{n+1}.$$
For ${\bf y}\in\FF_2^{n+1}$, let $\underline{{\bf y}}\in \FF_2^{n}$ be obtained from ${\bf y}$ by deleting the last entry. Then,
$$\overline{{\bf x}_1+{\bf x}_2}=\overline{{\bf x}}_1+\overline{{\bf x}}_2,\qquad\qquad \underline{{\bf y}_1+{\bf y}_2}=\underline{{\bf y}_1}+\underline{{\bf y}_2},$$
$\overline{{\bf x}}^{\tr} \overline{{\bf x}}=0$, and $\dim \langle \overline{{\bf x}}_1,\ldots, \overline{{\bf x}}_r\rangle=\dim \langle {\bf x}_1,\ldots, {\bf x}_r\rangle$ for all column vectors. Moreover, ${\bf y}=\overline{(\underline{{\bf y}})}$ whenever ${\bf y}^{\tr}{\bf y}=0$. Hence, if self-dual codes $C=\langle {\bf y}_1, \ldots , {\bf y}_{\frac{n+1}{2}}\rangle$ and $D=\langle {\bf z}_1, \ldots , {\bf z}_{\frac{n+1}{2}}\rangle$ satisfy $\langle \underline{{\bf y}_1}, \ldots , \underline{{\bf y}_{\frac{n+1}{2}}}\rangle=\langle \underline{{\bf z}_1}, \ldots , \underline{{\bf z}_{\frac{n+1}{2}}}\rangle$, then $C=D$.

Let ${\cal SD}_n\subseteq \sgl$ be the subset of all matrices $A$ such that
\begin{equation*}
d(A,I_n)=\frac{n+5}{2}\quad \textrm{and}\quad \rank(A-I_n)=\frac{n+1}{2}.
\end{equation*}
If $A\in {\cal SD}_n$, then Theorems~\ref{thm-nonalter}, \ref{thm-alter}, and Remark~\ref{opomba1} imply that $A-I_n$ is nonalternate and $A-I_n=\sum_{i=1}^{\frac{n+1}{2}} {\bf x}_i^2$ or $A-I_n$ is alternate, $\frac{n+1}{2}$ is even, and $A-I_n=\sum_{i=1}^{\frac{n+1}{2}} {\bf x}_i^2 +
\left(\sum_{i=1}^{\frac{n+1}{2}} {\bf x}_i\right)^2$,
for some linearly independent ${\bf x}_1,\ldots,{\bf x}_{\frac{n+1}{2}}\in\FF_2^n$ such that $[{\bf x}_i^{\tr} {\bf x}_j]_{i,j=1}^{\frac{n+1}{2}}$ is of rank one. In the nonalternate case, $\Tr\left([{\bf x}_i^{\tr} {\bf x}_j]_{i,j=1}^{\frac{n+1}{2}}\right)=0$.
\begin{prop}\label{selfdual}
Each $A\in {\cal SD}_n$ determines a unique self-dual code $C=\langle\overline{{\bf x}}_1,\ldots,\overline{{\bf x}}_{\frac{n+1}{2}}\rangle$ in $\FF_2^{n+1}$ where ${\bf x}_1,\ldots,{\bf x}_{\frac{n+1}{2}}$ are as in the previous paragraph.
\end{prop}
\begin{proof}
By Lemma~\ref{orthocode1}, the code $C$ is self-orthogonal. Since $$\dim \langle \overline{{\bf x}}_1,\ldots, \overline{{\bf x}}_{\frac{n+1}{2}}\rangle=\dim \langle {\bf x}_1,\ldots, {\bf x}_{\frac{n+1}{2}}\rangle=\frac{n+1}{2},$$ $C$ is self-dual. By Lemma~\ref{lema-pomozna}, $C$ is uniquely determined by $A$.
\end{proof}

Conversely, assume now that $C$ is any self-dual code in $\FF_2^{n+1}$. Let $\mathfrak{B}_C$ be the set of all its bases. Then $\mathfrak{B}_C=\mathfrak{B}_C^1\cup \mathfrak{B}_C^2$ where $\mathfrak{B}_C^1$ consists of bases having an odd number of member-vectors with the last entry 1. Similarly, bases in $\mathfrak{B}_C^2$ have an even number of vectors with the last entry 1. Given a basis ${\cal B}=\{{\bf y}_1,\ldots,{\bf y}_{\frac{n+1}{2}}\}\in \mathfrak{B}_C$, consider the matrices
\begin{align}
\label{eq57} A'_{{\cal B}}:=&I_n+\sum_{i=1}^{\frac{n+1}{2}} \underline{{\bf y}_i}^2,\\
\label{eq58}A''_{{\cal B}}:=&I_n+\sum_{i=1}^{\frac{n+1}{2}} \underline{{\bf y}_i}^2+(\underline{{\bf y}_1}+\cdots+ \underline{{\bf y}_{\frac{n+1}{2}}})^2.
\end{align}
From the proof of Theorem~\ref{particija} we will be able to observe that
\begin{equation}\label{eq59}
A'_{{\cal B}}\ \textrm{is invertible}\ \Longleftrightarrow \ {\cal B}\in \mathfrak{B}_C^2
\end{equation}
whereas $A''_{{\cal B}}$ is always invertible. Nevertheless, $A''_{{\cal B}}$ turns out to be `relevant' for even $\frac{n+1}{2}$ only. To each self-dual code $C$ in $\FF_2^{n+1}$ we associate the set
$${\cal F}_C:=\left\{\begin{array}{ll} \{A'_{{\cal B}} : {\cal B}\in \mathfrak{B}_C^2\}\cup \{A''_{{\cal B}} : {\cal B}\in \mathfrak{B}_C\}& \textrm{if}\ \frac{n+1}{2}\ \textrm{is even},\\
\{A'_{{\cal B}} : {\cal B}\in \mathfrak{B}_C^2\}& \textrm{if}\ \frac{n+1}{2}\ \textrm{is odd}.\end{array}\right.$$
\begin{remark}
For $n\geq 7$, there exist distinct bases ${\cal B}$ and $\widehat{{\cal B}}$ of a self-dual code $C$ such that $A'_{{\cal B}}=A'_{\widehat{{\cal B}}}$ and $A''_{{\cal B}}=A''_{\widehat{{\cal B}}}$. For example, we can obtain such $\widehat{{\cal B}}$ from ${\cal B}$ by replacing vectors ${\bf y}_1,{\bf y}_2,{\bf y}_3,{\bf y}_4$ with ${\bf y}_1+{\bf y}_2 + {\bf y}_3 + {\bf y}_4 + {\bf y}_i$ $(i=1,2,3,4)$. Moreover, for arbitrary $n\geq 3$, $A''_{{\cal B}}=A''_{\widehat{\widehat{{\cal B}}}}$ whenever $\widehat{\widehat{{\cal B}}}$ is obtained from ${\cal B}$ by replacing one of the vectors ${\bf y}_1,\ldots,{\bf y}_{\frac{n+1}{2}}$ in ${\cal B}$ with ${\bf y}_1+\cdots+ {\bf y}_{\frac{n+1}{2}}$.
\end{remark}

\begin{thm}\label{particija}
$\{{\cal F}_C : C\ \textrm{is a self-dual code in}\ \FF_2^{n+1}\}$ is a partition of ${\cal SD}_n$.
\end{thm}
\begin{proof}
As explained in Proposition~\ref{selfdual} and in the paragraph above it, each matrix $A\in {\cal SD}_n$ is of the form $A=A'_{{\cal B}}$ or $A=A''_{{\cal B}}$ where ${\cal B}=\{\overline{{\bf x}}_1,\ldots,\overline{{\bf x}}_{\frac{n+1}{2}}\}$ is a basis of a self-dual code $C=\langle\overline{{\bf x}}_1,\ldots,\overline{{\bf x}}_{\frac{n+1}{2}}\rangle$. Moreover, in the second case $\frac{n+1}{2}$ is even, while in the first case $0=\Tr\left([{\bf x}_i^{\tr} {\bf x}_j]_{i,j=1}^{\frac{n+1}{2}}\right)=\sum_{i=1}^{\frac{n+1}{2}} {\bf x}_i^{\tr} {\bf x}_i$ meaning that ${\bf x}_i^{\tr} {\bf x}_i=1$ for an even number of indices $i\in\{1,\ldots,\frac{n+1}{2}\}$. Therefore, ${\cal B}\in \mathfrak{B}_C^2$ in the case $A=A'_{{\cal B}}$. Hence, $A\in {\cal F}_C$.

The uniqueness part of Proposition~\ref{selfdual} implies that ${\cal F}_C\cap {\cal F}_{\widetilde{C}}=\emptyset$ whenever the self-dual codes $C$ and $\widetilde{C}$ are distinct. It remains to prove that ${\cal F}_C\subseteq {\cal SD}_n$ for each self-dual code $C$ in $\FF_2^{n+1}$.

Let ${\cal B}=\{{\bf y}_1,\ldots,{\bf y}_{\frac{n+1}{2}}\}\in \mathfrak{B}_C$ for some self-dual code $C$. By Lemma~\ref{l2},
\begin{align}
\label{eq55}\det A'_{{\cal B}}&=\det\left(I_{\frac{n+1}{2}}+[\underline{{\bf y}_i}^{\tr} \underline{{\bf y}_j}]_{i,j=1}^{\frac{n+1}{2}}\right),\\
\label{eq56}\det A''_{{\cal B}}&=\det\left(I_{\frac{n+3}{2}}+[\underline{{\bf y}_i}^{\tr} \underline{{\bf y}_j}]_{i,j=1}^{\frac{n+3}{2}}\right)
\end{align}
where $\underline{{\bf y}_{\frac{n+3}{2}}}=\underline{{\bf y}_{1}}+\cdots +\underline{{\bf y}_{\frac{n+1}{2}}}$. Since ${\bf y}_i^{\tr}{\bf y}_j=0$ for $i,j\leq \frac{n+1}{2}$, it follows that
\begin{align}
\label{eq60}\underline{{\bf y}_i}^{\tr} \underline{{\bf y}_j}&=\left\{\begin{array}{ll}1 & \textrm{if the last entries of}\ {\bf y}_i, {\bf y}_j\ \textrm{are both}\ 1,\\0&\textrm{otherwise},\end{array}\right.\\
\nonumber\underline{{\bf y}_{\frac{n+3}{2}}}^{\tr} \underline{{\bf y}_{\frac{n+3}{2}}}&=\left\{\begin{array}{ll}1 & \textrm{if}\ {\cal B}\in \mathfrak{B}_C^1,\\0&\textrm{if}\ {\cal B}\in \mathfrak{B}_C^2,\end{array}\right.\\
\nonumber\underline{{\bf y}_{\frac{n+3}{2}}}^{\tr} \underline{{\bf y}_i}=\underline{{\bf y}_i}^{\tr} \underline{{\bf y}_{\frac{n+3}{2}}}&=\left\{\begin{array}{ll}1 & \textrm{if}\ {\cal B}\in \mathfrak{B}_C^1\ \textrm{and the last entry of}\ {\bf y}_i\ \textrm{is}\ 1,\\0&\textrm{otherwise}.\end{array}\right.
\end{align}
Observe that at least one vector in ${\cal B}$ has the last entry equal to 1 because the opposite would imply that ${\bf e}_{n+1}\in C^{\bot}$, which contradicts the self-duality and the fact that ${\bf e}_{n+1}^{\tr}{\bf e}_{n+1}=1$. Hence, matrices $[\underline{{\bf y}_i}^{\tr} \underline{{\bf y}_j}]_{i,j=1}^{\frac{n+1}{2}}$ and $[\underline{{\bf y}_i}^{\tr} \underline{{\bf y}_j}]_{i,j=1}^{\frac{n+3}{2}}$ are both of rank one. Moreover, $\Tr\left([\underline{{\bf y}_i}^{\tr} \underline{{\bf y}_j}]_{i,j=1}^{\frac{n+3}{2}}\right)=0$, whereas $\Tr\left([\underline{{\bf y}_i}^{\tr} \underline{{\bf y}_j}]_{i,j=1}^{\frac{n+1}{2}}\right)=0$ if and only if ${\cal B}\in \mathfrak{B}_C^2$. By Proposition~\ref{prop0} and~\eqref{eq55}-\eqref{eq56}, $A''_{{\cal B}}\in \sgl$ while $A'_{{\cal B}}\in\sgl$ whenever ${\cal B}\in \mathfrak{B}_C^2$. Since ${\bf y}_1,\ldots,{\bf y}_{\frac{n+1}{2}}$ are linearly independent, it follows that $\rank(A'_{{\cal B}}-I_n)=\frac{n+1}{2}$. On the other hand, $$A''_{{\cal B}}-I_n=P\left(\begin{array}{cc}I_{\frac{n+1}{2}}+J_{\frac{n+1}{2}}&O\\
O&O\end{array}\right)P^{\tr}$$ for any $P\in GL_{n}(\FF_2)$ that has $\underline{{\bf y}_1},\ldots,\underline{{\bf y}_{\frac{n+1}{2}}}$ as the first $\frac{n+1}{2}$ columns. Since $\det\left(I_{\frac{n+1}{2}}+J_{\frac{n+1}{2}}\right)=1+{\bf j}_{\frac{n+1}{2}}^{\tr}{\bf j}_{\frac{n+1}{2}}$ by Lemma~\ref{l2}, it follows that  $\rank(A''_{{\cal B}}-I_n)=\frac{n+1}{2}$ whenever $\frac{n+1}{2}$ is even. It now follows from Theorems~\ref{thm-nonalter} and~\ref{thm-alter} that
$A'_{{\cal B}},A''_{{\cal B}}\in{\cal SD}_n$ whenever $A'_{{\cal B}},A''_{{\cal B}}\in {\cal F}_C$.
\end{proof}

\begin{exa}\label{example}
It is well-known (cf.~\cite[Theorem~9.5.1]{plessknjiga}) that there exist three self-dual codes in $\FF_2^4$, namely $C_1=\langle {\bf e}_1+{\bf e}_2, {\bf j}_4\rangle$, $C_2=\langle {\bf e}_1+{\bf e}_3, {\bf j}_4\rangle$, $C_3=\langle {\bf e}_1+{\bf e}_4, {\bf j}_4\rangle$. Their families are
\begin{align*}
{\cal F}_{C_1}&=\left\{\left(\begin{array}{ccc}0&1&1\\
1&0&1\\
1&1&1\end{array}\right), \left(\begin{array}{ccc}1&0&1\\
0&1&1\\
1&1&1\end{array}\right)\right\},\\
{\cal F}_{C_2}&=\left\{\left(\begin{array}{ccc}0&1&1\\
1&1&1\\
1&1&0\end{array}\right), \left(\begin{array}{ccc}1&1&0\\
1&1&1\\
0&1&1\end{array}\right)\right\},\\
{\cal F}_{C_3}&=\left\{\left(\begin{array}{ccc}1&1&1\\
1&0&1\\
1&1&0\end{array}\right), \left(\begin{array}{ccc}1&1&1\\
1&1&0\\
1&0&1\end{array}\right)\right\}.
\end{align*}
The set ${\cal SD}_3=\bigcup_{i=1}^3{\cal F}_{C_i}$ can be determined also with the help of Figure~\ref{slika}.
\end{exa}

In view of Theorem~\ref{particija} and the identification $C \leftrightarrow {\cal F}_C$, Theorem~\ref{thm-distances} says that self-dual codes are detected by graph parameters $d(A,I_n)$ and $d_{\widehat{\Gamma}_n}(A,I_n)$.
\begin{thm}\label{thm-distances} If $\tfrac{n+1}{2}$ is odd, then
\begin{equation}\label{eq69}
{\cal SD}_n=\left\{A\in\sgl\ :\ d(A,I_n)=\tfrac{n+5}{2},\ d_{\widehat{\Gamma}_n}(A,I_n)=\tfrac{n+1}{2}\right\}.
\end{equation}
If $\frac{n+1}{2}$ is even, then
\begin{equation}\label{eq70}
{\cal SD}_n=\left\{A\in\sgl\ :\ d(A,I_n)=\tfrac{n+5}{2},\ d_{\widehat{\Gamma}_n}(A,I_n)\in\{\tfrac{n+1}{2},\tfrac{n+3}{2}\}\right\}.
\end{equation}
\end{thm}
\begin{proof}
Let ${\cal M}$ be the set in the right-hand side of~\eqref{eq69}-\eqref{eq70}. If $A\in {\cal SD}_n$, then $\rank(A-I_n)=\frac{n+1}{2}$ and $A-I_n$ is nonzero. Hence, either $A-I_n$ is nonalternate and $d_{\widehat{\Gamma}_n}(A,I_n)=\rank(A-I_n)=\frac{n+1}{2}$, or $A-I_n$ is alternate with even rank $\frac{n+1}{2}$, in which case $d_{\widehat{\Gamma}_n}(A,I_n)=\rank(A-I_n)+1=\frac{n+3}{2}$. Therefore, $A\in {\cal M}$.

Conversely, assume that  $A\in {\cal M}$. Let $r=\rank(A-I_n)$. We split two cases.
\begin{myenumerate}{Case}
\item Let $d_{\widehat{\Gamma}_n}(A,I_n)=\frac{n+1}{2}$. Then, $\frac{n+1}{2}\geq r$. Since $\frac{n+5}{2}-r=d(A,I_n)-r\leq 2$ by Theorems~\ref{thm-nonalter} and \ref{thm-alter}, it follows that $r=\frac{n+1}{2}$, i.e. $A\in {\cal SD}_n$.
\item Let $d_{\widehat{\Gamma}_n}(A,I_n)=\frac{n+3}{2}$ and $\frac{n+1}{2}$ is even. If $A-I_n$ is alternate, then $r=d_{\widehat{\Gamma}_n}(A,I_n)-1=\frac{n+1}{2}$ and $A\in {\cal SD}_n$. Assume now that $A-I_n$ is nonalternate. Then, $r=d_{\widehat{\Gamma}_n}(A,I_n)=\frac{n+3}{2}$ is odd and $d(A,I_n)=r+1$. By Theorem~\ref{thm-nonalter}, $A-I_n=\sum_{i=1}^r {\bf x}_i^2$ for some linearly independent ${\bf x}_1,\ldots,{\bf x}_r\in \FF_2^n$ such that ${\bf x}_i^{\tr} {\bf x}_i=1$ for all $i$, and there exist $i,j$ such that ${\bf x}_i^{\tr} {\bf x}_j=0$. We get a contradiction by Lemma~\ref{diameter_lemma1}~\eqref{ii-diameter} because $r$ is odd.\qedhere
\end{myenumerate}
\end{proof}

In the rest of the paper, we describe how `classical' automorphisms of the graph $\Gamma_n$ that fix the identity matrix $I_n$, namely the maps
$$\Phi(A)=A^{-1}\quad \textrm{and} \quad \Phi(A)=PAP^{\tr}\ \textrm{where}\ P^{\tr}=P^{-1}\quad (A\in\sgl),$$
act on ${\cal SD}_n$ (the classification of all automorphisms of $\Gamma_n$ is expected in~\cite{inPreparation}). Observe that we really have $\Phi({\cal SD}_n)={\cal SD}_n$ for both automorphisms because
\begin{align*}
d\big(\Phi(A),I_n\big)&=d\big(\Phi(A),\Phi(I_n)\big)=d(A,I),\\
\rank(A^{-1}-I_n)&=\rank\big(A(A^{-1}-I_n)\big)=\rank(A-I_n),\\
\rank(PAP^{\tr}-I_n)&=\rank\big(P(A-I_n)P^{\tr}\big)=\rank(A-I_n)
\end{align*}
for all $A\in \sgl$. We denote ${\cal F}_{C}^{-1}=\{A^{-1}: A\in {\cal F}_{C}\}$, $P{\cal F}_{C}P^{\tr}=\{PAP^{\tr}: A\in {\cal F}_{C}\}$, and refer $P\in GL_n(\FF_2)$ with $P^{\tr}=P^{-1}$ as an \emph{orthogonal} matrix. Orthogonal matrices form a group, here denoted by ${\cal O}_n(\FF_2)$. Given $P\in {\cal O}_n(\FF_2)$, let $P\oplus 1\in {\cal O}_{n+1}(\FF_2)$ be defined by
$$\left(\begin{array}{cc}P&O_{n\times 1}\\
O_{1\times n}&1\end{array}\right).$$
Observe that $(P\oplus 1){\bf y}=\overline{P\underline{{\bf y}}}$ whenever ${\bf y}\in \FF_2^{n+1}$ satisfies ${\bf y}^{\tr}{\bf y}=0$. Given a self-dual code $C$ in $\FF_2^{n+1}$, we denote $(P\oplus 1)C=\{(P\oplus 1){\bf y}: {\bf y}\in C\}$.
\begin{thm}\label{thm-delovanjeinverza}
${\cal F}_C^{-1}={\cal F}_C$ for each self-dual code $C$ in $\FF_2^{n+1}$.
\end{thm}
\begin{proof}
Let $A\in {\cal F}_C$. Observe from \eqref{eq57}-\eqref{eq58} that there exists a basis ${\cal B}=\{{\bf y}_1,\ldots,{\bf y}_{\frac{n+1}{2}}\}\in \mathfrak{B}_C$ such that either $A=I_n+YY^{\tr}$ or $A=I_n+Y(I_{\frac{n+1}{2}}+J_{\frac{n+1}{2}})Y^{\tr}$ where $Y$ is the $n\times \big(\frac{n+1}{2}\big)$ matrix with $\underline{{\bf y}_i}$ as the $i$-th column. By Lemma~\ref{l2}, $A^{-1}=I_n+Y B Y^{\tr}$ for a suitable matrix $B\in S_{\frac{n+1}{2}}(\FF_2)$. Since $A^{-1}\in {\cal SD}_n$, it follows that $\rank (Y B Y^{\tr})=\frac{n+1}{2}$, i.e. $B\in SGL_{\frac{n+1}{2}}(\FF_2)$. By Lemma~\ref{alternate-canonical}, either $Y B Y^{\tr}=\sum_{i=1}^{\frac{n+1}{2}} {\bf z}_i^2 + ({\bf z}_1+ \cdots + {\bf z}_{\frac{n+1}{2}})^2$ or $Y B Y^{\tr}=\sum_{i=1}^{\frac{n+1}{2}} {\bf z}_i^2$  for some linearly independent vectors ${\bf z}_1,\ldots,{\bf z}_{\frac{n+1}{2}}\in \FF_2^n$, depending on whether $Y B Y^{\tr}$ is alternate or not. The alternate case is possible only if $\frac{n+1}{2}$ is even.

We claim that $\langle {\bf z}_1,\ldots,{\bf z}_{\frac{n+1}{2}}\rangle = \langle \underline{{\bf y}_1},\ldots, \underline{{\bf y}_{\frac{n+1}{2}}}\rangle$. If this is not true, then ${\bf z}_j\notin \langle \underline{{\bf y}_1},\ldots, \underline{{\bf y}_{\frac{n+1}{2}}}\rangle$ for some $j$. Let $R\in GL_n(\FF_2)$ be any invertible matrix with $Y$ as its $n\times \big(\frac{n+1}{2}\big)$ left block and ${\bf z}_j$ as its $\big(\frac{n+3}{2}\big)$-th column. Then,
$$R\left(\begin{array}{ccc}B&O&O\\
O&-1&O\\
O&O&O\end{array}\right)R^{\tr}=Y B Y^{\tr}-{\bf z}_j^2\in \left\{\sum_{i\neq j} {\bf z}_i^2 , \sum_{i\neq j} {\bf z}_i^2 + ({\bf z}_1+ \cdots + {\bf z}_{\frac{n+1}{2}})^2\right\}$$
is a contradiction because the left-hand side is a matrix of rank $\frac{n+3}{2}$, while the right-hand side is a matrix with rank at most $\frac{n+1}{2}$. Therefore, the set $\dot{{\cal B}}=\{\overline{{\bf z}}_1,\ldots,\overline{{\bf z}}_{\frac{n+1}{2}}\}$ is a basis of $C$ and either $A^{-1}=A''_{\dot{{\cal B}}}$ with $\frac{n+1}{2}$ even or  $A^{-1}=A'_{\dot{{\cal B}}}$. In the last case, \eqref{eq59} yields $\dot{{\cal B}}\in \mathfrak{B}_C^2$. Hence, $A^{-1}\in {\cal F}_C$.
\end{proof}
\begin{remark}
If $\frac{n+1}{2}$ is even and all members of a basis ${\cal B}=\{{\bf y}_1,\ldots,{\bf y}_{\frac{n+1}{2}}\}\in \mathfrak{B}_C$ have the last entry 1 (i.e. $\underline{{\bf y}_i}^{\tr}\underline{{\bf y}_j}=1$ for all $i,j$), then it follows from Lemma~\ref{l2} and Proposition~\ref{prop0} that $(A'_{{\cal B}})^{-1}=A''_{{\cal B}}$ and $(A''_{{\cal B}})^{-1}=A'_{{\cal B}}$.
\end{remark}

To understand how automorphisms of the form $A\mapsto PAP^{\tr}$, with $P$ orthogonal, act on ${\cal SD}_n$, we need two more lemmas. The proof of Lemma~\ref{lemma8-7} is straightforward and left to the reader.
\begin{lemma}\label{lemma8-7}
Let $C$ be a self-dual code in $\FF_2^{n+1}$ and $P\in {\cal O}_n(\FF_2)$. Then, $\widetilde{C}=(P\oplus 1)C$ is a self-dual code. Moreover, if ${\cal B}=\{{\bf y}_1,\ldots,{\bf y}_{\frac{n+1}{2}}\}\in \mathfrak{B}_C$, then $\widetilde{{\cal B}}=\{(P\oplus 1){\bf y}_1,\ldots,(P\oplus 1){\bf y}_{\frac{n+1}{2}}\}\in  \mathfrak{B}_{\widetilde{C}}$ and:
\begin{enumerate}
\item\label{codei} ${\cal B}\in \mathfrak{B}_C^2$ if and only if  $\widetilde{{\cal B}}\in \mathfrak{B}_{\widetilde{C}}^2$;
\item\label{codeii} if $A'_{{\cal B}}\in {\cal F}_{C}$, then $PA'_{{\cal B}}P^{\tr}=A'_{\widetilde{{\cal B}}}\in {\cal F}_{\widetilde{C}}$;
\item\label{codeiii} if $A''_{{\cal B}}\in {\cal F}_{C}$, then $PA''_{{\cal B}}P^{\tr}=A''_{\widetilde{{\cal B}}}\in {\cal F}_{\widetilde{C}}$;
\item\label{codeiv} $P{\cal F}_{C}P^{\tr}={\cal F}_{\widetilde{C}}$.
\end{enumerate}
\end{lemma}

\begin{lemma}\label{lema-baza-self-dual}
Each self-dual code $C$ in $\FF_2^{n+1}$ has a basis $\{{\bf y}_1,\ldots,{\bf y}_{\frac{n+1}{2}}\}$ such that ${\bf y}_1+\cdots +{\bf y}_{\frac{n+1}{2}}={\bf j}_{n+1}$.
\end{lemma}
\begin{proof}
Let $\{{\bf z}_1,\ldots,{\bf z}_{\frac{n+1}{2}}\}$ be an arbitrary basis of $C$. Since ${\bf j}_{n+1}\in C$, there exists a nonempty subset $S\subseteq \{1,\ldots,\frac{n+1}{2}\}$ such that ${\bf j}_{n+1}=\sum_{i\in S} {\bf z}_i$.

If $S^c=\emptyset$, we select ${\bf y}_i={\bf z}_i$. Otherwise, we fix $s\in S$ and $t\in S^{c}$ to define
$${\bf w}_i=\left\{\begin{array}{ll}{\bf z}_i& \textrm{if}\ i\neq s,\\
{\bf z}_s+{\bf z}_t& \textrm{if}\ i=s.\end{array}\right.$$
Then, ${\bf j}_{n+1}=\sum_{i\in S\cup\{t\}} {\bf w}_i$ and $\{{\bf w}_1,\ldots,{\bf w}_{\frac{n+1}{2}}\}\in \mathfrak{B}_{C}$. Now we repeat the last paragraph (possibly several times) with $S, {\bf z}_i$ replaced by $S\cup \{t\}, {\bf w}_i$.
\end{proof}

\begin{thm}\label{thm-code}
If $C,\widetilde{C}$ are self-dual codes in $\FF_2^{n+1}$, then there exists $P\in {\cal O}_{n}(\FF_2)$ such that $\widetilde{C}=(P\oplus 1)C$ and  statements~\eqref{codei}-\eqref{codeiv} in Lemma~\ref{lemma8-7} are true.
\end{thm}
\begin{proof}
It suffices to find $P\in {\cal O}_{n}(\FF_2)$ with $\widetilde{C}=(P\oplus 1)C$. The rest follows from Lemma~\ref{lemma8-7}.
If $n=3$, then it is easily deduced from Example~\ref{example} that ${\cal F}_{\widetilde{C}}=P{\cal F}_{C}P^{\tr}$ for an appropriate permutation matrix $P$. A straightforward argument shows that the same permutation matrix $P$ satisfies $\widetilde{C}=(P\oplus 1)C$.

Hence, we may assume that $n\geq 5$. We claim that there exist $\{{\bf z}_1,\ldots,{\bf z}_{\frac{n+1}{2}}\}\in \mathfrak{B}_C$ and $\{\widetilde{{\bf z}}_1,\ldots,\widetilde{{\bf z}}_{\frac{n+1}{2}}\}\in \mathfrak{B}_{\widetilde{C}}$ such that
\begin{equation}\label{eq64}
\underline{{\bf z}_1}+\cdots+\underline{{\bf z}_{\frac{n+1}{2}}}={\bf j}_{n}=\underline{\widetilde{{\bf z}}_1}+\cdots+\underline{\widetilde{{\bf z}}_{\frac{n+1}{2}}}
\end{equation}
and
\begin{equation}\label{eq65}
\underline{{\bf z}_{i}}^{\tr}\underline{{\bf z}_{j}}=\underline{\widetilde{{\bf z}}_{i}}^{\tr}\underline{\widetilde{{\bf z}}_{j}}\qquad \left(1\leq i,j\leq \frac{n+1}{2}\right).
\end{equation}
Lemma~\ref{lema-baza-self-dual} provides ${\cal B}=\{{\bf y}_1,\ldots,{\bf y}_{\frac{n+1}{2}}\}\in \mathfrak{B}_C, \widetilde{{\cal B}}=\{\widetilde{{\bf y}}_1,\ldots,\widetilde{{\bf y}}_{\frac{n+1}{2}}\}\in \mathfrak{B}_{\widetilde{C}}$ with
\begin{equation}\label{eq61}
{\bf y}_1+\cdots+{\bf y}_{\frac{n+1}{2}}={\bf j}_{n+1}=\widetilde{{\bf y}}_1+\cdots+\widetilde{{\bf y}}_{\frac{n+1}{2}}
\end{equation}
and therefore
\begin{equation}\label{eq67}
\underline{{\bf y}_1}+\cdots+\underline{{\bf y}_{\frac{n+1}{2}}}={\bf j}_{n}=\underline{\widetilde{{\bf y}}_1}+\cdots+\underline{\widetilde{{\bf y}}_{\frac{n+1}{2}}}.
\end{equation}
We separate two cases.
\begin{myenumerate}{Case}
\item\label{case1-sd} Let $n=5$. By~\eqref{eq61}, either all vectors ${\bf y}_1, {\bf y}_2, {\bf y}_3$ have the last entry 1, or there exist $i$ and distinct $j,k\in \{1,\ldots, \frac{n+1}{2}\}\backslash\{i\}$ such that ${\bf y}_i$ has the last entry 1, whereas ${\bf y}_j,{\bf y}_k$ have the last entry 0. In the first case we select $\{{\bf z}_1,{\bf z}_2,{\bf z}_3\}=\{{\bf y}_1,{\bf y}_2,{\bf y}_3\}$ whereas in the second case we define ${\bf z}_i={\bf y}_i$, ${\bf z}_j={\bf y}_j+{\bf y}_i$, ${\bf z}_k={\bf y}_k+{\bf y}_i$. We define $\widetilde{{\bf z}}_1, \widetilde{{\bf z}}_2, \widetilde{{\bf z}}_3$ analogously and achieve \eqref{eq64}, \eqref{eq65}.

\item Let $n\geq 7$. We may assume that each of the bases ${\cal B}, \widetilde{{\cal B}}$ contains at least one vector with the last entry 0, and at least two vectors with the last entry 1. In fact, if for example ${\cal B}$ does not meet this criteria, then as in Case~\ref{case1-sd} either all vectors in ${\cal B}$ have the last entry 1, or ${\bf y}_i$, for some $i$, is the only such vector. In both cases, we select distinct $j,k\in \{1,\ldots, \frac{n+1}{2}\}\backslash\{i\}$ and replace ${\bf y}_j, {\bf y}_k$ by ${\bf y}_j+{\bf y}_i, {\bf y}_k+{\bf y}_i$. Since $n\geq 7$, we  get the desired property while keeping \eqref{eq61}.

    \quad Further, by permuting the indices we may assume  the last entries of ${\bf y}_{\frac{n+1}{2}}$ and $\widetilde{{\bf y}}_{\frac{n+1}{2}}$ are both 1. Hence,
\begin{align}
\label{eq62}\underline{{\bf y}_{\frac{n+1}{2}}}^{\tr}\underline{{\bf y}_{\frac{n+1}{2}}}&=1=\underline{\widetilde{{\bf y}}_{\frac{n+1}{2}}}^{\tr}\underline{\widetilde{{\bf y}}_{\frac{n+1}{2}}},\\
\label{eq63}\underline{{\bf y}_{i_1}}^{\tr}\underline{{\bf y}_{i_1}}&=1=\underline{\widetilde{{\bf y}}_{j_1}}^{\tr}\underline{\widetilde{{\bf y}}_{j_1}},\\
\label{eq66}\underline{{\bf y}_{i_2}}^{\tr}\underline{{\bf y}_{i_2}}&=0=\underline{\widetilde{{\bf y}}_{j_2}}^{\tr}\underline{\widetilde{{\bf y}}_{j_2}}
\end{align}
for some $i_1,i_2,j_1,j_2\in\{1,\ldots,\frac{n-1}{2}\}$ and~\eqref{eq67} is true. Recall from the proof of Theorem~\ref{particija} that matrices $[\underline{{\bf y}_i}^{\tr} \underline{{\bf y}_j}]_{i,j=1}^{\frac{n+1}{2}}$ and $[\underline{\widetilde{{\bf y}}_i}^{\tr} \underline{\widetilde{{\bf y}}_j}]_{i,j=1}^{\frac{n+1}{2}}$ are both of rank one, but their traces are 1, as \eqref{eq61} implies that ${\cal B}\in \mathfrak{B}_{C}^1$ and  $\widetilde{{\cal B}}\in \mathfrak{B}_{\widetilde{C}}^1$. By~\eqref{eq62}-\eqref{eq66}, we deduce that matrices $[\underline{{\bf y}_i}^{\tr} \underline{{\bf y}_j}]_{i,j=1}^{\frac{n-1}{2}}$, $[\underline{\widetilde{{\bf y}}_i}^{\tr} \underline{\widetilde{{\bf y}}_j}]_{i,j=1}^{\frac{n-1}{2}}$ are both in $\ronetzero$ and they have at least one zero entry. Consequently, by Lemmas~\ref{lemma2}~\eqref{lemma2i}, \ref{lema-pomozna}, and \eqref{eq67}, there exist linearly independent ${\bf w}_1,\ldots, {\bf w}_{\frac{n-1}{2}}\in\FF_2^n$, linearly independent  $\widetilde{{\bf w}}_1,\ldots, \widetilde{{\bf w}}_{\frac{n-1}{2}}\in\FF_2^n$, and permutation matrices $Q,\widetilde{Q}\in GL_{\frac{n-1}{2}}(\FF_2)$ such that $\langle {\bf w}_1,\ldots, {\bf w}_{\frac{n-1}{2}}\rangle=\langle\underline{{\bf y}_1},\ldots,\underline{{\bf y}_{\frac{n-1}{2}}}\rangle$, $\langle \widetilde{{\bf w}}_1,\ldots, \widetilde{{\bf w}}_{\frac{n-1}{2}}\rangle=\langle\underline{\widetilde{{\bf y}}_1},\ldots,\underline{\widetilde{{\bf y}}_{\frac{n-1}{2}}}\rangle$,
\begin{equation}\label{eq68}
{\bf w}_1+\cdots+{\bf w}_{\frac{n-1}{2}}+\underline{{\bf y}_{\frac{n+1}{2}}}={\bf j}_{n}=\widetilde{{\bf w}}_1+\cdots+\widetilde{{\bf w}}_{\frac{n-1}{2}}+\underline{\widetilde{{\bf y}}_{\frac{n+1}{2}}}
\end{equation}
and
$$[{\bf w}_i^{\tr} {\bf w}_j]_{i,j=1}^{\frac{n-1}{2}}=Q\left(\begin{array}{cc}J_{2\times 2} & O\\
O&O\end{array}\right)Q^{\tr},\quad [\widetilde{{\bf w}}_i^{\tr} \widetilde{{\bf w}}_j]_{i,j=1}^{\frac{n-1}{2}}=\widetilde{Q}\left(\begin{array}{cc}J_{2\times 2} & O\\
O&O\end{array}\right)\widetilde{Q}^{\tr}.$$
Hence, by permuting the indices in $\widetilde{{\bf w}}_1,\ldots, \widetilde{{\bf w}}_{\frac{n-1}{2}}$ we may assume that ${\bf w}_i^{\tr} {\bf w}_j=\widetilde{{\bf w}}_i^{\tr} \widetilde{{\bf w}}_j$ for all $i,j$. If we define ${\bf z}_i=\overline{{\bf w}}_i$, $\widetilde{{\bf z}}_i=\overline{\widetilde{{\bf w}}_i}$ for $i\leq \frac{n-1}{2}$ and ${\bf z}_{\frac{n+1}{2}}={\bf y}_{\frac{n+1}{2}}$, $\widetilde{{\bf z}}_{\frac{n+1}{2}}=\widetilde{{\bf y}}_{\frac{n+1}{2}}$, then \eqref{eq65} is automatically guaranteed for $i,j\leq \frac{n-1}{2}$, while~\eqref{eq62} covers the case $i=\frac{n+1}{2}=j$. Moreover, $\{{\bf z}_1,\ldots,{\bf z}_{\frac{n+1}{2}}\}\in \mathfrak{B}_C$, $\{\widetilde{{\bf z}}_1,\ldots,\widetilde{{\bf z}}_{\frac{n+1}{2}}\}\in \mathfrak{B}_{\widetilde{C}}$, and \eqref{eq68} implies \eqref{eq64}. Lastly, the equality~\eqref{eq64} for $i\leq \frac{n-1}{2}$ implies that
\begin{align*}
\underline{{\bf z}_i}^{\tr}\underline{{\bf z}_{\frac{n+1}{2}}}&=\underline{{\bf z}_i}^{\tr}\left(\underline{{\bf z}_1}+\cdots+\underline{{\bf z}_{\frac{n-1}{2}}}+{\bf j}_{n}\right)\\
&=\underline{{\bf z}_i}^{\tr}\left(\underline{{\bf z}_1}+\cdots+\underline{{\bf z}_{\frac{n-1}{2}}}+\underline{{\bf z}_i}\right)\\
&=\underline{\widetilde{{\bf z}}_i}^{\tr}\left(\underline{\widetilde{{\bf z}}_1}+\cdots+\underline{\widetilde{{\bf z}}_{\frac{n-1}{2}}}+\underline{\widetilde{{\bf z}}_i}\right)\\
&=\underline{\widetilde{{\bf z}}_i}^{\tr}\left(\underline{\widetilde{{\bf z}}_1}+\cdots+\underline{\widetilde{{\bf z}}_{\frac{n-1}{2}}}+{\bf j}_{n}\right)\\
&=\underline{\widetilde{{\bf z}}_i}^{\tr}\underline{\widetilde{{\bf z}}_{\frac{n+1}{2}}},
\end{align*}
which completes the proof of the claim~\eqref{eq64}-\eqref{eq65}.
\end{myenumerate}

Consider the linear map $\sigma: \langle\underline{{\bf z}_1},\ldots,\underline{{\bf z}_{\frac{n+1}{2}}}\rangle\to\FF_2^n$ defined by $\sigma\left(\underline{{\bf z}_i}\right)=\underline{\widetilde{{\bf z}}_i}$ for all $i$. By~\eqref{eq64}-\eqref{eq65}, $\sigma({\bf j}_n)={\bf j}_n$ and \eqref{eq24} is satisfied for all ${\bf u}_1,{\bf u}_2\in \langle\underline{{\bf z}_1},\ldots,\underline{{\bf z}_{\frac{n+1}{2}}}\rangle$. By Lemma~\ref{witt-osnovna}, $\sigma$ can be linearly and injectively extended on $\FF_2^n$ such that \eqref{eq24} holds for all ${\bf u}_1,{\bf u}_2\in\FF_2^n$. Consequently, $\sigma({\bf x})=P{\bf x}$ for all ${\bf x}\in\FF_2$ where $P\in GL_n(\FF_2)$ has $\sigma({\bf e}_i)$ as its $i$-th column. Since the $(i,j)$-th entry of $P^{\tr}P$ equals $\sigma({\bf e}_i)^{\tr}\sigma({\bf e}_j)={\bf e}_i^{\tr}{\bf e}_j$ for all $i,j$, it follows that $P$ is orthogonal. Since,
$\langle P \underline{{\bf z}_1},\ldots, P \underline{{\bf z}_{\frac{n+1}{2}}} \rangle=\langle \underline{\widetilde{{\bf z}}_1},\ldots,\underline{\widetilde{{\bf z}}_{\frac{n+1}{2}}}\rangle$, we deduce that $\widetilde{C}=\langle \widetilde{{\bf z}}_1,\ldots,\widetilde{{\bf z}}_{\frac{n+1}{2}}\rangle=\langle \overline{P\underline{{\bf z}_1}},\ldots,\overline{P\underline{{\bf z}_{\frac{n+1}{2}}}}\rangle=\langle(P\oplus 1){\bf z}_1,\ldots,(P\oplus 1){\bf z}_{\frac{n+1}{2}}\rangle=(P\oplus 1)C$.
\end{proof}

By Theorem~\ref{thm-code}, the subgroup $\{P\oplus 1 : P\in {\cal O}_n(\FF_2)\}$ in ${\cal O}_{n+1}(\FF_2)$, which is isomorphic to ${\cal O}_n(\FF_2)$, acts transitively on the set of all self-dual codes in $\FF_2^{n+1}$. This improves the result of Janusz~\cite[Theorem~10]{janusz}, which states that ${\cal O}_{n+1}(\FF_2)$ acts transitively. Some generators of ${\cal O}_n(\FF_2)$ can be found in~\cite{janusz,sole}.

As we saw in this paper, graph $\Gamma_n$ contains a lot of information about binary (self-dual) codes. We expect that a deeper graph theoretical analysis of $\Gamma_n$ could provide new insights in coding theory in the future.\bigskip

\noindent{\bf Acknowledgements.} The work of Marko Orel is supported in part by the Slovenian Research and Innovation Agency (research program P1-0285 and research projects N1-0140, N1-0208, N1-0210, J1-4084, N1-0296 and J1-50000). The work of Dra\v{z}enka Vi\v{s}nji\'{c} is supported in part by the Slovenian Research and Innovation Agency (research program P1-0285 and Young Researchers Grant).

\end{document}